%% file: funcisosing.tex
\documentclass[12pt]{amsart}

\usepackage{amsmath}
\usepackage{amscd}
\usepackage{amssymb}
\usepackage[all]{xy}
\usepackage[dvips]{graphicx}

\input{macro.tex}

\author{Sergiy Maksymenko}
\title{Functions with isolated singularities on surfaces}
\address{Topology dept., Institute of Mathematics of NAS of Ukraine, Te\-re\-shchenkivska st. 3, Kyiv, 01601 Ukraine}
\email{maks@imath.kiev.ua}
\urladdr{http://www.imath.kiev.ua/~maks}


\begin{document}
\begin{abstract}
Let $\Mman$ be a smooth connected compact surface, $\Psp$ be either the real line $\RRR$ or the circle $\aCircle$, and $\func:\Mman\to\Psp$ be a smooth mapping.
In a previous series of papers for the case when $\func$ is a Morse map the author calculated the homotopy types of stabilizers and orbits of $\func$ with respect to the right action of the diffeomorphisms group of $\Mman$.
The present paper extends those calculations to a large class of maps $\Mman\to\Psp$ with degenerate singularities satisfying certain set of axioms.
\end{abstract}

\maketitle

\section{Introduction}
Let $\Mman$ be a smooth compact connected surface and $\Psp$ be either the real line $\RRR$ or the circle $\aCircle$.
Then the group $\DiffM$ of diffeomorphisms of $\Mman$ naturally acts from the right on the space $\smone$ by the formula:
$$ 
\dif\cdot \func = \func \circ \dif, \qquad \dif\in\DiffM, \ \func\in\smone.
$$ 
This action is one of the main objects in singularities theory.
For the case of surfaces it was extensively studied in recent years, see e.g.~\cite{BolsinovMatveevFomenko:UMN:1990, BolsinovFomenko:1997,  Oshemkov:PSIM:1995,  Prishlyak:UMZ:2000, Prishlyak:TA:2002,  Prishlyak:MFAT:2002, Sharko:UMZ:2003, Sharko:Zb:2003, Maks:AGAG:2006, Maks:TrMath:2008, Kadubovskiy:UMZ:2006, Yurchuk:Zb:2006}.

For $\func\in\smone$ let $\singf$ be the set of critical points of $\func$ and
$$\Orbf=\{\func \circ \dif \ | \ \dif \in \DiffM \},$$
$$\Stabf=\{\dif \ | \ \func = \func \circ \dif, \ \dif\in\DiffM\}$$
be respectively the orbit and the stabilizer of $\func$.
We will endow $\DiffM$ and $\smone$ with the corresponding topologies $\Wr{\infty}$.
Then these topologies induce certain topologies on $\Orbf$ and $\Stabf$.
Let $\DiffIdM$ and $\StabIdf$ be the identity path-component of $\DiffM$ and $\Stabf$, and $\Orbff$ the path-component of $\func$ in $\Orbf$ with respect to topologies $\Wr{\infty}$.

In~\cite{Maks:AGAG:2006, Maks:TrMath:2008} the author calculated the homotopy types of $\StabIdf$ and $\Orbff$ for all Morse maps $\func:\Mman\to\Psp$.
These calculations are essentially based on the description of homotopy types of groups of orbits preserving diffeomorphisms for certain classes of vector fields obtained in~\cite{Maks:TA:2003, Maks:loc-inv-shifts}.
In a series of papers \cite{Maks:hamv2, Maks:CEJM:2009, Maks:MFAT:2009, Maks:ImSh} the classes of vector fields were extended and using these results it was then announced in \cite{Maks:DNANU:2009} that calculations of~\cite{Maks:AGAG:2006, Maks:TrMath:2008} can be done for a large class of smooth maps $\Mman\to\Psp$ with isolated ``homogeneous'' singularities.

The aim of this paper is to show that the technique used in~\cite{Maks:AGAG:2006, Maks:TrMath:2008} can be formalized and thus extended to classes of isolated singularities even larger than homogeneous ones, see Theorems\;\ref{th:hom-type-StabIdf} and\;\ref{th:hom-type-Orbits}.

We will introduce three types of isolated critical points $\ST$, $\PT$, and $\NT$ for a \myemph{germ} of smooth maps $\func:\Mman\to\Psp$.
These points will be discussed in \S\ref{sect:special-crit-points} and now we only note that $\ST$-points are \myemph{saddles} while $\PT$- and $\NT$-points are local extremes%
\footnote{The symbols $\PT$ and $\NT$ stand for \myemph{periodicity} and \myemph{non-periodicity} of shift map.}.
All these points can be degenerate however they satisfy certain ``non-degeneracy'' conditions formulated in the terms of shift map of the corresponding local Hamiltonian vector field of $\func$.
In particular, class of $\ST$-points ($\PT$-points) have properties similar to non-degenerate saddles (local extremes) of Morse functions and include such points, while $\NT$-points behave like degenerate local extremes of homogeneous polynomials, see Lemma\;\ref{exmp:hompoly-wmf}.
These $\NT$-points bring new effects in comparison with Morse functions.

Now we put following three axioms on $\func$:

\begin{axiom}{\AxBd}
$\func$ is constant at each connected component of $\partial\Mman$ and $\singf\subset\Int{\Mman}$.
\end{axiom}

\begin{axiom}{\AxSPN}
Every critical point of $\func$ is either an $\ST$- or a $\PT$- or an $\NT$-point.
\end{axiom}

\begin{axiom}{\AxFibr}
The natural map $p:\DiffM \to \Orbf$ defined by $p(\dif) = \func\circ \dif^{-1}$ 
is a Serre fibration with fiber $\Stabf$ in the corresponding topologies $\Wr{\infty}$.
\end{axiom}

The following theorem describes the homotopy types of $\StabIdf$ and $\Orbff$ for a generic situation.
Detailed formulations are given in Theorems~\ref{th:hom-type-StabIdf} and~\ref{th:hom-type-Orbits} below.

\begin{theorem}\label{th:summary}
Suppose $\func$ satisfies axioms \AxBd-\AxFibr\ and has at least one $\ST$-point.
Let also $n$ be the total number of critical points of $\func$.
Then $\StabIdf$ is contractible, $\Orbff$ is weakly homotopy equivalent to a CW-complex of dimension $\leq 2n-1$.
Moreover, $\pi_i\Orbff=\pi_i\Mman$ for $i\geq 3$, $\pi_2\Orbff=0$, 
and for $\pi_1\Orbff$ we have the following exact sequence:
$$
1 \to \pi_1\DiffM \oplus \ZZZ^{\rankpiO} \to \pi_1\Orbff \to  \grp \to 1,
$$
where $\grp$ is a certain finite group and $\rankpiO\geq0$.
\end{theorem}

\subsection{Structure of the paper}
The exposition of the paper follows the line of\;\cite{Maks:AGAG:2006}.
The principal new feature of the paper is that we consider $\NT$-points.
This requires additional arguments almost everywhere, therefore in many places we repeat the arguments of\;\cite{Maks:AGAG:2006} with necessary modifications.

In\;\S\ref{sect:smooth_shifts} we recall some results concerning the shift map along the orbits of vector fields.
\S\ref{sect:func_on_surf} describes two constructions related to a smooth function $\func$ on a surface: foliation $\partf$ by connected components of level-sets of $\func$ and the Kronrod-Reeb graph $\Reebf$ of $\func$.
In \S\ref{sect:special-crit-points} we introduce three types of critical points $\ST$, $\PT$, and $\NT$.
Further in \S\ref{sect:main_results} we formulate main results of the paper Theorems\;\ref{th:hom-type-StabIdf} and\;\ref{th:hom-type-Orbits} and also discuss sufficiant conditions for axiom \AxFibr.
In\;\S\S\ref{sect:framings} we put on the Kronrod-Reeb graph of $\func$ additional data which describe combinatorial behavior of diffeomorphisms $\dif\in\Stabf$ near $\NT$-points.
\S\ref{sect:proof:th:hom-type-StabIdf} contains the proof of Theorems\;\ref{th:hom-type-StabIdf}.
The proof follows the line of\;\cite[Th.\;1.3]{Maks:AGAG:2006}.
The rest of the paper is devoted to the proof of Theorem\;\ref{th:hom-type-Orbits}.

\begin{remark}\rm
I must warn the reader that the paper\;\cite{Maks:AGAG:2006} contains the following ``dangerous'' places which are also corrected in the present and in previous papers by the author.

1) The calculation of homotopy types of stabilizers given in\;\cite[Th.\;1.3]{Maks:AGAG:2006} is essentially based on the principal result of another paper of mine\;\cite{Maks:TA:2003} which unfortunately contains some mistakes.
The corrections to\;\cite{Maks:TA:2003} are given in\;\cite{Maks:ImSh, Maks:loc-inv-shifts}, where it is also shown that for the case described in \cite{Maks:AGAG:2006} the results of\;\cite{Maks:TA:2003} holds true, see Theorem\;\ref{th:openness_of_shift_maps}.
Thus\;\cite[Th.\;1.3]{Maks:AGAG:2006} remains valid,

2) \cite[Eq.\;(8,6)]{Maks:AGAG:2006} is not true in general, see Remark\;\ref{rem:corrections} and Example\;\ref{exmp:func_on_prjplane} for details.
This changes the meaning of the group $G$ in\;\cite[Th.\;1.5]{Maks:AGAG:2006}: it remains finite however now it is a group of automorphisms of a more complicated object than the Kronrod-Reeb graph of $\func$, which takes to account orientations of level-sets of $\func$, see\;\S\ref{sect:framed-KR-graph}.
A correct formulation of \cite[Eq.\;(8,6)]{Maks:AGAG:2006} is given in Lemma\;\ref{lm:Dpn_and_ketrstong_cap_DiffIdM}.

3) In the proof of\;\cite[Th.\;1.5]{Maks:AGAG:2006} it was claimed without explanations that \myemph{a certain \myemph{central} extension of $\pi_1\DiffIdM$ with a free abelian group $\mathcal{J}_0$ is just a direct sum.}
In general, central extensions of abelian groups even with free abelian groups are not trivial.
We will show in\;Theorem\;\ref{th:splitting_pi1Orbf} that in our case that extension is trivial.
\end{remark}

\subsection{Notations}
If $\Mman$ is a non-orientable surface, then $p:\tMman\to\Mman$ will always denote the oriented double covering of $\Mman$ and $\xi:\tMman\to\tMman$ be a $\Cinf$ involution generating the group $\ZZZ_2$ of deck transformations.
For a function $\func:\Mman\to\Psp$ we put $\tfunc=p\circ\func:\tMman\to\Psp$.

\section{Shift maps along orbits of flows}\label{sect:smooth_shifts}

\subsection{$r$-homotopies}\label{sect:r-homotopy}
Let $M,N$ be smooth manifolds and $0\leq r \leq \infty$.
Say that a map $\Omega:M\times I \to N$ is an \myemph{$r$-homotopy} if the corresponding map $\omega:I \to \Cr{r}{M}{N}$ defined by $\omega(t)(x) = \Omega(x,t)$ is continuous from the standard topology of $I$ to the weak topology $\Wr{r}$ of $\Cr{r}{M}{N}$.
In other words, all partial derivatives of $\Omega$ in $x\in M$ up to order $r$ continuously depends on $t\in I$.
If in addition $\Omega_t:M\to N$ is an embedding for every $t\in I$, then $\Omega$ will be called an \myemph{$r$-isotopy}, see\;\cite{Maks:MFAT:2009}.

Thus a usual homotopy is a $0$-homotopy.
Moreover, every $\Cont{r}$-map $M\times I \to N$ is an $r$-homotopy, but not wise verse.

\subsection{Local flow}
Let $\AFld$ be a smooth vector field on a smooth manifold $\Mman$ tangent to $\partial\Mman$.
Then for every $x\in\Mman$ its \emph{integral trajectory} with respect to $\AFld$ is a unique mapping
$\orb_x: \RRR\supset(a_x,b_x) \to \Mman$ such that $\orb_x(0)=x$ and $\dot{\orb}_x = \AFld(\orb_x)$, where $(a_x,b_x) \subset\RRR$ is the maximal interval on which a map with the previous two properties can be defined.
Then the following set $\ADom = \mathop\cup\limits_{x\in\Mman} x \times (a_x, b_x)$, 
is an open neighbourhood of $\Mman\times0$ in $\Mman\times\RRR$, and the \myemph{local flow\/} of $\AFld$ is defined by
$$\AFlow: \Mman\times\RRR \; \supset \;\ADom\longrightarrow\Mman,
\qquad
\AFlow(x,t) = \orb_x(t).
$$
If $\Mman$ is compact, then $\AFlow$ is defined on all of $\Mman\times\RRR$, e.g.~\cite{PalisdeMelo:1982}.
The set of zeros of $\AFld$ will be denoted by $\FixA$.

\subsection{Shift map}
Let $\Vman \subset \Mman$ be a submanifold (possibly with boundary or even with corners) such that $\dim\Vman=\dim\Mman$.
Denote by $\funcAV$ the subset of $C^{\infty}(\Vman,\RRR)$ consisting of functions $\afunc$ whose graph $\Gamma_{\afunc}=\{(x,\afunc(x)) \ : \ x\in\Vman\}$ is included in $\ADom$.
Then we can define the following map:
\begin{equation}\label{equ:glob_shift_map}
\ShA(\afunc)(x) = \AFlow(x,\afunc(x)), \quad \afunc\in\funcAV, \ x\in\Vman.
\end{equation}
We will call $\ShA$ the \myemph{shift map along orbits of $\AFld$ on $\Vman$} and denote its image in $C^{\infty}(\Vman,\Mman)$ by $\imShAV$.
\begin{definition}
\it Let $\dif:\Vman\to\Mman$ be a smooth map, $\Vman'\subset\Vman$ a submanifold and $\afunc:\Vman'\to\RRR$ a smooth function.
We will say that $\afunc$ is a \myemph{shift function for $\dif$ on $\Vman'$} if $\dif(x)=\AFlow(x,\afunc(x))$ for all $x\in\Vman'$.
\end{definition}

\subsection{Shift map at a singular point}
Let $z\in\Vman\cap\FixA$.
Denote by $\gDAz$ the space of at $z$ germs of orbit preserving diffeomorphisms $\dif:(\Mman,z)\to(\Mman,z)$.
Thus if $\dif$ is defined on some neighbourhood $\Wman$ of $z$, then $\dif(\Wman \cap\orb) \subset\orb$ for every orbit $\orb$ of $\AFld$.

Let also $\gCiMRz$ be the space of germs of $\Cinf$ functions $\afunc:\Mman\to\RRR$ at $z$.
Since $z$ is a singular point for $\AFld$ we have a well-defined \myemph{shift map}
$$\gShAz:\gCiMRz\to\gDAz, \qquad \gShAz(\afunc)(x)=\AFlow(x,\afunc(x)).$$
Denote by $\gimShAz\subset\gDAz$ the image of $\gShAz$.
Then $\gimShAz$ is a \myemph{subgroup} of $\gDAz$, see\;\cite[Eqs.\;(8),(9)]{Maks:TA:2003} or\;\cite[Lm.\;3.1]{Maks:CEJM:2009}.

There is a natural homomorphism $\tnz(\dif):\gDAz\to\Aut(T_{z}\Mman)$ associating to each $\dif\in\gDAz$ the corresponding linear automorphism $T_{z}\dif$ of the tangent space at $z$.

Choose local coordinates $(x_1,\ldots,x_n)$ at $z$.
Then we can regard $\tnz$ as a map $\tnz:\gDAz\to\GLR{n}$ associating to each $\dif\in\gDAz$ its Jacobi matrix at $z$.

Let also $\AFld=(\AFld_1,\ldots,\AFld_{n})$ be the coordinate functions of $\AFld$.
Then the following matrix 
\begin{equation}\label{equ:j1Fz}
\nabla\AFld(z) = \left(\begin{smallmatrix}
\ddd{\AFld_1}{x_1}(z) & \cdots & \ddd{\AFld_1}{x_n}(z) \\ 
\cdots & \cdots & \cdots \\
\ddd{\AFld_n}{x_1}(z) & \cdots & \ddd{\AFld_n}{x_n}(z) 
\end{smallmatrix} \right)
\end{equation}
will be called the \myemph{linear part} of $\AFld$ at $z$.

It is easy to show, \cite[Lm.\;5.3]{Maks:CEJM:2009}, that if $\afunc\in\Ci{\Vman}{\RRR}$ then $\tnz(\ShAV(\afunc))=\tnz(\AFlow_{\afunc(z)}) = e^{\nabla\AFld(z)\cdot t}$,
whence
\begin{equation}\label{equ:im_j1}
\tnz(\gimShAz) = \tnz(\{\AFlow_{t}\}_{t\in\RRR}) = \{ e^{\nabla\AFld(z) t} \}_{t\in\RRR}.
\end{equation}

\subsection{Kernel of shift map}
The set $\ker(\ShAV) = \ShAV^{-1}(i_{\Vman})$ will be called the \myemph{kernel} of $\ShAV$.
It consists of all $\Cinf$ functions $\afunc:\Vman\to\RRR$ such that $\AFlow(x,\afunc(x))=x$ for all $x\in\Vman$.
\begin{lemma}\label{lm:ker_of_shift_map}{\rm\cite{Maks:TA:2003}}
\it Let $\afunc,\bfunc\in\funcAV$.
Then $\ShAV(\afunc) = \ShAV(\bfunc)$ iff $\afunc-\bfunc\in\ker(\ShAV)$.
In other words $$\AFlow(x,\afunc(x))\equiv\AFlow(x,\bfunc(x))\ \ \Leftrightarrow \ \ \AFlow(x,\afunc(x)-\bfunc(x))\equiv x.$$

Suppose $\Vman$ is connected and the set $\FixA$ of singular points of $\AFld$ is nowhere dense in $\Vman$.
Then one of the following conditions holds true:

{\bf Nonperiodic case:}
$\ker(\ShAV)=\{0\}$ and $\ShAV:\funcAV\to\imShAV$ is a bijection.
This holds for instance if $\AFld$ has at least one non-closed orbit, or for some singular point $z$ of $\AFld$ the linear part of $\AFld$ at $z$ vanishes, i.e. $\nabla\AFld(z)=0$; or

{\bf Periodic case:}
$\ker(\ShAV)=\{n\theta\}_{n\in\ZZZ}$ for some $\Cinf$ strictly positive function $\theta:\Vman\to(0,+\infty)$.
In this case 
\begin{itemize}
 \item 
every $x\in\Vman\setminus\FixA$ is periodic so $\funcAV=\Ci{\Vman}{\RRR}$, 
 \item
there exists an open and everywhere dense subset $Q\subset\Vman\setminus\FixA$ such that $\theta(x)=\Per(x)$ for all $x\in Q$;
 \item
$\ShAV^{-1}\circ\ShAV(\afunc) = \{ \afunc + n\theta\}$ for every $\afunc\in\Ci{\Vman}{\RRR}$.
\end{itemize}
\end{lemma}

Let $\EAV$ be the subset of $\Ci{\Vman}{\Mman}$ consisting of maps $\dif:\Vman\to\Mman$ such that 
\begin{enumerate}
\item[(i)]
$\dif(\omega) \subset \omega$ for every orbit $\omega$ of $\AFld$;
\item[(ii)]
$\dif$ is a local diffeomorphism at each singular points of $\AFld$.
\end{enumerate}
Let also $\DAV$ be the subset of $\EAV$ consisting of immersions $\Vman\to\Mman$.

For $0\leq r \leq \infty$ denote by $\EidAV{r}$ (resp. $\DidAV{r}$) the path component of the identity inclusion $i_{\Vman}:\Vman\subset\Mman$ in $\EidAV{r}$ (resp. $\DidAV{r}$) with respect to the topology $\Wr{r}$.
It consists of maps $\dif\in\EidAV{r}$ (resp. $\dif\in\DidAV{r}$) which are $r$-homotopic to $i_{\Vman}$ in $\EAV$ (resp. $\DAV$).

If $\Vman=\Mman$, we will omit $\Vman$ from notations.
Moreover, we will also often omit superscript $\infty$ and denote $\EidAV{\infty}$ and $\DidAV{\infty}$ simply by $\EidAV{}$ and $\DidAV{}$.

\begin{lemma}\label{lm:shift-func-on-reg}{\rm\cite{Maks:TA:2003,Maks:ImSh}}
\it Let $H_t:\Vman\times I\to \Mman$ be an $r$-homotopy such that $H_0=i_{\Vman}$ and $H_t\in\EAV$.
Then there exists a unique $r$-homotopy $\Lambda:(\Vman\setminus\FixA)\times I\to\RRR$ such that $\Lambda_{0}=0$, $\Lambda_t:\Vman\setminus\FixA\to\RRR$ is $\Cont{\infty}$, and $H_t(x)=\AFlow(x,\Lambda_{t}(x))$ for all $x\in\Vman\setminus\FixA$ and $t\in I$.

In particular, for every $\dif\in\EidAV{0}$ there exists a smooth shift function on $\Vman\setminus\FixA$.
\end{lemma}

For $\afunc\in \funcAV$ and $z\in\Vman$ we will denote by $\AFld(\afunc)$ the Lie derivative of $\afunc$ along $\AFld$ at $z$.
Then, \cite[Theorem~19]{Maks:TA:2003}, $\ShAV(\afunc)$ is a local diffeomorphism at $z$ iff $\AFld(\afunc)(z)\not=-1$.
Put 
\begin{equation}\label{equ:gamma_plus}
\Gamma^{+}_{\Vman}=\{ \afunc \in \funcAV \ : \ \AFld(\afunc) > -1 \}.
\end{equation}
Evidently, $\Gamma^{+}_{\Vman}$ is $\Sr{1}$-open and convex subset of $\funcAV$.
It also follows from \cite[Theorem~25]{Maks:TA:2003} that 
\begin{equation}\label{equ:im_shift}
\Gamma^{+}_{\Vman} \;=\; \ShA^{-1}(\DidAV{\infty}).
\end{equation}
\begin{lemma}\label{lm:imShAV_EidAVr_implies_for_diff}{\cite{Maks:MFAT:2009}}
The following inclusions hold true:
$$
\imShAV \subset \EidAV{\infty} \subset \cdots 
\subset \EidAV{r}\subset \cdots \subset \EidAV{0},
$$
$$
\ShAV(\Gamma^{+}_{\Vman}) \subset  \DidAV{\infty} \subset \cdots 
\subset \DidAV{r}\subset \cdots \subset \DidAV{0}.
$$
If $\imShAV=\EidAV{r}$ for some $r\geq0$, then $\ShAV(\Gamma^{+}_{\Vman})=\DidAV{r}$.
\end{lemma}

\subsection{Openness of shift map}
We recall here the principal results obtained in\;\cite{Maks:loc-inv-shifts}, see Theorem\;\ref{th:openness_of_shift_maps} below.
A subset $\Vman\subset\Mman$ will be called a \myemph{$\Dm$-submanifold}, if $\Vman$ is a \myemph{connected} submanifold with boundary (possibly with corners) of $\Mman$ and $\dim\Vman=\dim\Mman$.
We will also say that $\Vman$ is a \myemph{$\Dm$-neighbourhood} for each $z\in\Int{\Vman}$.

\begin{lemma}\label{lm:ShAV_open_reformulations}
Endow $\funcAV$ and $\imShAV$ with topologies $\Wr{\infty}$.
If $\FixA\cap\Vman=\varnothing$ then the shift map $\ShAV$ is locally injective, whence the following conditions are equivalent:
\begin{enumerate}
 \item 
$\ShAV:\funcAV\to\imShAV$ is an open map.
 \item
$\ShAV:\funcAV\to\imShAV$ is a local homeomorphism.
\end{enumerate}

Suppose these conditions hold true.
Consider the restriction 
$$\ShAV|_{\Gamma^{+}_{\Vman}}: \Gamma^{+}_{\Vman} \to \ShAV(\Gamma^{+}_{\Vman}).$$

If $\ShAV$ is \myemph{non-periodic}, then $\ShAV$ and $\ShAV|_{\Gamma^{+}_{\Vman}}$ are homeomorphisms onto their images.
In particular, $\imShAV$ and $\ShAV(\Gamma^{+}_{\Vman})$ are contractible.

Suppose $\ShAV$ is \myemph{periodic}.
Then the maps $\ShAV$ and $\ShAV|_{\Gamma^{+}_{\Vman}}$ are $\ZZZ$-covering maps onto their images, and $\imShAV$ and $\ShAV(\Gamma^{+}_{\Vman})$ are homotopy equivalent to the circle $S^1$.
\end{lemma}

We will be interesting in establishing (1) and (2)  of Lemma\;\ref{lm:ShAV_open_reformulations} for the shift map $\ShA$, i.e. for the case $\Vman=\Mman$.
It is convenient to formulate this as the following condition.

\medskip

\begin{itemize}
\item[\condShAop]
\it The shift map $\ShA:\Ci{\Mman}{\RRR}\to\imShA$ is a local homeomorphism with respect to topologies $\Wr{\infty}$.
\end{itemize}

\medskip

We will now recall sufficient conditions for \condShAop\ obtained in \;\cite{Maks:loc-inv-shifts}.

Let $\Vman$ be a compact $\Dm$-submanifold, and $\Uman$ be an open neighbourhood of $\Vman$.
Then the restriction $\AFld|_{\Uman}$ of $\AFld$ to $\Uman$ generates a local flow $$\AFlowU:\Uman\times\RRR\supset\domAU\longrightarrow\Uman.$$

The corresponding shift map of $\AFld|_{\Uman}$ will be denoted by $\ShAUV$. 
Thus $$\ShAUV:\funcAUV\to\imShAUV.$$

Let us introduce the following conditions for $\Vman$ and $\Uman$.

\medskip

\begin{itemize}
 \item[\condShAVop]
\it The shift map $\ShAV:\funcAV\to\imShAV$ is $\Wr{\infty}$-open, that is open between the corresponding topologies $\Wr{\infty}$;

 \item[\condShAUVop]
\it The shift map $\ShAUV:\funcAUV\to\imShAUV$ is $\Wr{\infty}$-open.

 \item[\condImShUopImSh]
\it The set $\imShAUV$ is $\Wr{\infty}$-open in $\imShAV$.
\end{itemize}

\medskip

Finally for each point $z\in\Mman$ consider the following properties.

\medskip

\begin{itemize}
\item[\condBsShAVop]
\it There exists a base $\beta_{z}=\{\Vman_{j}\}_{j\in J}$ at $z$ consisting of compact $\Dm$-neighbour\-hoods of $z$ such that every $\Vman\in\beta_{z}$ satisfies \condShAVop.

\item[\condBsShAUVop]
\it There exist an open neighbourhood $\Uman$ of $z$ and a base $\beta_{z}=\{\Vman_{j}\}_{j\in J} \subset\Uman$ at $z$ consisting of compact $\Dm$-neighbour\-hoods such that every $\Vman\in\beta_{z}$ satisfies \condShAUVop.

 \item[\condBsimShimShUop]
\it There exists a neighbourhood $\Uman$ of $z$ every compact $\Dm$-subma\-nifold $\Vman\subset\Uman$ satisfies \condImShUopImSh.

 \item[\condznpnrec]
\it $z$ is non-periodic and non-recurrent.

 \item[\condzpid]
\it $z$ is periodic and the Poincar\'e return map of the orbit $\orb_{z}$ of $z$ is the identity.

 \item[\condzUinv]
\it $z\in\FixA$ and there exists an $\AFld$-invariant neighbourhood $\Wman$ of $z$.

 \item[\condzUpbi]
\it $z\in\FixA$ and there exists a neighbourhood $\Wman$ of $z$ with the following property:
if $x\in\partial\Wman=\overline{\Wman}\setminus\Wman$ then there exists a neighbourhood $\gamma\subset\orb_{x}$ of $x$ in the orbit $\orb_x$ such that $\gamma\cap\partial\Wman = \{x\}$.
\end{itemize}

\begin{theorem}\label{th:openness_of_shift_maps}\cite{Maks:loc-inv-shifts}
\it The following implications hold true:

{\rm1)}~~~ \condShAVop \ \ $\Leftrightarrow$ \ \ \condShAUVop\ \ $\&$  \condImShUopImSh;

{\rm2)}~~~ \condBsShAVop\ for every $z\in\Mman$ \ \ $\Rightarrow$ \ \ \condShAop;

{\rm3)}~~~ \condznpnrec\ \ $\vee$ \ \condzpid\  \ $\Rightarrow$ \ \ \condBsShAVop;

{\rm4)}~~~ \condzUinv\ \ $\vee$ \ \condzUpbi\  \ $\Rightarrow$ \ \ \condBsimShimShUop.

In particular, suppose that every regular point $z\in\Mman\setminus\FixA$ of $\AFld$ satisfies either of the conditions \condznpnrec, \condzpid, and every singular point $z\in\FixA$ of $\AFld$ satisfies  \condzUinv, \condzUpbi, and \condBsShAUVop.
Then the shift map $\ShA:\Ci{\Mman}{\RRR}\to\imShA$ is either a homeomorphism or a $\ZZZ$-covering map between topologies $\Wr{\infty}$.
\end{theorem}

\begin{remark}\label{rem:ShAUV_open_doesnotdependonU}\rm
Statement 1) of Theorem\;\ref{th:openness_of_shift_maps} in particular implies that if the map $\ShAV$ is open (condition \condShAVop), then the map $\ShAUV$ is open for \myemph{any} open neighbourhood $\Uman$ of $\Vman$ (condition \condShAUVop).
\end{remark}

\section{Functions on surfaces}\label{sect:func_on_surf}
In this section we will assume that $\Mman$ is a compact surface and $\func:\Mman\to\Psp$ a $\Cinf$ map satisfying the following conditions:
\begin{enumerate}
 \item[(i)]
$\func$ takes constant value on each connected component of $\partial\Mman$
 \item[(ii)]
all critical points of $\func$ are isolated and contained in $\Int{\Mman}$.
\end{enumerate}

Let $z\in\Int\Mman$.
Then in some local charts at $z$ in $\Mman$ and at $\func(z)$ in $\Psp$ we can regard $\func$ as a function 
\begin{equation}\label{equ:loc_pres_func}
(\Mman,z)\; \supset \;(\CCC,0) \xrightarrow{~~\bar\func~~} (\RRR,0)\; \subset \; (\Psp,\func(z))
\end{equation}
such that $z=0\in\CCC$ and $\bar\func(0)=0\in\RRR$.

We say that $z$ is a \myemph{local extreme} for $\func$ if it is a local extreme for $\bar\func$.
Moreover, if we fix an orientation of $\Psp$ and restrict ourselves to representations~\eqref{equ:loc_pres_func} in which the embedding $\RRR\subset\Psp$ preserves orientation, then $z$ will be called a \myemph{local maximum (minimum)} of $\func$ whenever so is $0\in\CCC$ for $\bar\func$.

\subsection{Isolated critical points.}\label{sect:isolated_cr_pt}
Suppose that $z\in\Int\Mman$ is an isolated critical point of $\func$.
Then there exists a germ of a \myemph{homeomorphism} $\dif:(\CCC,0)\to(\CCC,0)$ such that
$$
\bar\func\circ\dif(z) = 
\left\{
\begin{array}{ll}
\pm|z|^2, & \text{if $0$ is a local extremum, \cite{Dancer:2:JRAM:1987}},  \\
\text{Re}(z^n) & \text{for some $n\in\NNN$, otherwise, \cite{Prishlyak:TA:2002}}.
\end{array}
\right.
$$
If $0$ is not a local extreme for $\func$, then the number $n$ does not depend on a particular choice of $\dif$, and in this case $z$ will be called a \myemph{(generalized) $n$-saddle}.

The topological structure of the foliation $\partf$ near local extremes, $1$-, and $3$-saddles is illustrated in Figure~\ref{fig:isol_pt}.
The corresponding critical components of level-set of $\func$ are designed in bold.
\begin{figure}[ht]
\begin{center}
\begin{tabular}{ccccc}
\includegraphics[width=2cm]{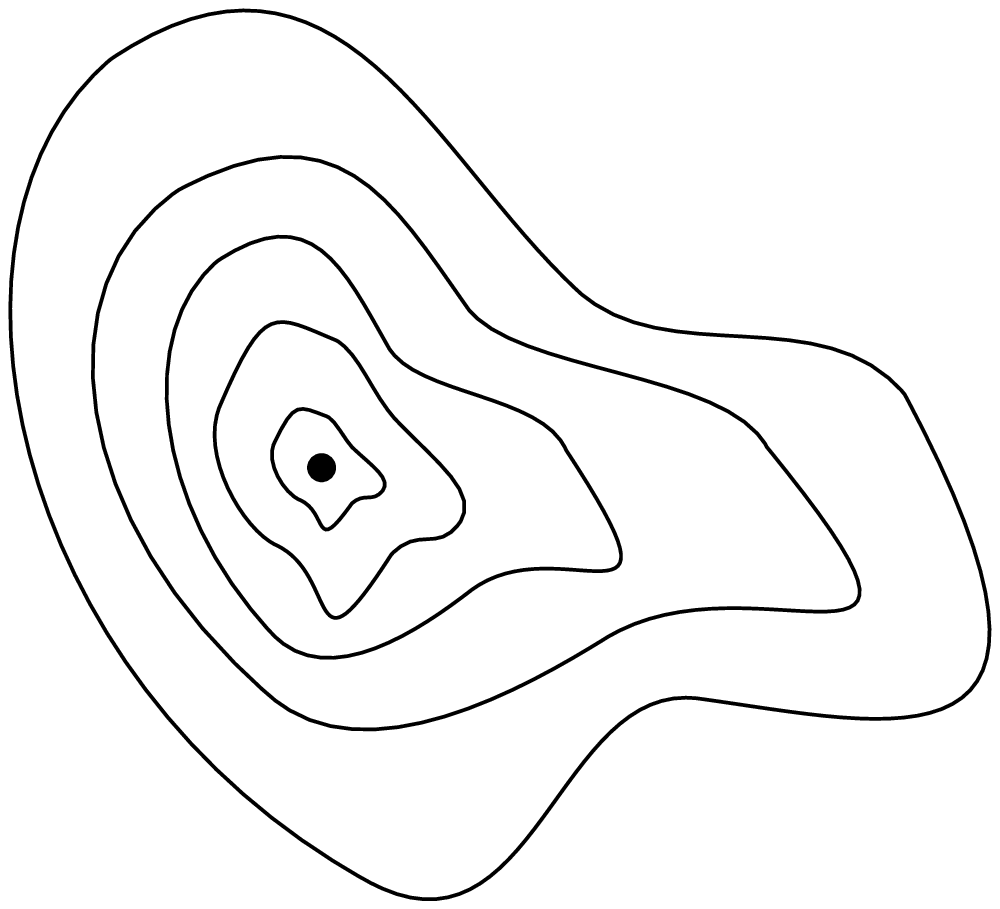} 
& \quad &
\includegraphics[width=3cm]{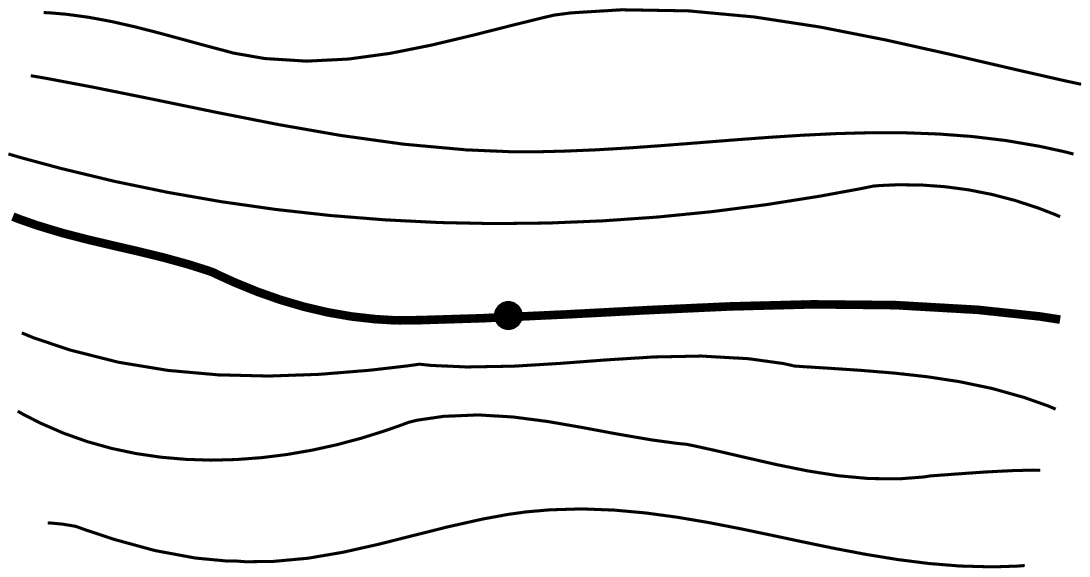} 
& \quad &
\includegraphics[width=3cm]{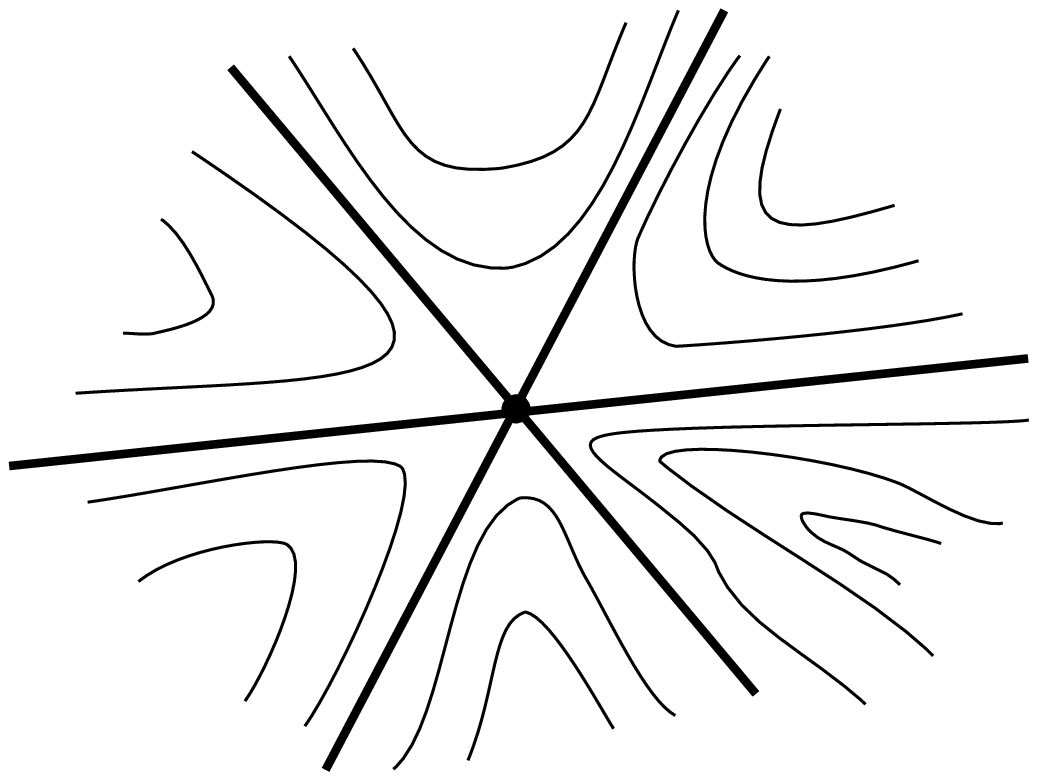} \\
a) local extreme & &
b) $1$-saddle & & 
c) $3$-saddle
\end{tabular}
\end{center}
\caption{Isolated critical points}\protect\label{fig:isol_pt}
\end{figure}

\subsection{Foliation $\partf$ of $\func$}\label{sect:Foliation-of-f}
Notice that $\func$ defines on $\Mman$ a certain one-dimensional foliation $\partf$ with singularities in the following way: \myemph{a subset $\omega\subset\Mman$ is a leaf of $\partf$ if and only if $\omega$ is either a critical point of $\func$ or a connected component of the set $\func^{-1}(c)\setminus\singf$ for some $c\in\Psp$}.
Thus the leaves of $\partf$ are $1$-dimensional submanifolds of $\Mman$ and singular points of $\func$.

Denote by $\partfreg$ the union of all leaves of $\partf$ homeomorphic to the circle and by $\partfcr$ the union of all other leaves.
It follows from (i) that $\partial\Mman \subset\partfreg$.

The leaves in $\partfreg$ (resp. $\partfcr$) will be called \myemph{regular} (resp. \myemph{critical}).
Similarly, connected components of $\partfreg$ (resp. $\partfcr$) will be called \myemph{regular} (resp. \myemph{critical}) components of $\partf$.

Evidently, every critical leaf of $\partfcr$ either homeomorphic to an open interval or is a singular point of $\func$.
If $\func$ has at least one critical point or $\partial\Mman=\varnothing$, then every regular component of $\partf$ is diffeomorphic with $S^1\times(0,1)$.

Denote by $\Dpartf$ the group of diffeomorphisms $\dif$ of $\Mman$ such that $\dif(\omega)=\omega$ for every leaf $\omega$ of $\partf$, and let $\Dpartfpl$ be its subgroup consisting of diffeomorphisms of $\Mman$ preserving orientations of all $1$-dimensional leaves of $\partf$.

For each critical point $z\in\singf$ we will denote by $\gDpartfz$ the group of germs of diffeomorphisms $\dif:(\Mman,z)\to(\Mman,z)$ at $z$ preserving leaves of $\partf$.
More precisely, let $\Vman$ be a neighbourhood of $z$ and $\dif:\Vman\to\Mman$ be an embedding such that $\dif(z)=z$.
Then $\dif\in\gDpartfz$ if and only if $\dif(\omega\cap\Vman)\subset\omega$ for each leaf $\omega$ of $\partf$.

\begin{lemma}\label{lm:Dpartf_Stab_identity_comp}{\rm\cite[Lm.\;3.4]{Maks:AGAG:2006}}
If $\func$ satisfies {\rm(i)} and {\rm(ii)}, then
$$\DiffId^{+}(\partf)^{r} \ = \ \DiffId(\partf)^{r} \ = \ \StabIdfr{r}, \qquad 0 \leq r \leq \infty.$$
\end{lemma}
\begin{proof}
In \cite[Lm.\;3.4]{Maks:AGAG:2006} the proof was given for the case $r=0$.
But literally the same arguments hold for all $r$ if we replace everywhere in \cite[Lm.\;3.4]{Maks:AGAG:2006} the word ``isotopy'' with ``$r$-isotopy''.
\end{proof}

\subsection{\KR-graph of $\func$}\label{sect:KR-graph-of-f}
Let $\Reebf$ be the \myemph{Kronrod-Reeb graph} (\myemph{\KR-graph\/}) of $\func$, e.g.~\cite{Kronrod:UMN:1950,BolsinovFomenko:1997,Sharko:UMZ:2003,Maks:AGAG:2006}.
This graph is obtained from $\Mman$ by shrinking every connected component of every level-set of $\func$ to a point, see~Figure~\ref{fig:KRgraph}.
Then we have the following decomposition of $\func$:
$$
\begin{CD}
 \func = \KRfunc \circ \prKR: \Mman @>{\prKR}>> \Reebf @>{\KRfunc}>>  \Psp,
\end{CD}
$$
where $\prKR$ is a factor-map and $\KRfunc$ is the induced function which will be caller \myemph{\KR-function} for $\func$.
Evidently, the edges (resp. vertexes) of $\Reebf$ correspond to regular (resp. critical) components of $\partf$.

\begin{figure}[ht]
\centerline{\includegraphics[height=3cm]{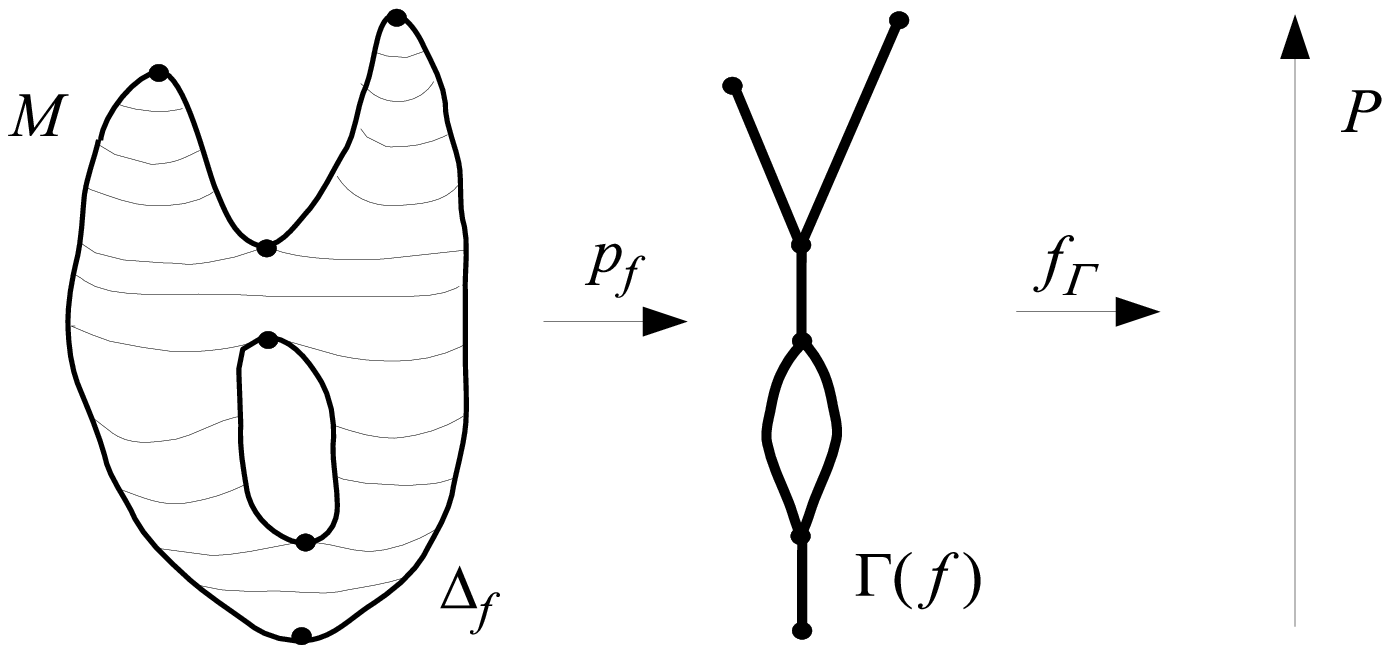}}
\caption{Foliation $\partf$ and \KR-graph $\Reebf$ of $\func$}\protect\label{fig:KRgraph}
\end{figure}

\subsection{The action of $\Stabf$ on $\Reebf$}
Notice that every $\dif\in\Stabf$ interchanges the leaves of $\partf$ and even yields a homeomorphism of $\Reebf$.
So we have a natural homomorphism $\stonkr:\Stabf\to\AutfR$.

Evidently, $\dif\in\ker(\stonkr)$ iff $\dif$ preserves every regular leaf of $\partfreg$.
On the other hand it is easy to construct examples when $\dif\in\ker(\stonkr)$ interchanges critical leaves of $\partf$.
It follows that 
$$
\Dpartfpl \; \subset \; \Dpartf \; \subset \; \ker(\stonkr).
$$

\begin{remark}\label{rem:corrections}\rm
I should warn the reader that \cite[Eq.\;(8.6)]{Maks:AGAG:2006} wrongly claims that $\Dpartfpl\cap\DiffIdM=\ker(\stonkr)\cap\DiffIdM$.
In fact, the proof contains a misprint and a mistake: it refers to\;\cite[Lm.\;3.6(2)]{Maks:AGAG:2006} which does not exist instead of \;\cite[Lm.\;3.5(2)]{Maks:AGAG:2006}.
Nonetheless, \cite[Lm.\;3.5(2)]{Maks:AGAG:2006} is not applicable since it claims about $\dif\in\StabIdf$ but not about $\dif\in\Dpartfpl\cap\DiffIdM$: the difference is that every $\dif\in\StabIdf$ is isotopic to $\id_{\Mman}$ via an $\func$-preserving isotopy, while $\dif\in\Dpartfpl\cap\DiffIdM$ also preserves $\func$ and is isotopic to $\id_{\Mman}$ but the isotopy is not assumed to be $\func$-preserving.
On the other hand, it follows from\;\cite[Lm.\;3.5(1)]{Maks:AGAG:2006} that \cite[Eq.\;(8.6)]{Maks:AGAG:2006} holds for orientable surfaces.
The following example shows that a difference between $\Dpartfpl\cap\DiffIdM$ and $\ker(\stonkr)\cap\DiffIdM$ appears indeed.
\end{remark}

\begin{example}\label{exmp:func_on_prjplane}\rm
Let $\prjplane$ be a real projective plane and $\func:\prjplane\to\RRR$ be a Morse function having exactly one minimum $x$, one saddle $y$ and one maximum $z$.
Regard $\prjplane$ as a space obtained from unit $2$-disk $D^2$ by identifying the opposite points on its boundary $\partial D^2$.
The foliation $\partf$ on $D^2$ and the \KR-graph $\Reebf$ of $\func$ are shown in Figure\;\ref{fig:rp1func}.
The vertexes of $\Reebf$ are denoted by the same letters as the corresponding critical points of $\func$ and the up-down arrows show how to glue the boundary of $D^2$.

Let $\dif:D^2\to D^2$ be a mirror symmetry with respect to the horizontal line passing through $x$, see Figure\;\ref{fig:rp1func}.
Then $\dif$ trivially acts on $\Reebf$ though it changes orientations of all regular leaves.
On the other hand, since $\Diff(\prjplane)$ is connected, it follows that $\dif$ is isotopic to $\id_{\prjplane}$.
Thus 
$ \dif \in (\ker(\stonkr) \setminus \Dpartfpl) \cap\DiffId(\prjplane).$
\begin{figure}[ht]
\centerline{\includegraphics[height=4cm]{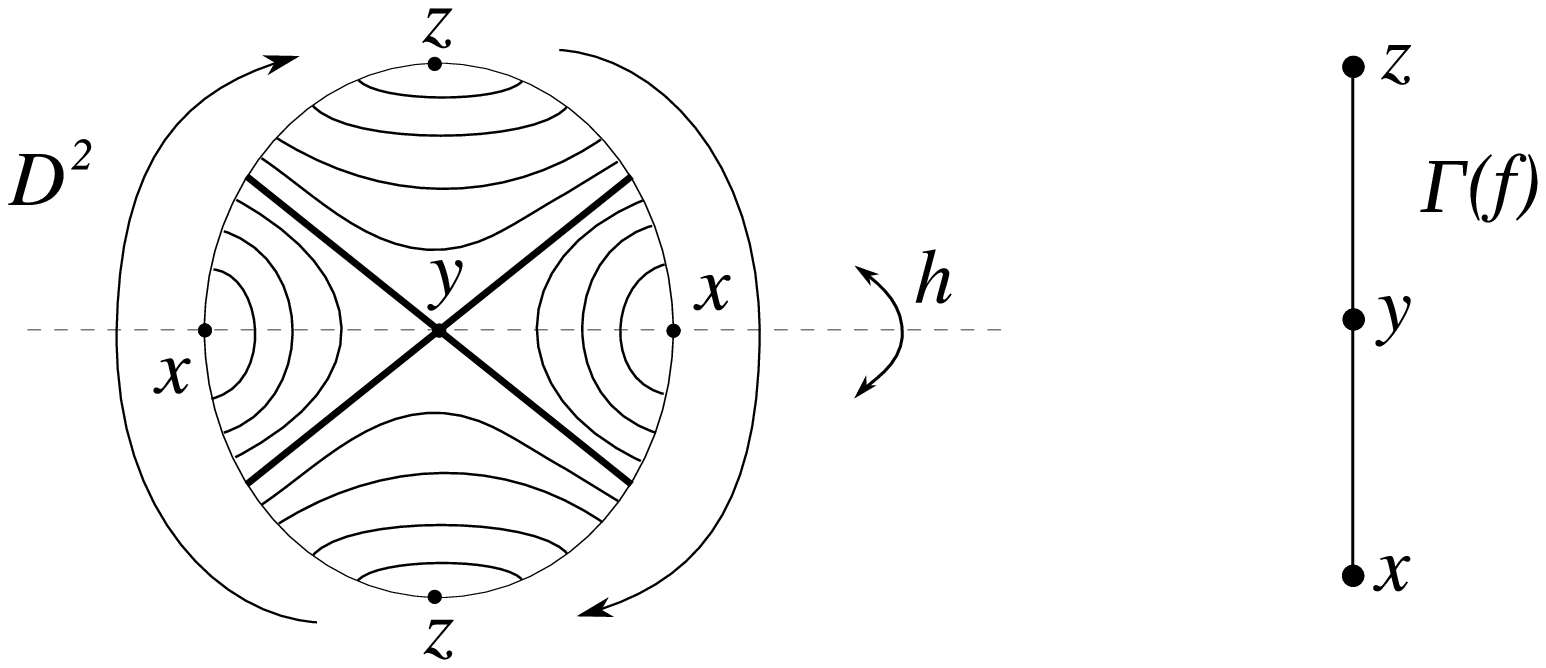}}
\caption{}\protect\label{fig:rp1func}
\end{figure}
\end{example}

\section{Special critical points}\label{sect:special-crit-points}
In this section we introduce three types of critical point which will play a key role throughout the paper.

\begin{definition}(Special critical points)\label{def:special-cr-pt}
\it Let $z\in\singf$ be an isolated critical point of $\func$.
We will say that $z$ is \myemph{special} for $\func$ if there exists a neighbourhood $\Uman$ of $z$ and a vector field $\AFld$ on $\Uman$ with the following properties:
\begin{itemize}
 \item[\SPECA] 
$d\func(\AFld)\equiv 0$ and $z$ is a unique singular point of $\AFld$;
 \item[\SPECB]
there is a base $\beta_{z}=\{\Vman_{j}\}_{j\in J}$ of compact connected $\Dm$-neigh\-bourhoods of $z$ such that for each $\Vman\in\beta_{z}$ the shift mapping
$\ShAV:\funcAV\to\imShAV$ is a local homeomorphism with respect to topologies $\Wr{\infty}$.
\end{itemize}
The corresponding vector field $\AFld$ will be called \myemph{special} as well.
\end{definition}

Let $z\in\singf$ be a special critical point of $\func$ and $\AFld$ be a vector field at $z$ satisfying \SPECA\ and \SPECB.
Then by \SPECA\ the orbits of $\AFld$ coincide with the level-curves of $\func$ near $z$, whence 
\begin{equation}\label{equ:gDaz_gDpartfz}
\gDAz = \gDpartfz.
\end{equation}
Moreover, it follows from Remark\;\ref{rem:ShAUV_open_doesnotdependonU} that in Definition\;\ref{def:special-cr-pt} a neighbourhood $\Uman$ can be taken arbitrary small.

\begin{definition}[$\ST$-point]\label{defn:S-point}\it
Say that $z$ is an $\ST$-point for $\func$ if $z$ is not a local extreme for $\func$, (i.e. a \myemph{saddle} point) and it has a vector filed $\AFld$ satisfying \SPECA\ and \SPECB\ and such that $\gimShAz=\gDAz$.

Such a vector field will be called \myemph{special} for $z$.
\end{definition}

Now let $z$ be a local extreme of $\func$.
Then we can assume that $\Uman$ is connected and $\AFld$-invariant, see\;\S\ref{sect:isolated_cr_pt}.
It also follows that all the orbits of $\AFld$ except for $z$ are periodic and wrap around $z$.
Let $\theta:\Uman\setminus z \to (0,+\infty)$ be the function associating to each $x\in\Uman\setminus z$ its period $\Per(x)$.
Then it easily follows from smoothness of the Poincar\'e's first return map that $\theta$ is $\Cinf$ on $\Uman\setminus z$ but it can be even discontinuous at $z$.
We will call $\theta$ the \myemph{period function} for $\AFld$.

It is shown in \cite{Maks:sym-nondeg-topcenter} that after some linear change of coordinates the linear part $\nabla\AFld(z)$ of $\AFld$ at $z$, see\;\eqref{equ:j1Fz}, can be reduced to one of the following forms:
\begin{equation}\label{equ:linear_part_j1Fz}
a)~
\left(\begin{smallmatrix}
 0 & \lambda \\ -\lambda & 0
\end{smallmatrix}\right),
\qquad\qquad
b)~
\left(\begin{smallmatrix}
 0 & \lambda \\ 0 & 0
\end{smallmatrix}\right),
\qquad\qquad
c)~
\left(\begin{smallmatrix}
 0 & 0 \\ 0 & 0
\end{smallmatrix}\right),
\end{equation}
for some $\lambda\in\RRR\setminus\{0\}$.
By\;\eqref{equ:im_j1} the image $\tnz(\gimShAz)$ of $\gimShAz$ under $\tnz$ coincides with $\{e^{\nabla\AFld(z)\cdot t}\}_{t\in\RRR}$, whence we obtain the following three possibilities for $\tnz(\gimShAz)$:
\begin{equation}\label{equ:tnz_giimShAz}
a)~
SO(2),
\qquad\qquad
b)~
\left\{ \left(\begin{smallmatrix}
 1 & t \\ 0 & 1
\end{smallmatrix}\right), t\in\RRR \right\},
\qquad\qquad
c)~
\left\{ \left(\begin{smallmatrix}
 1 & 0 \\ 0 & 1
\end{smallmatrix}\right) \right\}.
\end{equation}

\begin{definition}[$\PT$-point]\label{defn:P-point}
Let $z$ be a special local extreme of $\func$.
Say that $z$ is a \myemph{$\PT$-point} for $\func$ if there exists a vector field $\AFld$ on some neighbourhood $\Uman$ of $z$ satisfying \SPECA\ and \SPECB\ and such that the corresponding period function $\theta:\Uman\setminus z\to(0,+\infty)$ smoothly extends to all of $\Uman$.

In this case $\AFld$ will also be called \myemph{special}.
\end{definition}

\begin{lemma}\label{lm:P-points}\cite{Takens:AIF:1973,Maks:sym-nondeg-topcenter}
Let $z\in\singf$ be a local extreme of $\func$.
Choose some local coordinates $(x,y)\in\RRR^2$ at $z$ in which $z=\orig$.
Let $\AFld=(\AFld_1,\AFld_2)$ be a $\Cinf$ vector field near $z$ such that $\AFld(\func)\equiv0$ and $z$ is a unique singular point of $\AFld$.
Then the following conditions \PTA-\PTF\ are equivalent.
\begin{enumerate}
\item[\PTA]
$z$ is a $\PT$-point for $\func$ and $\AFld$ is the corresponding special vector field for $z$;

\item[\PTB]
The eigen values of $\nabla\AFld(z)$ are non-zero purely imaginary, so $\nabla\AFld(z)$ can be reduced to the form a) of\;\eqref{equ:linear_part_j1Fz}.

 \item[\PTC]
There exist germs at $z$ of $\Cinf$ functions $\bfunc, \dAx, \dAy:\Uman\to\RRR$ such that $\bfunc(z)\not=0$, and $\dAx$ and $\dAy$ are \myemph{flat} at $z$, and $\AFld$ is $\Cinf$ equivalent to the following vector field
$$
\bfunc(x^2+y^2) \left( -y \dd{x} + x \dd{y} \right) + \dAx\dd{x}+ \dAy\dd{y}.
$$
\item[\PTD]
There exist a connected smooth $2$-submanifold $\Vman\subset\Uman$ such that $z\in\Int{\Vman}$ and  the shift map $\ShAV$ is periodic, i.e. $\ker(\ShAV)\not=\{0\}$.

\item[\PTE]
For every connected smooth $2$-submanifold $\Vman\subset\Uman$ the shift map $\ShAV$ is periodic.

\item[\PTF]
$\tnz(\gimShAz)$ is conjugate in $\GLR{2}$ to $\mathrm{SO}(2)$.
\end{enumerate}

In these cases $\gimShAz = \gDpAz$, $\ker(\ShAV)=\{n\theta|_{\Vman}\}_{n\in\ZZZ}$ for any connected $2$-submanifold $\Vman\subset\Uman$.

Moreover, there are $\Cinf$ germs $g,\mu:\Uman\to\RRR$ at $z$ such that $\mu$ is flat at $z=\orig$ and $\func$ is $\Cinf$ equivalent to the function $g(x^2+y^2)+\mu(x,y)$.

If either of conditions \PTA-\PTF\ is violated, then $\lim\limits_{x\to z}\theta(z) = +\infty$.
\end{lemma}
\begin{proof}
Equivalence \PTA$\Leftrightarrow$\PTB$\Leftrightarrow$\PTC\ was established in \cite{Maks:sym-nondeg-topcenter}.
In fact F.\;Takens\;\cite{Takens:AIF:1973} described normal forms for vector fields satisfying \PTB.
He proved that there are two types of such forms.
The first one is described by \PTC.
Moreover, it was observed in\;\cite{Maks:sym-nondeg-topcenter} that vector fields of the second type have non-closed orbits near $z$, whence in our case $\AFld$ can belong only to the first type.
This implies \PTB$\Rightarrow$\PTC.
The proof of \PTC$\Rightarrow$\PTA$\Rightarrow$\PTB\ is one of the main results of  \cite{Maks:sym-nondeg-topcenter}.

\PTD$\Rightarrow$\PTA.
Suppose $\ker(\ShAV)=\{n\nu\}_{n\in\ZZZ}$ for some $\Cinf$ function $\nu:\Vman\to(0,+\infty)$.
Then by Lemma\;\ref{lm:ker_of_shift_map} $\nu=\Per\equiv\theta$ on open and every where dense subset $Q\subset\Vman\setminus\{z\}$, whence $\nu=\theta$ on all of $\Vman\setminus\{z\}$, and thus $\theta$ extends to a $\Cinf$ function $\nu$ near $z$.

\PTA$\Rightarrow$\PTE.
Suppose $\theta$ is $\Cinf$ on all of $\Uman$ and let $\Vman\subset\Uman$ be a $2$-submanifold.
Then $\AFlow(x,\theta(x))=x$ for all $x\in\Vman$, whence $\theta|_{\Vman}\in\ker(\ShAV)\not=\{0\}$.
By the previous arguments, $\ker(\ShAV)=\{n\theta|_{\Vman}\}_{n\in\ZZZ}$.

\PTE$\Rightarrow$\PTD\ is evident, and \PTB$\Leftrightarrow$\PTF\ follows from\;\eqref{equ:tnz_giimShAz}.

All other statements are also proved in\;\cite{Maks:sym-nondeg-topcenter}.
\end{proof}

\begin{definition}[$\NT$-point]\label{defn:N-point}\it
Let $z$ be a special local extreme of $\func$.
Say that $z$ is an \myemph{$\NT$-point} for $\func$ if there exists a vector field $\AFld$ near $z$ satisfying \SPECA\ and \SPECB\ and such that
\begin{enumerate}
 \item[\NA]
the corresponding period function $\theta:\Uman\setminus\{z\}\to(0,+\infty)$ can not be continuously extended to all of $\Uman$, so by Lemma\;\ref{lm:P-points} $\lim\limits_{x\to z}\theta(z) = +\infty$;
 \item[\NB]
$\ker(\tnz) \subset \gimShAz$;
 \item[\NC]
if $\nabla\AFld(z)=0$, then $\tnz(\gDAz)$ is finite.
\end{enumerate}
Again, such a vector field will be called \myemph{special}.
\end{definition}
Condition \NB\ means that for each $\dif\in\gDAz$ with $\tnz(\dif)=\id$ there exists a germ of a $\Cinf$ function $\afunc\in\gCiMRz$ such that $\dif(x) = \AFlow(x,\afunc(x))$ for all $x$ sufficiently close to $z$.

Consider the following subsets of $\GLR{2}$:
$$
\begin{array}{rclcrcl}
\Agrp_{++} &=& \left\{ \left(\begin{smallmatrix}
 1 &  t \\ 0 & 1
\end{smallmatrix}\right) \ | \ t\in\RRR \right\}, &\qquad &
\Agrp_{--} &=& \left\{ \left(\begin{smallmatrix}
 -1 &  t \\ 0 & -1
\end{smallmatrix}\right) \ | \  t\in\RRR \right\}, 
\\ [2mm]
\Agrp_{+-} &=& \left\{ \left(\begin{smallmatrix}
 1 &  t \\ 0 & -1
\end{smallmatrix}\right) \ | \ t\in\RRR \right\}, & &
\Agrp_{-+} &=& \left\{ \left(\begin{smallmatrix}
 -1 &  t \\ 0 & 1
\end{smallmatrix}\right) \ | \  t\in\RRR \right\},
\end{array}
$$
$$
\Agrp_{+}=\Agrp_{++}\cup \Agrp_{--}, \qquad 
\Agrp=\Agrp_{++}\cup \Agrp_{--}\cup \Agrp_{-+} \cup \Agrp_{+-}.
$$
Then $\Agrp_{+} \subset \Agrp$ are subgroups of $\GLR{2}$, $\Agrp_{++}$ is the unity component of both $\Agrp$ and $\Agrp_{+}$, $\Agrp_{+}/\Agrp_{++}\approx\ZZZ_2$, and $\Agrp/\Agrp_{++}\approx\ZZZ_2\times\ZZZ_2$.

\begin{lemma}\label{lm:N-points}\cite{Maks:sym-deg-topcenter}
Let $z$ be an $\NT$-point of $\func$ and $\AFld$ be a special vector field for $z$.
Then the following statements hold.
\begin{enumerate}
 \item[\NTA]
$\nabla\AFld(z)$ is nilpotent, whence it is similar to one of the matrices {\rm b)} or {\rm c)} of\;\eqref{equ:linear_part_j1Fz}.
 \item[\NTB]
If $\nabla\AFld(z)=\left(\begin{smallmatrix}
0 & 1 \\ 0 & 0
\end{smallmatrix}\right)$, see {\rm b)} of\;\eqref{equ:linear_part_j1Fz}, then 
\begin{gather}
\tnz(\gimShAz)=\Agrp_{++}, \label{equ:tnz_gimShAz_App} \\
\gimShAz=\tnz^{-1}(\Agrp_{++}), \label{equ:gimShAz_inv_tnz_App} \\
\tnz(\gDpAz)\subseteq \Agrp_{+}, \label{equ:tnz_gpDAz} \\
\tnz(\gDAz)\subseteq \Agrp.  \label{equ:tnz_gDAz}
\end{gather}
Hence either $\gimShAz=\gDpAz$ or 
$\gDpAz/\gimShAz\approx\ZZZ_2$, and $\gDpAz/\gimShAz$ is isomorphic with a subgroup of $\ZZZ_2\times\ZZZ_2$.

\item[\NTC]
If $\nabla\AFld(z)=\left(\begin{smallmatrix}
0 & 0 \\ 0 & 0
\end{smallmatrix}\right)$,
see {\rm c)} of\;\eqref{equ:linear_part_j1Fz}, then $\ker(\tnz)=\gimShAz$, and the image $\tnz(\gDpAz)$ is a finite cyclic subgroup of $\GLR{2}$ of some order $n_z$, so $\gDpAz/\gimShAz\approx\ZZZ_{n_z}$.
If in addition $\gDAz\not=\gDpAz$, then $\gDpAz/\gimShAz$ is a dihedral group $\DDD_{n_z}$.
\end{enumerate}
\end{lemma}
\begin{proof}
\NTA\ follows from \PTF\ and Eq.\;\eqref{equ:linear_part_j1Fz}.

\NTB. 
Eq.\;\eqref{equ:tnz_gimShAz_App} follows from\;\eqref{equ:im_j1}.
To establish Eq.\;\eqref{equ:gimShAz_inv_tnz_App} consider $\dif\in\tnz^{-1}(\Agrp_{++})$.
Since $\tnz(\gimShAz)=\Agrp_{++}$, there exists $\gdif\in\gimShAz$ such that $\tnz(\gdif)=\tnz(\dif)$.
Hence $\gdif^{-1}\circ\dif\in \ker(\tnz) \subset\gimShAz$, whence $\dif\in\gimShAz$ as well.

Eq.\;\eqref{equ:tnz_gDAz} is proved in\;\cite{Maks:sym-nondeg-topcenter}, whence  $\tnz(\gDpAz)\subset \Agrp \cap \GLRp{2}=\Agrp_{+}.$
This proves Eq.\;\eqref{equ:tnz_gpDAz}.

\NTC\ follows from the well-known fact that every finite subgroup of $\GLRp{n}$ is cyclic.
\end{proof}

Due to \NTB\ and \NTC\ of Lemma\;\ref{lm:N-points} we will distinguish two types of $\NT$-points.

\begin{definition}
An $\NT$-point $z$ of $\func$ will be called an \myemph{$\NNT$-point} if $\nabla\AFld(z)$ is non-zero nilpotent.
Otherwise, $\nabla\AFld(z)=0$ and we will call $z$ an \myemph{$\NZT$-point}.
\end{definition}

\subsection{Examples}\label{sect:examples_of_SPN_points}
Let $\func:\RRR^2\to\RRR$ be a homogeneous polynomial \myemph{without multiple linear factors}, that is
\begin{equation}\label{equ:g_homog_poly}
\func=L_1 \cdots L_{a} \cdot Q_1^{q_1} \cdots Q_{b}^{q_b}, 
\end{equation}
where $L_i$, $(i=1,\ldots,a)$, is a non-zero linear function, $Q_j$, $(j=1,\ldots,b)$, is an irreducible over $\RRR$ (definite) quadratic form, $q_j\geq 1$, $L_i/L_{i'}\not=\mathrm{const}$ for $i\not=i'$, and $Q_j/Q_{j'}\not=\mathrm{const}$ for $j\not=j'$.
Put
$$
D = Q_1^{q_1-1} \cdots Q_{b}^{q_b-1}.
$$
Then $\func=L_1 \cdots L_{a} \cdot Q_1 \cdots Q_{b}\cdot D$ and it is easy to see that \myemph{$D$ is the greatest common divisor of the partial derivatives $\func'_{x}$ and $\func'_{y}$}.
The following \myemph{polynomial} vector field on $\RRR^2$:
$$
\AFld(x,y)=-(\func'_{y}/D)\,\tfrac{\partial}{\partial x} + (\func'_{x}/D)\, \tfrac{\partial}{\partial y}
$$ 
will be called the \myemph{reduced Hamiltonian} vector field of $\func$.
In particular, if $\func$ has no multiple factors, i.e. $l_i=q_j=1$ for all $i,j$, then $D\equiv1$ and $\AFld$ is the usual \myemph{Hamiltonian} vector field of $g$.

Notice that if $\func$ had multiple linear factors, then $0\in\RRR^2$ would not be an isolated critical point.

\begin{lemma}
The origin $0\in\RRR^2$ is a \myemph{special} critical point of $\func$ belonging to one of the types $\ST$, $\PT$, or $\NT$, and $\AFld$ is the corresponding special vector field.
More precisely,
\begin{enumerate}
 \item 
if $a>0$, then $0$ is an $\ST$-point for $\func$;
 \item
if $a=0$ and $b=1$, i.e. $\func=Q_1^{q_1}$, and thus in some local coordinates $\func(x,y)=(x^2+y^2)^{q_1}$, then $0\in\RRR^2$ is a $\PT$-point of $\func$;
 \item
otherwise, when $a=0$ and $b\geq2$, so $\func=Q_1^{q_1}\cdots Q_b^{q_b}$, then $0\in\RRR^2$ is an $\NT$-point (even an $\NZT$-point) of $\func$.
\end{enumerate}
\end{lemma}
\begin{proof}
The assumption \SPECA\ of the Definition\;\ref{def:special-cr-pt} that $\AFld(\func)\equiv0$ is evident, while \SPECB\ is proved in\;\cite{Maks:loc-inv-shifts}.
Hence $0\in\RRR^2$ is special.

(1) If $a>0$, then the identity $\gimShAz=\gDpAz$ follows from\;\cite{Maks:hamv2,Maks:CEJM:2009}, see\;\cite[Th.\;11.1]{Maks:CEJM:2009}.

(2) Suppose $a=0$ and $b=1$.
Then we can assume that $\func(x,y)=(x^2+y^2)^{q_1}$, whence $\AFld(x,y)=-y\dd{x}+x\dd{y}$.
Hence by \PTB\ of Lemma\;\ref{lm:P-points} $0$ is a $\PT$ point for $\func$.
Moreover, the assumption \SPECB\ of the Definition\;\ref{def:special-cr-pt} is independently reproved in\;\cite{Maks:sym-nondeg-topcenter}.

(3) Suppose $a=0$ and $b\geq2$.
Then (i) of Definition\;\ref{defn:N-point} is established in\;\cite{Maks:sym-nondeg-topcenter},
and (ii) and (iii) in\;\cite{Maks:hamv2,Maks:CEJM:2009}, see\;\cite[Th.\;7.1\;\&\;11.1]{Maks:CEJM:2009}.
It is also easy to see that $\nabla\AFld(0)=0$, so $0$ is an $\NZT$-point.
\end{proof}

\subsection{Extension of shift functions}
In this paragraph we prove two lemmas about extensions of shift functions.
 
\begin{lemma}\label{lm:ext_shift_func_for_PN_points}
Let $z$ be a special critical point of $\func$ being a local extreme, $\AFld$ be the corresponding special vector field defined on some neighbourhood $\Uman$ of $z$, $\Wman\subset\Vman$ be two open connected $\AFld$-invariant neighbourhoods of $z$ such that $\overline{\Wman}\subset\Vman$, and $\dif:\Vman\to\Vman$ be a diffeomorphism preserving orientation and orbits of $\AFld$.

{\rm(1)}~Let $Y\subset\Vman\setminus\{z\}$ be any open connected subset and $\afunc:Y\to\RRR$ be any shift function for $\dif$ on $Y$.
If $z$ is a $\PT$-point for $\func$, then $\afunc$ uniquely extends to a unique $\Cinf$ shift function for $\dif$ on all of $\Vman$.

{\rm(2)}~Any shift function $\afunc:\Wman\to\RRR$ for $\dif$ on $\Wman$ uniquely extends to a $\Cinf$ shift function for $\dif$ on all of $\Vman$.
\end{lemma}
\begin{proof}
Since $\dif$ preserves orientation, it is not hard to show that $\dif$ has a $\Cinf$ shift function $\bfunc:\Vman\setminus\{z\}\to\RRR$, see\;\cite{Maks:sym-nondeg-topcenter}.
Such a function is defined up to a summand $n\theta$, $(n\in\ZZZ)$.
In particular, any shift function for $\dif$ defined on $Y\subset\Vman\setminus\{z\}$ coincides with $\bfunc+n\theta$ for some $n\in\ZZZ$.

(1) Suppose $z$ is a $\PT$-point, so $\theta$ extends to $\Cinf$ function on $\Vman$.
As just noted $\afunc=\bfunc+n\theta$ on $Y$ for some $n$.

We claim that \myemph{$\bfunc$ extends to a $\Cinf$ function on $\Vman$.}
Indeed, by Lemma\;\ref{lm:P-points} $\dif\in\gDAz=\gimShAz$, so $\dif$ has a $\Cinf$ shift function $\bfunc'$ on some neighbourhood of $z$.
Then $\bfunc'=\bfunc+k\theta$ for some $k\in\ZZZ$.
Since $\bfunc'$ and $\theta$ are $\Cinf$ near $z$, so is $\bfunc$.

Hence $\afunc$ also uniquely extends to a $\Cinf$ function $\bfunc+n\theta$ on $\Vman$.

(2) We have that $\afunc=\bfunc+n\theta$ on $\Wman\setminus\{z\}$ for some $n$, though $\bfunc$ and $\theta$ are even not defined at $z$.
On the other hand they are $\Cinf$ on $\Vman\setminus\{z\}$, so we define $\afunc$ on $\Vman\setminus\Wman$ by $\afunc=\bfunc+n\theta$.
Then $\afunc$ is $\Cinf$ on all of $\Vman$.
\end{proof}

\begin{lemma}\label{lm:ext_shfunc_under_homotopies}
Let $z\in\singf$ be a critical point of $\func$ of either of types $\ST$ or $\PT$ or $\NT$, $\AFld$ be a special vector field for $z$ defined on some neighbourhood $\Uman$ of $z$ containing no other critical points of $\func$, $\Vman\subset\Uman$ be an open neighbourhood of $z$, and $\VmanS=\Vman\setminus\{z\}$.
Let also $H:\Vman\times I\to\Uman$ be an $r$-homotopy in $\EAV$ such that $H_0=i_{\Vman}:\Vman\subset\Uman$, and $\Lambda:\VmanS\times I\to\RRR$ be a unique $r$-homotopy such that $\Lambda_0=0$ and $\Lambda_t$ is a smooth shift function for $H_t$ on $\VmanS$ for every $t\in I$, see Lemma~\ref{lm:shift-func-on-reg}.
If $z$ is an $\NT$-point suppose in addition that $r\geq1$.
Then for each $t \in I$ the function $\Lambda_t$ extends to a $\Cinf$ function on all of $\Vman$.
\end{lemma}
\begin{proof}
First we show that $H_t\in\gimShAz$ for each $t\in I$, i.e. $H_t$ has a $\Cinf$ shift function $\afunc_t$ on some neighbourhood $\Wman$ of $z$ in $\Vman$.

Indeed, since $H_0=i_{\Vman}$, it follows the germ of $H_t$ at $z$ belongs to $\gDpAz$.
Hence if $z$ is either an $\ST$- or a $\PT$-point, then by Definitions\;\ref{defn:S-point} and \ref{defn:P-point} $H_t\in\gimShAz$.

Suppose $z$ is an $\NT$-point.
Then by assumption $r\geq1$, so the matrix $\tnz(H_t)$ continuously depends on $t$.
If $\nabla\AFld(z)$ is non-zero nilpotent, then we get from \NTC\ of Lemma\;\ref{lm:N-points} that $\tnz(H_t)$ belongs to the unity component $\Agrp_{++}$ of the group $A$ as well as $\tnz(H_0)=\id$.
Hence $H_t \in \tnz^{-1}(\Agrp_{++})=\gimShAz$.

Similarly, if $\nabla\AFld(z)=0$, then the set of possible values for $\tnz(H_t)$ is finite, whence $\tnz(H_t)=\tnz(H_0)=\id$ for all $t\in I$.
Therefore by (ii) of Definition\;\ref{defn:N-point} $H_t\in\ker(\tnz)\subset\gimShAz$ as well.

Thus $\Lambda_t$ and $\afunc_t$ are shift functions for $H_t$ on some neighbourhood connected neighbourhood $\Wman$ of $z$, whence $\Lambda_t-\afunc_t\in\ker(\Shift_{\WmanS})$, where $\WmanS=\Wman\setminus\{z\}$.

If $z$ is an $\ST$-point, then $\WmanS$ contains non-closed orbits of $\AFld$, whence $\ker(\Shift_{\WmanS})=\{0\}$.
Therefore $\Lambda_t=\afunc_t$ is $\Cinf$ on $\WmanS$.
Thus if we put $\Lambda_t(z)=\afunc_t(z)$ then $\Lambda_{t}$ will become $\Cinf$ on all of $\Vman$.

Suppose that $z$ is a $\PT$-point.
Then $\ker(\Shift_{\WmanS})=\{n\theta\}_{n\in\ZZZ}$, so $\Lambda_t-\afunc_t=n\theta$ for some $n\in\ZZZ$.
But by Definition\;\ref{defn:P-point} $\theta$ smoothly extends to all of $\Wman$, whence so does $\Lambda_t=\afunc_t+n\theta$.

Suppose that $z$ is an $\NT$-point, so $\lim\limits_{x\to z}\theta(x)=+\infty$.
This implies that the shift map $\ShAW$ is non-periodic, so $\afunc_t$ is a unique shift function for $H_t$ on $\Wman$.
Notice that by assumption $r\geq1$, so the restriction $H:\Wman\times I\to\Uman$ is a $1$-homotopy in $\imShAW$ having a shift function $\Lambda:\WmanS\times I\to\RRR$ such that $\Lambda_0=0$ is $\Cinf$ on $\Wman$.
Then the main result of\;\cite{Maks:ImSh} is applicable to this situation and claims that $\Lambda_t=\afunc_t$ on $\WmanS$ for all other $t\in I$.
Hence $\Lambda_t$ smoothly extends to all of $\Vman$.
\end{proof}

\section{Main results}\label{sect:main_results}
The sets of $\ST$-, $\PT$-, and $\ST$-points of $\func$ will be denoted by $\singfS$, $\singfP$, and $\singfN$ respectively.
The following Theorems~\ref{th:hom-type-StabIdf} and~\ref{th:hom-type-Orbits} extend the results of~\cite{Maks:AGAG:2006, Maks:TrMath:2008}.
\begin{theorem}\label{th:hom-type-StabIdf}{\rm c.f.\;\cite[Th.\;1.3]{Maks:AGAG:2006}.}
\it Suppose that $\func:\Mman\to\Psp$ satisfies axioms {\rm\AxBd} and {\rm\AxSPN}.
Then 
\begin{equation}\label{equ:Stab_Stab1}
\StabIdfr{\infty}=\cdots=\StabIdfr{2}=\StabIdfr{1}.
\end{equation}
Moreover, each of the following conditions implies $\StabIdfr{1}=\StabIdfr{0}$:
\begin{enumerate}
 \item[\rm(a)]
$\func$ has no critical points of type $\NT$, i.e. $\singfN=\varnothing$;
 \item[\rm(b)] 
$\Mman$ is a $2$-disk, $\singf$ consists of a unique critical point $z$ which is of type $\NT$ with $n_z=1$;
 \item[\rm(c)] 
$\Mman$ is a $2$-sphere, $\singf$ consists of exactly two critical points $z'$ and $z$ such that $z'$ is of type $\PT$, and $z$ is of type $\NT$ with $n_{z}=1$.
\end{enumerate}

The space $\StabIdf:=\StabIdfr{\infty}$ is contractible if and only if one of the following conditions holds true:
\begin{itemize}
\item
$\func$ has at least one critical point of type $\ST$ or $\NT$;
\item
$\Mman$ is non-orientable.
\end{itemize}
In all other cases $\StabIdf$ is homotopy equivalent to $S^1$.
\end{theorem}

\begin{theorem}\label{th:hom-type-Orbits}{\rm c.f.\;\cite[Th.\;1.5]{Maks:AGAG:2006}, \cite{Maks:TrMath:2008}.}
Suppose that $\func:\Mman\to\Psp$ satisfies axioms \AxBd-\AxFibr\ and $\func$ has at least $\ST$-point.
Let $n$ be the total number of critical points of $\func$.
Then $\Orbff$ is weakly homotopy equivalent to a CW-complex of dimension $\leq 2n-1$.
Moreover, $\pi_i\Orbff=\pi_i\Mman$ for $i\geq 3$, $\pi_2\Orbff=0$, 
and for $\pi_1\Orbff$ we have the following exact sequence:
\begin{equation}\label{equ:pi1_Of}
1 \to \pi_1\DiffM \oplus \ZZZ^{\rankpiO} \to \pi_1\Orbff \to  \grp \to 1,
\end{equation}
where $\grp$ is a certain finite group and $\rankpiO\geq0$.
\end{theorem}

\subsection{Discussion of axioms}\label{sect:discussion_of_axioms}
Axioms \AxBd\ and \AxSPN\ put restrictions only on the foliation $\partf$ of $\func$.
For instance suppose that $\func:\Mman\to\RRR$ satisfies them and has a global minimum $\min\func=0$.
Then for every $n\geq2$ the function $\func^n$ also satisfies these axioms.
Replacing $\func$ by its $n$-th degree makes that global minimum of $\func$ ``more degenerate'' but does not changes the foliation $\partf$.

On the other hand, if $\func$ satisfies \AxFibr, then usually the orbit $\Orbit(\func^n)$ of $\func^n$ for $n\geq2$ has \myemph{infinite codimension} in $\smr$ and the verification of \AxFibr{} for $\func^n$ is not an easy problem, see~\cite{Poenaru:PMIHES:1970}.

\subsection{Sufficient condition for \AxFibr.}
Let $\func:\RRR^2\to\RRR$ be a smooth function, $z\in\RRR^2$, and $\func(z)=0$.
Denote by $\smrz$ the algebra of germs of smooth functions $\RRR^2\to\RRR$.
Then the \myemph{Jacobi ideal} of $\func$ at $z$ is the ideal $\Delta(\func,z)$ in $\smrz$ generated by germs partial derivatives $\func'_{x}(z)$ and $\func'_{y}(z)$ of $\func$ at $z$.

The \myemph{codimension} $\mu_{\RRR}(\func,z)$ of a $\func$ at $z$ (or a \myemph{real Milnor number} of $z$) is the codimension of the Jacobi ideal $\Delta(\func,z)$ in $\smrz$:
$$
\mu_{\RRR}(\func,z)= \dim_{\RRR} [ \smrz /\Delta(\func,z) ],
$$
see\;\cite{Sergeraert:ASENS:1972}.
\begin{lemma}\label{lm:suffcond_AxFibr}{\rm c.f.\;\cite{Sergeraert:ASENS:1972}, \cite[\S\;11.30]{Maks:AGAG:2006}.}
Let $\Cinf_{\partial}(\Mman,\Psp)$ be the subspace of $\Ci{\Mman}{\RRR}$ consisting of maps satisfying axiom \AxBd, and let $\func\in\Cinf_{\partial}(\Mman,\Psp)$.
Suppose that $\mu_{\RRR}(\func,z)<\infty$ for each critical point $z\in\singf$.
Then $\Orbf$ is a Fr\'echet submanifold of $\Cinf_{\partial}(\Mman,\Psp)$ of finite codimension and the map $p:\DiffM\to\Orbf$ is a principal locally trivial $\Stabf$-fibration.
In particular, $p$ is a Serre fibration, so $\func$ satisfies \AxFibr.
\end{lemma}

\begin{lemma}\label{exmp:hompoly-wmf}{\rm c.f.~\cite{Maks:DNANU:2009}}.
Suppose that $\func:\Mman\to\Psp$ satisfies \AxBd\ and has the following property:
for every critical point $z$ of $\func$ there exists a local presentation $\func_{z}:\RRR^2\to\RRR$ of $\func$ in which $z=(0,0)$ and $\func_{z}$ is a homogeneous polynomial without multiple factors.
Then $\func$ satisfies all other axioms \AxSPN, \AxFibr.
\end{lemma}
\begin{proof}
Axiom \AxSPN\ follows from results of\;\cite{Maks:hamv2, Maks:CEJM:2009}, see also\;\cite{Maks:loc-inv-shifts}.

For the axiom \AxFibr\ it suffices establish finiteness of $\mu_{\RRR}(\func_{z},z)$.
Regard $\func_{z}$ as a polynomial two complex variables with real coefficients.
Since $\func_z$ has no multiple factors, it follows that $z=(0,0)$ is an isolated critical point of $\func_z$, whence the \myemph{complex Milnor number} defined analogously with respect to the algebra of germs of analytical functions $(\CCC^2,0)\to\CCC$ is finite: $\mu_{\CCC}(\func_z,z)<\infty$, see\;\cite{ArnoldGuseinZadeVarchenko:Sing1}.
It remains to note the following inequaltiy $\mu_{\RRR}(\func_z,z)\leq 2\mu_{\CCC}(\func_z,z)$, which can be easily proved.
\end{proof}

\section{Framings}\label{sect:framings}
In this section we will assume again that $\func$ satisfies conditions (i) and (ii) at the beginning of \S\ref{sect:Foliation-of-f}.
The results of this section will be used for the proof of Theorem\;\ref{th:hom-type-Orbits} only.

\subsection{Framings for $\NT$-points}\label{sect:framed-N-points}
Suppose $z$ is an $\NT$-point of $\func$ and let $\AFld$ be the corresponding special vector field defined on some neighbourhood $\Uman$ of $z$.
Recall that the group $\Ggrpz=\gDpAz/\gimShAz$ is cyclic and we have denoted its order by $n_z$.

For each non-zero tangent vector $v\in T_{z}\Mman$ put
$$
\zvfrm{z}{v} \ := \ \{\pm T_{z}\dif(v) \ | \ \dif\in\gDpAz \}.
$$
Thus $\zvfrm{z}{v}$ is the union of orbits of $v$ and $-v$ under the natural action of $\gDpAz$ on $T_{z}\Mman$.
Equivalently, we may put
$$
\zvfrm{z}{v} \ := \ \{\pm \xi(v) \ | \ \xi\in\tnz(\gDpAz) \}.
$$

\begin{definition}
The set $\zvfrm{z}{v}$ will be called a \myemph{framing at $z$} if it is finite and invariant with respect to the whole group $\gDAz$.
\end{definition}

\begin{lemma}\label{lm:framing_prop}
\begin{enumerate}
 \item 
Framings exist.
 \item
Let $\zvfrm{z}{v}$ be any framing.
If $n_z$ is odd (resp. even), then $\zvfrm{z}{v}$ consists of $2n_z$ (resp. $n_z$) elements.
 \item
The kernel of the action of $\gDpAz$ on any framing $\zvfrm{z}{v}$ is $\gimShAz$.
Hence $\gDpAz$ induces a \myemph{free} action of $\Ggrpz$ on $\zvfrm{z}{v}$.
 \item 
Let $\zvfrm{z}{v}$ be a framing at $z$ and $\dif\in\Stabf$.
Then $z':=\dif(z)\in \singfN$ and
$$\zvfrm{z'}{v}\,:=\, T_{z}\dif(\zvfrm{z}{v}) \subset T_{z'}\Mman$$
is a framing at $z'$.
 \item
If $\gdif\in\Stabf$ is such that $\gdif(z)=\dif(z)$, then 
$$T_{z}\gdif(\zvfrm{z}{v}) =T_{z}\dif(\zvfrm{z}{v}).$$
\end{enumerate}
\end{lemma}
\begin{proof}
(1)-(3).
First suppose that $z$ is an $\NNT$-point, so we can choose local coordinates at $z$ in which $\nabla\AFld(z)=\left(\begin{smallmatrix}
0 & 1 \\ 0 & 0
\end{smallmatrix}\right)$.
Then by Lemma\;\ref{lm:N-points} 
$$
\tnz(\gimShAz)=\Agrp_{++},
\qquad 
\tnz(\gDpAz) \subset \Agrp_{+},
\qquad 
\tnz(\gDAz) \subset \Agrp.
$$
Let $v=(a,0)\in T_{z}\Mman$ be any vector with respect to these coordinates such that $a\not=0$.
Evidently, the orbit of $v$ with respect to each of the groups $\Agrp_{+}$ and $\Agrp$ is $\{\pm v\}$, see Figure\;\ref{fig:framings}a), while $\Agrp_{++}$ is the stabilizer of $v$ with respect to $\Agrp_{+}$.
This implies that $\zvfrm{z}{v}=\{\pm v\}$, this set is invariant with respect to $\gDAz$, and $\gimShAz$ is the kernel of the action of $\gDpAz$.

If $\gimShAz=\gDpAz$, i.e. $n_z=1$, then $|\zvfrm{z}{v}|=2=2n_z$.
Otherwise, $\gDpAz/\gimShAz=\ZZZ_2$, so $n_z=2=|\zvfrm{z}{v}|$.

On the other hand, if $w\in T_{z}\Mman$ is another tangent vector which is not collinear to $v$, then $w=(b,c)$ with $c\not=0$ and $\zvfrm{z}{w}=\{(t, \pm c) \ | \ t\in\RRR\}$ is a pair of lines.
Thus $\zvfrm{z}{w}$ is not a framing, see Figure\;\ref{fig:framings}b).

\medskip 

Suppose now that $z$ is an $\NZT$-point, i.e. $\nabla\AFld(z)=0$.
Then by Lemma\;\ref{lm:N-points} $\gimShAz$ is the kernel of $\tnz$, 
$$\tnz(\gDpAz)=\gDpAz/\gimShAz=\Ggrpz$$ is a finite cyclic group of order $n_z$.
Whence for any non-zero $v\in T_{z}\Mman$ the set $\zvfrm{z}{v}$ is finite and $\gimShAz$ is the kernel of the action of $\gDpAz$ on $\zvfrm{z}{v}$.

It remains to choose $v$ for which $\zvfrm{z}{v}$ is invariant with respect to $\gDAz$.
We can assume that $\gDAz\not=\gDpAz$, so $\tnz(\gDAz)$ is a dihedral group.
Let $\dif\in\gDAz\setminus\gDpAz$.
Then $\dif$ changes orientation at $z$, whence the linear map $T_{z}\dif$ is a reflection with respect to a line $\{tv \ | \ t\in\RRR\}$ determined by some non-zero vector $v\in T_{z}\Mman$.
In particular, $T_z\dif(v)=v$.
Now it is evident that $\zvfrm{z}{v}$ is invariant with respect to $\gDAz$, so $\zvfrm{z}{v}$ is a framing, see Figure\;\ref{fig:framings}c).

Statements (4) and (5) are easy and we leave them for the reader.
\end{proof}
\begin{figure}[ht]
\begin{center}
\begin{tabular}{ccccc}
\includegraphics[height=2cm]{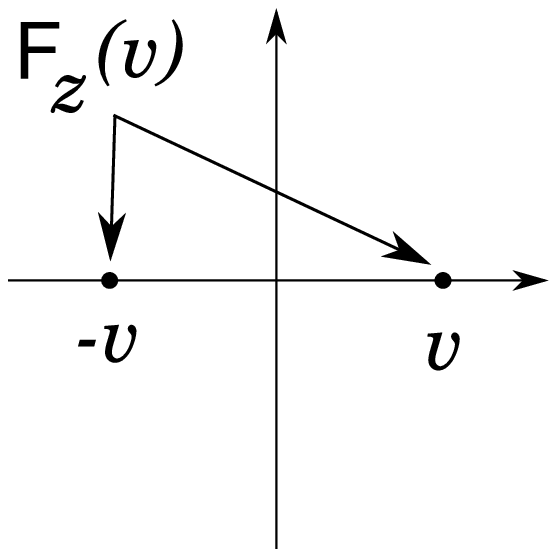} 
& \qquad\qquad &
\includegraphics[height=2cm]{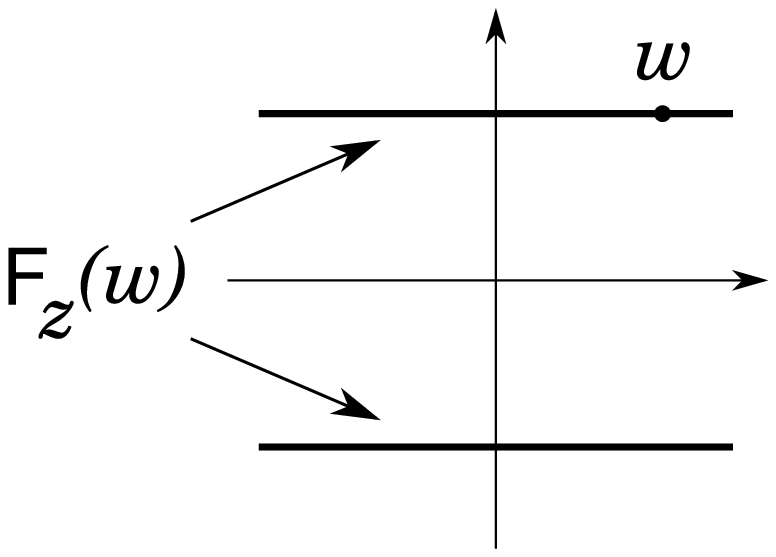} 
& \qquad\qquad &
\includegraphics[height=2cm]{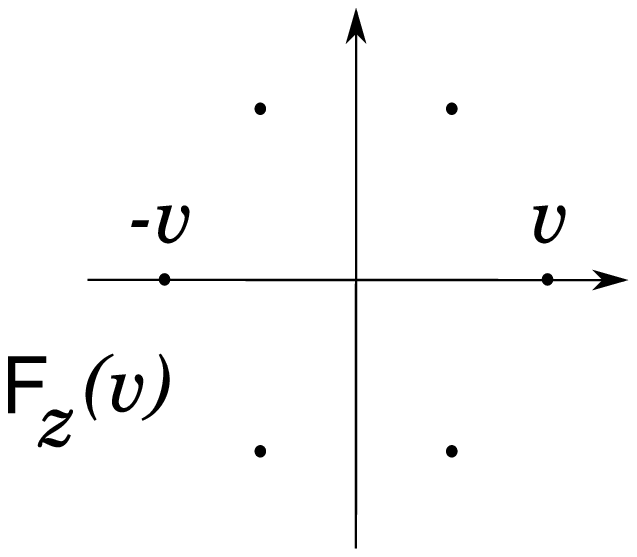} \\
a) framing & &
b) not a framing & & 
c) framing 
\end{tabular}
\end{center}
\caption{}
\protect\label{fig:framings}
\end{figure}

\begin{definition}\label{defn:compat_N_framing}
Say that a family of framings $\{\zvfrm{z}{v_z}\ : \ z\in\singfN\}$ is \myemph{compatible} if it is invariant with respect to $\Stabf$, i.e. 
$$T_{z}\dif(\zfrm{z})=\zfrm{\dif(z)}, \qquad \forall\,\dif\in\Stabf.$$
\end{definition}

\begin{corollary}
Compatible framings exist.
\end{corollary}
\begin{proof}
Since $\dif(\singfN)=\singfN$ for every $\dif\in\Stabf$, we can divide $\singfN$ by orbits with respect to $\Stabf$.
Let $O=\{z_1,\ldots,z_l\} \subset \singfN$ be one of these orbits.
It suffices to define a compatible framing on $O$.

Fix any framing $\zfrm{z_1}$ for $z_1$.
Let $z_i\in O$ and $\dif\in\Stabf$ be such that $\dif(z_1)=z_i$.
Then we set $\zfrm{z_i}=T_{z_1}\dif(\zfrm{z_1})$.
By Lemma~\ref{lm:framing_prop} this definition does not depend on a particular choice of such $\dif$.
\end{proof}

\subsection{Framed \KR-graph $\FReebf$ of $\func$.}\label{sect:framed-KR-graph}
Let $\Reebf$ be the \KR-graph of $\func$, $\prKR:\Mman\to\Reebf$ be the factor map, and $w$ be a vertex of $\Reebf$.
Say that 
\begin{itemize}
 \item 
$w$ is a \myemph{$\partial$-vertex} if $\prKR^{-1}(w)$ is a connected component of $\partial\Mman$;
 \item
$w$ is an \myemph{$\ST$-vertex} if $\prKR^{-1}(w)$ contains $\ST$-points of $\func$ (in this case $\prKR^{-1}(w)$ is a critical component of $\partf$ being not local extreme of $\func$);
 \item
$w$ is \myemph{$\PT$-vertex} (resp. \myemph{$\NT$-vertex}) if $\prKR^{-1}(w)$ is a $\PT$-point (resp. $\NT$-point) of $\func$.
\end{itemize}

Let $w$ be an $\NT$-vertex of $\Reebf$ and $z=\prKR^{-1}(w)$ be the corresponding $\NT$-point of $\func$.
Then by \myemph{cyclic order} of $w$ we will call the corresponding cyclic order $n_z$ of $z$ and denote it by $n_{w}$, thus $n_{w}:=n_{\prKR^{-1}(w)}=n_z$.

Let $k_z$ be the total number of vectors in (any) framing at $z$, see Lemma\;\ref{lm:framing_prop}.
We will denote this number by $k_z$ and $k_{w}$ as well.
Thus $k_z=n_z$ for even $n_z$ and $k_z=2n_z$ for odd $n_z$.

To each $\NT$-vertex $w$ let us glue $k_{w}$ edges $e_{w}(1),\ldots,e_{w}(k_{w})$ as shown in Figure\;\ref{fig:tang_edges}.
We will call them \myemph{edges tangent to $w$}.
The set of tangent edges to $w$ will be denoted by $\tew{w}$.
They should be thought as ``\myemph{lying in the plane orthogonal to the edge incident to $w$}''.
\begin{figure}[ht]
\centerline{\includegraphics[height=2cm]{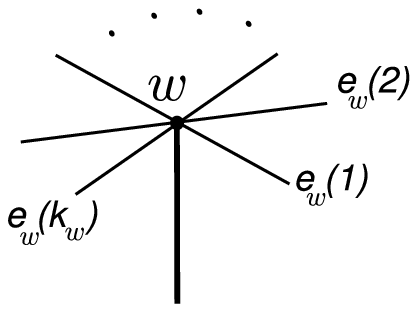}}
\caption{Tangent edges to a vertex $w$}\protect\label{fig:tang_edges}
\end{figure}

Denote the obtained graph by $\FReebf$ and call it the \myemph{framed \KR-graph} of $\func$.
Thus $\Reebf$ is a subgraph of $\FReebf$ and we can extend the KR-function $\KRfunc:\Reebf\to\Psp$ to all of $\FReebf$ to be constant on tangent edges:
$\KRfunc(e_{w}(i))=\KRfunc(w)$.

\subsection{Automorphisms of $\FReebf$}\label{sect:Aut_framed_KR_graph}
Let $E(\Reebf)$ be the set of edges of $\Reebf$.
By an \myemph{automorphism} of $\FReebf$ we will mean a pair $(\nu, o)$, where 
\begin{itemize}
 \item 
$\nu:\FReebf\to\FReebf$ is a homeomorphism which preserves types of vertexes and the \KR-function i.e. $\KRfunc\circ\nu=\KRfunc:\FReebf\to\Psp$, (in particular, $\nu$ sends tangent edges to tangent edges), and
 \item
$o:E(\Reebf)\to\ZZZ_2$ is any function.
\end{itemize}

Define the composition of automorphisms by:
$$(\nu_1, \, o_1)\circ (\nu_2, \, o_2) \ = \ (\nu_1\circ \nu_2, \, o_1\circ\nu_2 \;\cdot\; o_2),$$
where $\cdot$ is a multiplication in $\ZZZ_2=\{\pm1\}$.

Then it is easy to see that all the automorphisms of $\FReebf$ constitute a group.
We will denote this group by $\AutfFR$.
The unit of $\AutfFR$ is $(\id_{\FReebf},1)$, where $1:E(\FReebf)\to\ZZZ_2$ is a constant map to $1 \in\ZZZ_2$, and the inverse of $(\nu, \, o)$ is $(\nu^{-1}, -o\circ\nu^{-1})$.

\subsection{Action of $\Stabf$ on $\FReebf$}\label{sect:action_Stabf_on_FKR}
Suppose $\func$ has at least one critical point.
We will now define a certain homomorphism 
$$
  \stongr:\Stabf\to\AutfFR.
$$
Fix 
\begin{itemize}
 \item 
a compatible framing $\{\zfrm{z}\ : \ z\in\singfN\}$ for $\NT$-points of $\func$,
 \item
for each $z\in\singfN$ a bijection $\xi_{z}:\zfrm{z}\to\tew{\prKR(z)}$ between the framing $\zfrm{z}$ at $z$ and the set of tangent edges at the corresponding vertex $\prKR(z)$ of $\FReebf$ (see Figure\;\ref{fig:frm-tng}, and
 \item
orientation of connected components of $\partfreg$ (this is possible since $\singf\not=\varnothing$, so all regular components of $\partf$ are diffeomorphic to $S^1\times(0,1)$ and thus they are orientable surfaces).
\end{itemize}

\begin{figure}[ht]
\centerline{\includegraphics[height=2cm]{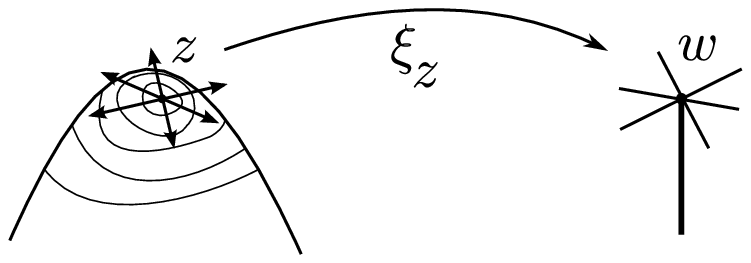}}
\caption{}\protect\label{fig:frm-tng}
\end{figure}

Now let $\dif\in\Stabf$.
We have to associate to $\dif$ a pair $(\nu,o)$.

Notice that $\dif$ yields a certain automorphism $\nu=\stonkr(\dif)$ of $\Reebf$.
To extend it to the set of tangent edges.
Let $\dif\in\Stabf$, $z\in\singfN$, $z'=\dif(z)$, $w=\prKR(z)$ and $w'=\prKR(z')$ be the corresponding vertexes on $\FReebf$.
Then $z'\in\singfN$ as well and $\dif$ yields a bijection between the framings $\zfrm{z}$ and $\zfrm{z'}$.
So we define $\nu:\tew{\omega}\to\tew{\omega'}$ by 
$$\xi_{z'}\circ T_{z}\dif \circ \xi_{z}^{-1}:
\tew{\omega}
\xrightarrow{~\xi_{z}^{-1}~}
\zfrm{z} 
\xrightarrow{~T_{z}\dif~}
\zfrm{z'} 
\xrightarrow{~\xi_{z'}}
\tew{\omega'}.
$$

It remains to define a function $o:E(\Reebf)\to\ZZZ_2$.
Let $e$ be an edge of $\Reebf$ and $\gamma=\prKR(e)$ be the corresponding connected component of $\partfreg$.
Then $\gamma'=\dif(\gamma)$ is also a connected component of $\partfreg$.
Notice that we fixed orientations of $\gamma$ and $\gamma'$.
Therefore we define $o(e)=+1$ if $\dif:\gamma\to\gamma'$ preserves orientations and $o(e)=-1$ otherwise.

A direct verification shows that the correspondence $\dif\mapsto(\nu,o)$ is a well-defined homomorphism $\stongr:\Stabf\to\AutfFR$.

\subsection{The kernel of $\stongr$ and the group $\Dpn$}
Denote by $\DiffMn$ the subgroup of $\DiffM$ consisting of all diffeomorphisms $\dif$ of $\Mman$ such that 
\begin{itemize}
 \item
$\dif(\singf)=\singf$;
 \item
$\dif\in\gimShAzz$ for each $\NT$-point $z\in\singfN$ and (any) special vector field $\AFld_{z}$ for $z$.
In other words, $\tnz(\dif)\in\Agrp_{++}$ if $\nabla\AFld(z)=\left(\begin{smallmatrix}
0 & 1 \\ 0 & 0
\end{smallmatrix}\right)$, and $\tnz(\dif)=\id$ if $\nabla\AFld(z)=0$.
\end{itemize}
Let $\Orbfn$ the orbit of $\func$ with respect to the action of $\DiffMn$.
Put 
$$
\Dpn = \Dpartfpl \cap \DiffMn.
$$

Then it is easy to see that $\Dpn$ is a normal subgroup of $\Stabf$ and 
$$\Dpn \ \subset \ \ker(\stongr).$$
However in general $\Dpn \not= \ker(\stongr)$.
The difference is that each $\dif\in\Dpn$ is required to preserve \myemph{all $1$-dimensional} leaves and their orientations, while each $\gdif\in\ker(\stongr)$ must only preserve \myemph{all regular} leaves and their orientation.
In fact, it is easy to construct an example of $\func$ and $\dif\in\ker(\stongr)\setminus\Dpn$ which interchanges critical leaves of $\partf$, see e.g.\;\cite[Lm.\;6.9 \& Fig.\;6.1]{Maks:CMH:2005}.

The following lemma repairs\;\cite[Eq.\;(8.6)]{Maks:AGAG:2006}, see Remark\;\ref{rem:corrections}.
\begin{lemma}\label{lm:Dpn_and_ketrstong_cap_DiffIdM}
$\Dpn\cap\DiffIdM = \ker(\stongr) \cap \DiffIdM$.
\end{lemma}
\begin{proof}
Let $\dif\in\ker(\stongr) \cap \DiffIdM$.
Then, $\dif$ preserves all regular leaves of $\partf$ with their orientation and is isotopic to $\id_{\Mman}$.
Now by \cite[Th.\;7.1]{Maks:AGAG:2006} $\dif$ also preserves all critical leaves with their orientation.
Moreover, $\dif\in\gimShAz$ for each $z\in\singfN$, whence $\dif\in\Dpn$.
\end{proof}

\begin{lemma}\label{lm:intersections_with_DiffIdMn}
The intersections of the identity component $\DiffIdMn$ of $\DiffMn$ with each of $\Dpn$, $\ker(\stongr)$, and $\Stabf$ coincide:
$$
\Dpn\cap\DiffIdMn = \ker(\stongr)\cap\DiffIdMn = \Stabf\cap\DiffIdMn.
$$
\end{lemma}
\begin{proof}
The relation $\Dpn\cap\DiffIdMn=\ker(\stongr)\cap\DiffIdMn$ follows from Lemma\;\ref{lm:Dpn_and_ketrstong_cap_DiffIdM}.
For the second equality it suffices to establish that 
$$\ker(\stongr)\cap\DiffIdMn \ \ \supset \ \ \Stabf\cap\DiffIdMn.$$

Let $\dif\in\Stabf\cap\DiffIdMn$ and $\stonkr(\dif)$ be the automorphism of $\Reebf$ induced by $\dif$.
We have to show that $\stonkr(\dif)=\id_{\Reebf}$ and that $\dif$ preserves orientation of all regular leaves on $\partf$.

First we show that $\stonkr(\dif)=\id_{\Reebf}$.
The arguments are similar to the proof of\;\cite[Eq.\;(8.8)]{Maks:AGAG:2006}.
Notice that $\dif$ yields the identity automorphism of the first homology group $H_1(\Reebf)$.
Indeed, there is a map $s_{\func}:\Reebf\to\Mman$ such that $s_{\func}\circ \prKR=\id_{\Reebf}$, see e.g.\;\cite{Kudryavtseva:MatSb:1999}.
So we have the following commutative diagram:
$$
\xymatrix{
\Reebf \ar[r]_{s_{\func}} \ar@/^/[rr]^{\id_{\Reebf}} \ar[d]_{\stonkr(\dif)} & \Mman \ar[d]^{\dif} \ar[r]_{\prKR} & \Reebf \ar[d]^{\stonkr(\dif)} \\
\Reebf \ar[r]^{s_{\func}} \ar@/_/[rr]_{\id_{\Reebf}} & \Mman \ar[r]^{\prKR} & \Reebf
}
$$
Then for the first homology groups we have another commutative diagram in which rows are exact:
$$
\begin{CD}
0 @>{}>> H_1(\Reebf) @>{s_{\func}}>> H_1(\Mman,\singf) @>{\prKR}>> H_1(\Reebf) @>{}>> 0\\
& & @V{\stonkr(\dif)_{1}}VV @VV{\dif_{1}=\id}V @VV{\stonkr(\dif)_{1}}V \\
0 @>{}>> H_1(\Reebf) @>{s_{\func}}>> H_1(\Mman,\singf) @>{\prKR}>> H_1(\Reebf) @>{}>> 0
\end{CD}
$$
Since $\dif$ is isotopic to $\id_{\Mman}$ relatively $\singf$, we obtain that $\dif_1=\id_{H_1(\Mman,\singf)}$ whence $\stonkr(\dif)_{1}=\id_{H_1(\Reebf)}$, and $\dif$ preserves vertexes of the \KR-graph $\Reebf$ of $\func$.
This implies that $\stonkr(\dif)=\id_{\Reebf}$.

\smallskip 

It remains to show that \myemph{$\dif$ preserves orientation of regular leaves of $\partf$}.
Indeed let $z$ be any $\NT$-point.
Then $\dif$ is isotopic to $\id_{\Mman}$ via an isotopy fixed at $z$, whence $\dif$ preserves orientation at $z$ and therefore it also preserves orientations of all leaves in a neighbourhood of $z$.
If $\Mman$ is orientable, the same is true for all $1$-dimensional leaves of $\partf$ due to\;\cite[Lm.\;3.5(1)]{Maks:AGAG:2006}.

Suppose that $\Mman$ is non-orientable.
Then $\dif$ lifts to a unique symmetric orientation preserving diffeomorphism $\tdif$ of $\tMman$, that is $p\circ\tdif=\dif\circ p$.
It easily follows that $\tdif\in\Stabtf\cap\DiffIdtMn$.
Indeed, 
$$\tfunc\circ\tdif=\func\circ p\circ\tdif=\func\circ \dif\circ p =\func \circ p = \tfunc,$$
so $\dif\in\Stabtf$.
Moreover, any isotopy $\dif_t:\Mman\to\Mman$ between $\dif_0=\id_{\Mman}$ and $\dif_1=\dif$ in $\DiffIdMn$ lifts to an isotopy $\tdif_t:\Mman\to\Mman$ between $\tdif_0=\id_{\Mman}$ and some $\tdif=\tdif_1$ in $\DiffIdtMn$.
Then $\tdif$ is a desired lifting of $\dif$.
Its uniqueness follows from Lemma\;\ref{lm:lift_symm_diff}. 

Now by the orientable case $\tdif$ preserves orientation of all regular leaves of $\partittf$, whence so does $\dif$ for $\partf$.
\end{proof}

\begin{lemma}\label{lm:DMcr_contr}{\rm c.f.~\cite[Lm.~2.7]{Maks:AGAG:2006}}
If $\func$ has at least one $\ST$-point and satisfies axiom \AxSPN, then $\DiffIdMcr$ and $\DiffIdMn$ are contractible.
\end{lemma}
\begin{proof}
{\it Contractibility of  $\DiffIdMcr$.}
Let $b$ be the total number of connected components of $\partial\Mman$ and $n$ be the total number of critical points of $\func$.
It follows from the description of homotopy types of the groups of diffeomorphisms of compact surfaces, see~\cite{EarleEells:BAMS:1967, EarleSchatz:DG:1970, Gramain:ASENS:1973}, that the group $\DiffIdMcr$ is contractible if and only if $\chi(\Mman)<n+b$.

By the assumption $n\geq 1$.
Therefore if $\chi(\Mman)\leq0< 1 \leq n$, our statement is evident.
Suppose that $\Mman=S^2$ or $\prjplane$.
Then $b=0$.
Moreover, in these cases every smooth map $\func:\Mman\to\Psp$ is null-homotopic, whence $\func$ has also maximum and minimum.
Therefore $n \geq 3 > 2 \geq \chi(\Mman)$.

\medskip 

{\it Contractibility of  $\DiffIdMn$.}
Denote by $\DiffMcrid$ the subgroup of $\DiffMcr$ consisting of all diffeomorphisms $\dif$ such that $\tnz(\dif)=\id$ for each $z\in\singf$.
Denote by $s$, $p$, $a$, and $b$ respectively the total number of critical points of $\func$ of types $\ST$, $\PT$, $\NNT$, and $\NZT$, and put $n=s+p+a+b$.
Then we can order all critical points $z_1,\ldots,z_{n}$ of $\func$ in such a way that the first $s$ points are of type $\ST$, the next $p$ points are of type $\PT$, the next $a$ points are of type $\NNT$, and the last $b$ ones are of type $\NZT$.
For every $z_i\in\singf$ fix local coordinates $(x,y)$ in which $z_i=0$.
Then we have a natural map 
$$
\xi:\DiffMcr\to\prod_{i=1}^{n} \GLR{2},
\qquad
\xi(\dif)=(\tn{z_1}(\dif),\ldots,\tn{z_n}(\dif)).
$$

It is easy to show that $\xi$ is a locally trivial fibration with fiber $\DiffMcrid$.
Notice that $\GLR{2}$ has the homotopy type of $O(2)$ being a disjoint union of two circles.
Moreover, as just proved $\DiffIdMcr$ is contractible.
Then from exact sequence of homotopy groups we obtain $\pi_{i}\DiffMcrid=0$ for $i\geq1$, and thus the boundary homomorphism $\partial_1$ is an isomorphism:
\begin{equation}\label{equ:pi1_prod_GLR2__pi0_DiffMcrid}
0 \to \pi_{1} \prod_{i=1}^{n} \GLR{2} \xrightarrow{~\partial_1~}  \pi_{0}\DiffMcrid \to 0,
\end{equation}
whence $\pi_{0}\DiffMcrid=\ZZZ^{n}$.
Consider the following subset
$$Q:=(\GLR{2})^{s+p} \times (\Agrp_{++})^{a} \times (\id)^{b}  \subset \prod_{i=1}^{n} \GLR{2}.$$
Then $\DiffMn = \xi^{-1}(Q)$, whence the restriction $\xi:\DiffMn\to Q$ is fibration with the same fiber $\DiffMn$.
Since $\Agrp_{++}$ is diffeomorphic to $\RRR$, we see that $Q$ is homotopy equivalent to $(S^1)^{s+p}$.
Therefore from exact sequence of homotopy groups we get $\pi_i\DiffMn =\pi_i Q =0$ for $i\geq2$ and also obtain the following exact sequence:
\begin{equation}\label{equ:pi1_Q__pi0_DiffMcrid}
0=\pi_1\DiffMcr \to \pi_1\DiffMn \to \pi_1 Q  \xrightarrow{~\partial_1~} \pi_0\DiffMcr.
\end{equation}
Evidently, the inclusion $Q\subset\prod_{i=1}^{n}\GLR{2}$ induces a monomorphism on $\pi_1$-groups, whence  by\;\eqref{equ:pi1_prod_GLR2__pi0_DiffMcrid} we have that in\;\eqref{equ:pi1_Q__pi0_DiffMcrid} the map $\partial_1$ is a monomorphism.
Therefore $\pi_1\DiffMn=0$.
Thus all the homotopy groups of $\pi_i\DiffIdMn$ vanish.
\end{proof}

\begin{lemma}\label{lm:partial_1}
If $\func$ satisfies \AxFibr, then 
\begin{enumerate}
\item
$p^{-1}(\Orbfn)=\DiffMn$, whence the map 
$$ 
p:\DiffMn \to \Orbfn, \qquad p(\dif)=\func\circ\dif^{-1}
$$
is a Serre fibration as well;
\item
$\Orbff$ (resp. $\Orbffn$) is the orbit of $\func$ with respect to the action of $\DiffIdM$ (resp. $\DiffIdMn$), {\rm c.f.\;\cite{Maks:TrMath:2008}};
\item
the boundary map $\partial_1:\pi_1\Orbff\to\pi_0\Stabf$ is a homomorphism and the image of the induced map $p_1:\pi_1\DiffM \to \pi_1\Orbf$ is included in the center of $\pi_1\Orbf$, {\rm\cite[Lemma~2.2.]{Maks:AGAG:2006}}.
\end{enumerate}
\end{lemma}

\section{Proof of Theorem\;\ref{th:hom-type-StabIdf}}\label{sect:proof:th:hom-type-StabIdf}
\subsection*{Orientable case}
Suppose $\Mman$ is orientable and $\func$ satisfies axioms \AxBd\ and \AxSPN.
For each $z\in\singf$ let $\AFld_{z}$ be the corresponding special vector field defined on some neighbourhood $\Uman_{z}$ of $z$.
\begin{lemma}\label{lm:Hamilt_vf_and_prop_orient}
There exists a vector field $\AFld$ on $\Mman$ such that 

{\rm(i)} $d\func(\AFld)\equiv 0$ and $\AFld(z)=0$ iff $z\in\singf$.

{\rm(ii)} $\AFld =\pm \AFld_{z}$ near $z$ for every $z\in\singf$.

Moreover, for any vector field $\AFld$ satisfying {\rm(i)} and {\rm(ii)} the following conditions holds true:

{\rm(iii)} The shift map $\ShA:\Ci{\Mman}{\RRR}\to\imShA$ is either a homeomorphism or a $\ZZZ$-covering map with respect to topologies $\Wr{\infty}$.
Hence by Lemma\;\ref{lm:ShAV_open_reformulations} so is the restriction $\ShA|_{\Gamma^{+}}: \Gamma^{+} \to \ShA(\Gamma^{+})$, and both spaces $\imShA$ and $\ShA(\Gamma^{+})$ are either contractible or homotopy equivalent to $S^1$.

{\rm(iv)} $\imShA=\EidAFlow{1}$ and, by Lemma\;\ref{lm:imShAV_EidAVr_implies_for_diff}, $\ShA(\Gamma^{+})=\DidAFlow{1}$.

{\rm(v)}
If $\func$ has no $\NT$-points then $\imShA=\EidAFlow{0}$ and $\ShA(\Gamma^{+})=\DidAFlow{0}$.

{\rm(vi)}
Suppose $\func$ has no $\ST$-points and only one $\NT$-point $z$, though it may have $\PT$-points.
If in addition $n_z=1$, then $\imShA=\EidAFlow{0}$ and $\ShA(\Gamma^{+})=\DidAFlow{0}$.
In this case $\func$ satisfies one of the conditions {\rm(b)} or {\rm(c)} of Theorem~\ref{th:hom-type-StabIdf}.

{\rm(vii)} $\DpartfIdr{r}=\DidAFlow{r}=\StabIdfr{r}$ for all $r=0,\ldots,\infty$.

{\rm(viii)} If $\imShA=\EidAFlow{r}$, then $\Didpnr{r}=\StabIdfr{r}$.
\end{lemma}

Suppose Lemma\;\ref{lm:Hamilt_vf_and_prop_orient} is proved.
Then by (iv) and (vii) we get
$$
\ShA(\Gamma^{+})=\StabIdfr{\infty}=\cdots = \StabIdfr{1},
$$
which proves\;\eqref{equ:Stab_Stab1}.
Moreover, the case (a) of Theorem\;\ref{th:hom-type-StabIdf} corresponds to (v), while (b) and (c) correspond to (vii) of Lemma\;\ref{lm:Hamilt_vf_and_prop_orient}, and in these cases 
$\ShA(\Gamma^{+})=\StabIdfr{1}=\StabIdfr{0}$.

Further, by (iii) of Lemma\;\ref{lm:Hamilt_vf_and_prop_orient} $\ShA(\Gamma^{+})=\StabIdfr{\infty}$ is either contractible or homotopy equivalent to the circle.

If $\func$ has a critical point $z$ of type $\ST$ or $\NT$, then there exists a neighbourhood $\Vman$ of $z$ such that $\ShAV$ is non-periodic, hence for each $\dif\in\StabIdf$ there may exists at most one $\Cinf$ shift function on $\Vman$.
This implies that $\ShA$ is non-periodic as well, whence $\StabIdf$ is contractible.

On the other hand, suppose all critical points of $\func$ are of type $\PT$.
Then $\func$ has at most two critical points and it is not hard to prove that $\ShA$ is periodic, whence $\StabIdf$ is homotopy equivalent to the circle.
This case is analogous to\;\cite[Th.\;1.9]{Maks:AGAG:2006}.
Orientable case of Theorem\;\ref{th:hom-type-StabIdf} is completed modulo Lemma\;\ref{lm:Hamilt_vf_and_prop_orient}

\begin{proof}[Proof of Lemma\;\ref{lm:Hamilt_vf_and_prop_orient}]
(i), (ii).
Since $\Mman$ is orientable, it has a symplectic structure and we can construct the corresponding Hamiltonian vector field $\AFld'$ of $\func$.
By definition this vector field satisfies (i) and also is parallel to $\AFld_z$ near each $z\in\singf$.
Changing the sign of $\AFld_z$ (if necessary) we may assume that $\AFld'$ and $\AFld_{z}$ has the same directions near $z$.
Then using partition unity technique we can glue $\AFld'$ with all of $\AFld_{z}$ so that the resulting vector field $\AFld$ on $\Mman$ would satisfy (i) and (ii).

\medskip 

(iii).
We have to show that the map $\ShA:\Ci{\Mman}{\RRR}\to\imShA$ is either a homeomorphism or a $\ZZZ$-covering map.
Due to Theorem\;\ref{th:openness_of_shift_maps} it suffices to prove that every regular point $z\in\Mman\setminus\FixA$ satisfies either of the conditions \condznpnrec, \condzpid, and every singular point $z\in\FixA$ of $\AFld$ satisfies \condzUinv, \condzUpbi, and \condBsShAUVop.

Let $z$ be a regular point of $\func$.
Then $z$ is also non-singular for $\AFld$.
Denote by $\omega$ the connected component of a level-set $\func^{-1}\func(z)$ containing $z$.

If $\omega$ contains critical points of $\func$, then the orbit $\orb_z$ of $z$ is either ``an arc connecting two critical points'' or a ``loop at some critical point'' of $\func$, see points $z_1$ and $z_2$ in Figure~\ref{fig:reg_nonrec_pmap}.
In both cases the limit sets of the orbit $\orb_{z}$ of $z$ is finite, whence $z$ is non-recurrent, and thus it satisfies \condznpnrec.

Otherwise $\omega$ contains no critical points and therefore is diffeomorphic to $S^1$.
Hence $\omega$ is a periodic orbit of $\AFld$, see Figure~\ref{fig:reg_nonrec_pmap}.
Then it is easy to see that the first return map of such orbit is the identity, whence $z$ has property \condzpid.

\begin{figure}[ht]
\centerline{\includegraphics[height=3cm]{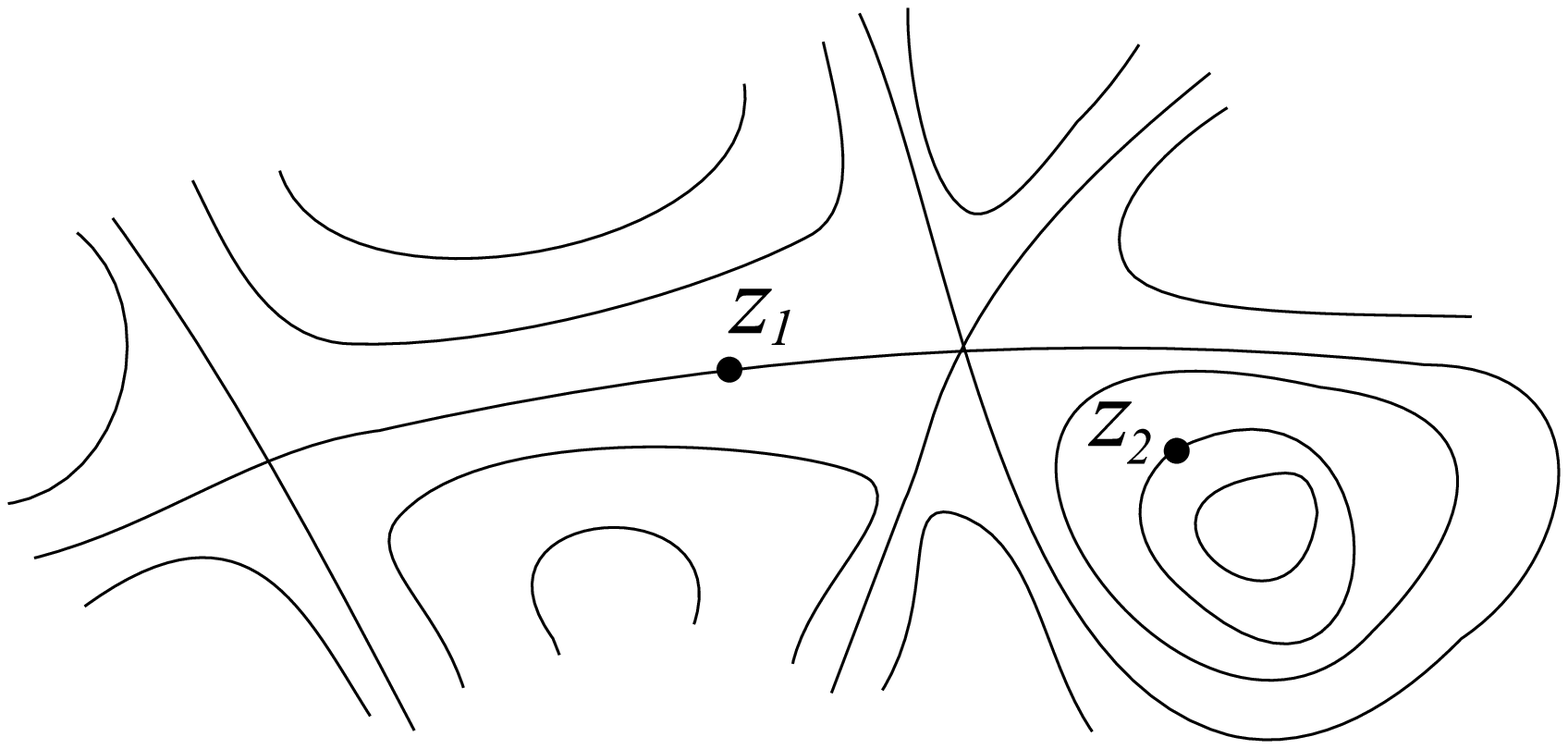}}
\caption{}\protect\label{fig:reg_nonrec_pmap}
\end{figure}

Let $z$ be a critical point of $\func$ and $\Uman$ be a neighbourhood of $z$ on which $\AFld= \AFld_{z}$. 
Then by condition \SPECB\ of Definition\;\ref{def:special-cr-pt} there exists a base $\beta_{z}=\{\Vman_j\}_{j\in J}\subset\Uman$ of $\Dm$-neighbourhoods of $z$ such that for each $\Vman\in\beta_{z}$ the shift map of $\AFld_z$ 
\begin{equation}\label{equ:ShAUV}
\ShAUV:\mathsf{func}(\AFld_z|_{\Uman}, \Vman) \to Sh(\AFld_z|_{\Uman}, \Vman) 
\end{equation}
is a local homeomorphism between the corresponding topologies $\Wr{\infty}$.
Since $\AFld=\AFld_{z}$ on $\Uman$, the map\;\eqref{equ:ShAUV} is the same as the following one:
$$
\ShAUV:\funcAUV \to \imShAUV.
$$
In particular, $\ShAUV$ is open as well.
This implies \condBsShAUVop\ for $z$.

Suppose $z$ is a \myemph{local extreme} of $\func$, then it has arbitrary small $\AFld$-invariant neighbourhoods, see Figure~\ref{fig:Uprime}a), i.e. satisfies \condzUinv.

Suppose $z$ is a \myemph{saddle}.
Then there exists arbitrary small $2$-disk $\Uman$ such that $\partial\Uman$ is smooth and transversal to orbits everywhere except for finitely many points $x_1,\ldots,x_k$, and for each $x_i$ there exists an open arc $\gamma_i$ on $\orb_{x_i}$ containing $x_i$ such that $\gamma_i\cap\partial\Uman=\{x_i\}$, see Figure~\ref{fig:Uprime}b).
This implies property \condzUpbi\ for $z$. 

\begin{figure}[ht]
\begin{center}
\begin{tabular}{ccc}
\includegraphics[height=2cm]{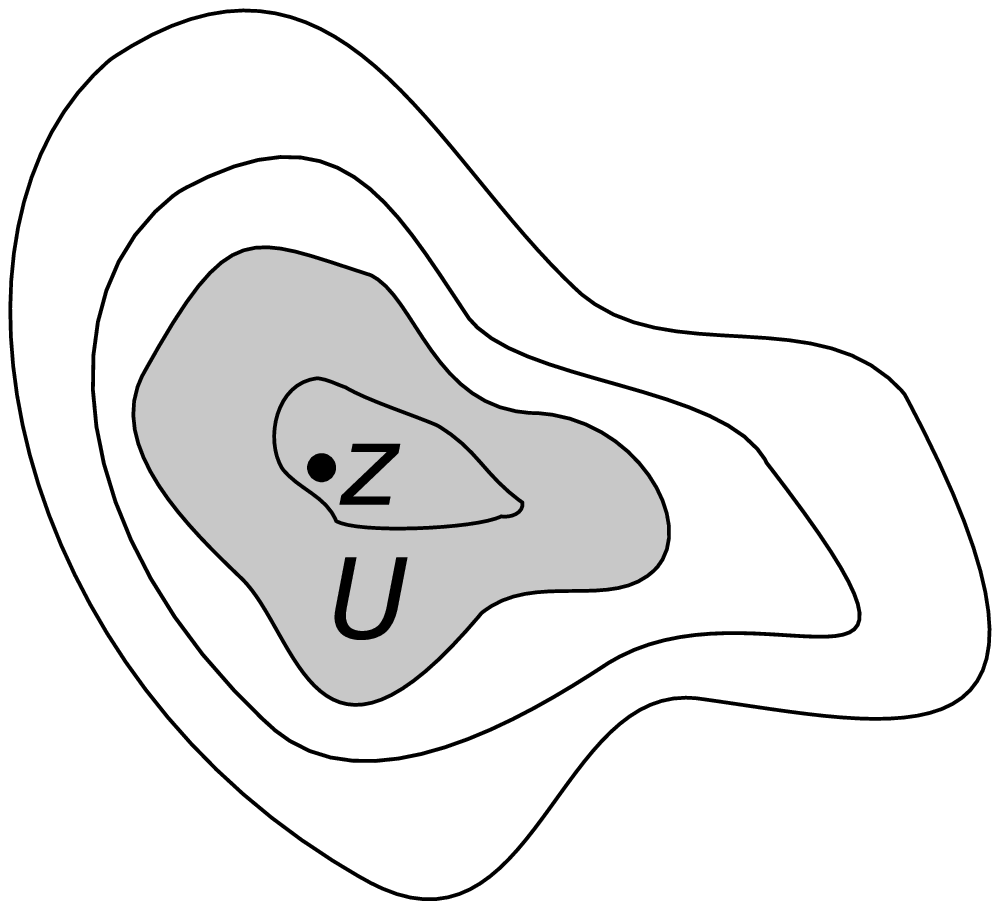} 
& \quad &
\includegraphics[height=2cm]{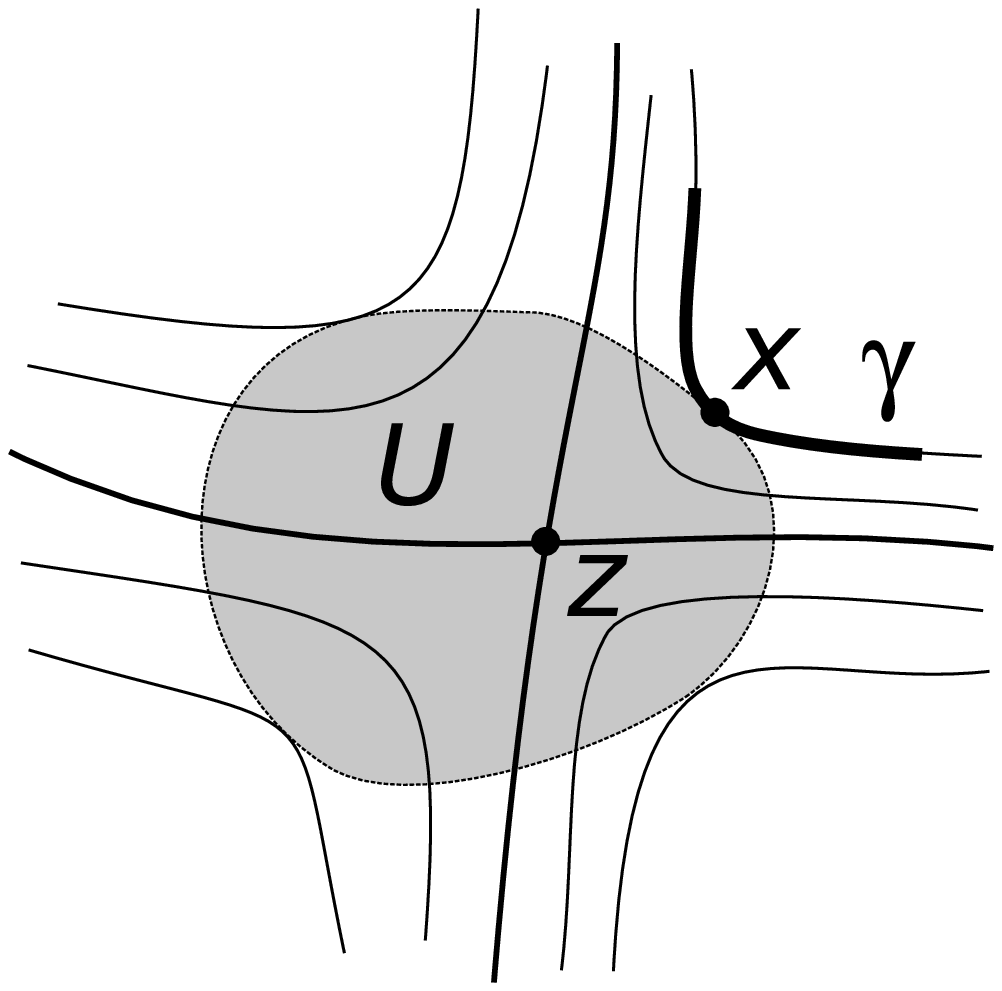} 
\\
 (a) & &
 (b)
\end{tabular}
\caption{A neighbourhood $\Uman$}\protect\label{fig:Uprime}
\end{center}
\end{figure}

(iv), (v). 
We have to identify $\imShA$ with $\EidAFlow{r}$ for $r=0$ or $1$.
Let $\dif\in\EidAFlow{r}$, so there exists an $r$-homotopy $H:\Mman\times I\to\Mman$ such that $H_0=\id_{\Mman}$, $H_1=\dif$, and $H_t\in\EAFlow$ for all $t\in I$.
Then by Lemma\;\ref{lm:shift-func-on-reg} there exists an $r$-homotopy $\Lambda:(\Mman\setminus\FixA) \times I \to \RRR$ such that $\Lambda_{0}\equiv0$ and $H_t(x)=\AFlow(x,\Lambda_t(x))$ for all $x\in\Mman\setminus\FixA$.

If $r=1$ or if $r=0$ but $\func$ has no $\NT$-points, then by Lemma\;\ref{lm:ext_shfunc_under_homotopies} (applied to each $z\in\singf$) the function $\Lambda_t$, $t\in(0,1]$, extends to a $\Cinf$ function on all of $\Mman$.
Hence $H_t\in\imShA$.
In particular, $\dif=H_1\in\imShA$, and so $\imShA=\EidAFlow{1}$.

\medskip 

(vi).
Suppose that $\func$ has no $\ST$-points and only one $\NT$-point $z$ which also satisfies $n_z=1$.
Thus $\func$ has only local extremes and one of them is $z$.
Since $\Mman$ is connected there may exist at most two such points.
Thus either $\singf=\{z,z'\}$, where $z'$ is a $\PT$-point, $\Mman=S^2$, and $\func$ satisfies (b) of Theorem~\ref{th:hom-type-StabIdf}, or $\singf=\{z\}$, $\Mman=D^2$, and $\func$ satisfies (c) of Theorem~\ref{th:hom-type-StabIdf}.

Notice that the orbits of $\AFld$ on $\Mman\setminus\singf$ are closed, the shift map $\Shift_{\Mman\setminus\singf}$ is periodic, and the function $\theta:\Mman\setminus\singf\to(0,+\infty)$ associating to every $x\in\Mman\setminus\singf$ its period $\Per(x)$ generates the kernel $\ker(\Shift_{\Mman\setminus\singf})$.
Moreover, if $\singf=\{z,z'\}$ and $z'$ is a $\PT$-point then by Definition\;\ref{defn:P-point} $\theta$ smoothly extends to some neighbourhood of $z'$.
Thus we can assume that \myemph{$\theta$ is $\Cinf$ on $\Mman\setminus\{z\}$.}

Now let $\dif\in\EidAFlow{0}$.
We have to verify that $\dif\in\imShA$.

By Lemma~\ref{lm:shift-func-on-reg} there exists a smooth shift function $\afunc:\Mman\setminus\singf\to\RRR$ for $\dif$ on $\Mman\setminus\singf$.
Moreover, if $\singf=\{z,z'\}$ where $z'$ is a $\PT$-point, then by Lemma\;\ref{lm:ext_shfunc_under_homotopies} $\afunc$ smoothly extends to a $\Cinf$ function near $z'$.
Therefore we can assume that $\afunc$ is $\Cinf$ on $\Mman\setminus\{z\}$ as well as $\theta$.
Then for each $k\in\ZZZ$ the function $\afunc+k\theta$ is also a $\Cinf$ shift function for $\dif$ on $\Mman\setminus\{z\}$.

By definition $\dif$ is a local diffeomorphism at $z$ and since $\dif\in\EidAFlow{0}$ it follows that $\dif$ preserves orientation at $z$, so $\dif\in\gDpAz$.
Then the assumption $n_z=1$, means that $\tnz(\gDpAz)=\id$, so
$$\dif \ \in \ \gDpAz \ \subset \ \ker(\tnz) \ \subset  \ \gimShAz.$$
Hence there exists a $\Cinf$ shift function $\bfunc$ for $\dif$ on some neighbourhood $\Wman$ of $z$.

Therefore $\afunc$ and $\bfunc$ are $\Cinf$ shift functions for $\dif$ on $\Wman\setminus\{z\}$, whence $\bfunc-\afunc=n\theta$ for some $k\in\ZZZ$.
Hence $\afunc+n\theta$ is $\Cinf$ shift function for $\dif$ on all of $\Mman$, so $\dif\in\imShA$.

\medskip 

(vii).
It follows from (i) that the foliation $\partf$ coincides with the foliation by orbits of $\AFld$, whence
$\Dpartf=\DAFlow$, and by Lemma\;\ref{lm:Dpartf_Stab_identity_comp}
$\DpartfIdr{r}=\DidAFlow{r}\subset\StabIdfr{r}$.

\medskip 

(viii).
Suppose $\imShA=\EidAFlow{r}$ for some $r\geq0$.
Then
$$ \imShA \ = \ \EidAFlow{r} \ \supset \ \DidAFlow{r} \ = \ \StabIdfr{r} \ = \ \DpartfIdr{r} \ \supset \ \Didpnr{r}.$$ 

In particular, each $\dif\in\StabIdfr{r}$ has a $\Cinf$ shift function on $\Mman$ with respect to $\AFld$.
This implies that $\tnz(\dif)\in\gimShAz$ for each $\ST$-point $z$ of $\func$, whence by definition
$\dif\in\Dpn$.
Thus $\StabIdfr{r}\subset\Dpn$, and therefore $\StabIdfr{r}\subset\Didpnr{r}$.
\end{proof}

\subsection*{Non-orientable case}

Suppose $\Mman$ is non-orientable.
Let $p:\tMman\to\Mman$ be the oriented double covering of $\Mman$ and $\xi:\tMman\to\tMman$ be a $\Cinf$ involution generating the group $\ZZZ_2$ of deck transformations.

A vector field $\AFld$ on $\tMman$ tangent to $\partial\tMman$ (as well as its flow $\AFlow$) will be called \myemph{skew-symmetric} if $\xi^{*}(\AFld)=-\AFld$, i.e., $\AFld \circ \xi=-T\xi \circ \AFld$.
This is equivalent to the requirement that $\AFlow_{t} \circ \xi=\xi\circ\AFlow_{-t}$ for all $t\in\RRR$.

A smooth map $\tdif:\tMman\to\tMman$ will be called \myemph{symmetric\/} if $\tdif\circ\xi = \xi \circ \tdif$.
Denote by $\tDiffptM$ the group of all orientation preserving symmetric diffeomorphisms of $\tMman$.

\begin{lemma}\label{lm:lift_symm_diff}
Every $\dif\in\DiffM$ has a unique lifting $\tdif\in\tDiffptM$, so $p\circ\tdif=\dif\circ p$.
Moreover, the correspondence $\dif\mapsto\tdif$ is a homeomorphism $\eta:\DiffM\to\tDiffptM$ with respect to each topology $\Wr{r}$, $(0\leq r\leq \infty)$.
\end{lemma}
\begin{proof}
Notice that each $\dif\in\DiffM$ has exactly two symmetric liftings.
If $\tdif\in\tDifftM$ is a lifting of $\dif$, then another one is given by $\xi\circ\tdif$.
Since $\xi$ reverses orientation of $\tMman$, we can assume that $\tdif$ preserves orientation of $\tMman$, i.e. $\tdif\in\tDiffptM$.
Then it is easy to see that the correspondence $\eta:\dif\mapsto\tdif$ is a bijection between $\DiffM$ and $\tDiffptM$.
The verification that $\eta$ is a homeomorphism with respect to the topology $\Wr{r}$ for each $0\leq r\leq \infty$ is direct and we left it for the reader.
\end{proof}

Suppose $\func$ satisfies axioms \AxBd\ and \AxSPN.
Since $p$ is a local diffeomorphism, it follows that the map $\tfunc=\func\circ p:\tMman\to\Psp$ also satifies these axioms.

Let $y\in\singf$,  $\BFld_y$ be a special vector field near $y$, and $z,z'\in\singtf$ be critical points of $\tfunc$ such that $p^{-1}(y)=\{z,z'\}$.
Define vector fields $\AFld_z$ and $\AFld_{z'}$ near $z$ and $z'$ respectively as pullbacks of $\BFld_y$ via $p$:
$$
\AFld_z=p^{*}\BFld_y, \qquad 
\AFld_{z'}=-p^{*}\BFld_y.
$$
Let $\AFld'$ be any vector field on $\tMman$ satisfying (i) and (ii).
Then the vector field $\AFld = \tfrac{1}{2}(\AFld' -\xi^{*}\AFld')$ is skew-symmetric and also satisfies (i) and (ii) and therefore all other statements of Lemma\;\ref{lm:Hamilt_vf_and_prop_orient}, see\;\cite[Lm.\;5.1]{Maks:AGAG:2006} for details.

\begin{lemma}\label{lm:Hamilt_vf_and_prop_nonorient}{\rm c.f. \cite[Lm.\;4.9\;\&\;5.1]{Maks:AGAG:2006}.}
The following conditions hold true:

{\rm(ix)}
Let $\tDiff(\partittf)=\Diff(\partittf)\cap\tDiffptM$ be the group of symmetric diffeomorphisms preserving foliation $\partittf$.
Then $\eta(\Diff(\partf))=\tDiff(\partittf)$, see Lemma\;\ref{lm:lift_symm_diff}.
In particular, for all $r=0,\ldots,\infty$ we have homeomorphisms $\DiffId(\partf)^{r}\cong\tDiffId(\partittf)^{r}$

{\rm(x)}
Put $E_0=\{\afunc\in\Ci{\tMman}{\RRR} \ | \ \afunc\circ\xi=-\afunc\}.$
Then the shift map $\ShA$ of $\AFld$ yields a homeomorhism $\ShA: E_0\cap\Gamma^{+} \to \tDiff(\partittf)$ with respect to topologies $\Wr{\infty}$.
Since $E_0\cap\Gamma^{+}$ is convex, $\tDiff(\partittf)$ is contractible.
\end{lemma}
\begin{proof}
(ix) follows from Lemma\;\ref{lm:lift_symm_diff}, and (x) from\;\cite[Lm.\;4.9]{Maks:AGAG:2006}.
\end{proof}

Now we can complete non-orientable case of Theorem\;\ref{th:hom-type-StabIdf}.
By (x) of Lemma\;\ref{lm:Hamilt_vf_and_prop_nonorient} 
$\ShA(\Gamma^{+}\cap E_0) = \tDiffId(\partittf)^{1}$, and 
$\ShA(\Gamma^{+}\cap E_0) = \tDiffId(\partittf)^{0}$ if $\func$ and therefore $\tfunc$ have no $\NT$-points.
Moreover for all $r=0,\ldots,\infty$ we have the following identifications
$$
\tDiffId(\partittf)^{r}
 \ \stackrel{\text{Lm.\;\ref{lm:lift_symm_diff}}}{\cong} \ 
\DiffId(\partf)^{r}
\ \ \stackrel{\text{Lm.\;\ref{lm:Dpartf_Stab_identity_comp}}}{=\!=\!=\!=\!=} \ \
\StabIdfr{r}.
$$
This implies that $\StabIdfr{\infty}=\cdots=\StabIdfr{1}$ and this space is contractible.
Moreover, $\StabIdfr{1}=\StabIdfr{0}$ whenever $\func$ has no $\NT$-points.

Theorem\;\ref{th:hom-type-StabIdf} is proved.
\qed

\section{Proof of Theorem\;\ref{th:hom-type-Orbits}} \label{sect:proof:th:hom-type-Orbits}
Suppose $\func:\Mman\to\Psp$ satisfies \AxBd-\AxFibr\ and has at least one $\ST$-point.
Statement about higher homotopy groups of $\Orbf$ is a simple consequence of Theorem\;\ref{th:hom-type-StabIdf} and \AxFibr.

Indeed, choose $\id_{\Mman}$ to be a base point in $\Stabf$ and $\DiffM$, and $\func$ to be a base point in $\Orbf$.
Then axiom \AxFibr\ implies that there exists an exact sequence of homotopy groups of the fibration $p$:
$$
\cdots \to \;
\pi_k \Stabf
\;\xrightarrow{~i_k~}\;
\pi_k \DiffM
\;\xrightarrow{~p_k~}\;
\pi_k \Orbf
\;\xrightarrow{~\partial_k~}\;
\pi_{k-1}\Stabf
\;\to \cdots,
$$
where $i_{k}$ is induced by the inclusion $i:\Stabf\subset\DiffM$, and $\partial_k$ is the boundary homomorphism.

Recall that $\pi_i\DiffM = \pi_i \Mman$ for $i\geq3$, and $\pi_2\DiffM=0$, see\;\cite{EarleEells:BAMS:1967, EarleEells:DG:1970, EarleSchatz:DG:1970, Gramain:ASENS:1973}.
Since by Theorem\;\ref{th:hom-type-StabIdf} $\StabIdf$ is contractible, we obtain isomorphisms $\pi_k\Orbf\approx\pi_k\DiffM$ for $k\geq2$ and also the following exact sequence:
$$
1 \to \pi_1\DiffM \;\xrightarrow{~p_1~}\;
 \pi_1\Orbf \;\xrightarrow{~\partial_1~}\;
 \pi_0(\Stabf\cap\DiffIdM) \to 1,
$$
where
$\pi_0(\Stabf\cap\DiffIdM)$ can be regarded as the kernel of the induced map
$i_0:\pi_0\Stabf \to \pi_0\DiffM.$

It remains to establish the short exact sequence\;\eqref{equ:pi1_Of} for $\pi_1\Orbf$ and show that $\Orbff$ is weakly homotopy equivalent to some finite dimensional CW-complex.

{\em Proof of\;\eqref{equ:pi1_Of}.}
The arguments are similar to\;\cite[\S9]{Maks:AGAG:2006}.
Recall that we have a homomorphism $\stongr:\pi_0\Stabf\to\AutfFR$.
Let $G = \stongr(\ker i_0)$ be the image of the subgroup $\ker i_0=\pi_0(\Stabf\cap\DiffIdM)$ of $\pi_0\Stabf$ under $\stongr$, and $\JgrpId$ be the kernel of the restriction of $\stongr$ to $\ker i_0$.
Thus 
$$
\JgrpId:=\pi_0(\ker(\stongr)\cap\DiffIdM) 
\ \stackrel{\text{Lemma\;\ref{lm:Dpn_and_ketrstong_cap_DiffIdM}}}{=\!=\!=\!=\!=\!=\!=} \
\pi_0(\Dpn\cap\DiffIdM).
$$
Since $\Jgrp=\pi_0(\Dpn)\approx\ZZZ^{\intk}$ is a free abelian group, we see that so is its subgroup $\JgrpId$.
Thus $\JgrpId\approx\ZZZ^{\intl}$ for some $\intl\geq0$.

Then we have the following commutative diagram in which all vertical and horizontal lines are exact:
$$
\begin{array}{ccccccccc}
  &     &                &     &  1         &     &  1 \\ 
  &     &                &     & \downarrow &     & \downarrow \\ 
1 & \to & \pi_1\DiffIdM  & \xrightarrow{~p~} & \hgrp          & \xrightarrow{~\partial_1~} & \JgrpId\approx\ZZZ^{\intl} & \to & 1 \\
  &     &  ||            &     & \downarrow &     &  \downarrow \\ 
1 & \to & \pi_1\DiffIdM  & \xrightarrow{~p~} & \pi_1\Orbf & \xrightarrow{~\partial_1~} & \ker(i_0)        & \to & 1 \\
  &     &                &     & \downarrow &           & \downarrow \\ 
  &     &                &     &  G         &  =\!=\!=  &  G \\ 
  &     &                &     & \downarrow &           & \downarrow \\ 
  &     &                &     &  1         &           &  1 \\ 
\end{array}
$$
where $\hgrp=\partial_1^{-1}(\JgrpId)$.
To show that the left vertical sequence coincides with\;\eqref{equ:pi1_Of} we have to prove the following:
\begin{theorem}\label{th:splitting_pi1Orbf}
A short exact sequence 
\begin{equation}\label{lm:pi1DidM_in_the_center_of_H}
1\to \pi_1\DiffIdM \ \xrightarrow{~~p_1~~} \ \hgrp \xrightarrow{~~\partial_1~~} \JgrpId \to 1
\end{equation}
admits a section $s:\hgrp \to\JgrpId$ such that $\partial_1 \circ s = \id_{\JgrpId}$.
Since this sequence is a \myemph{central} extension, see {\rm(3)} of Lemma\;\ref{lm:partial_1}, it splits, so $\hgrp\approx\pi_1\DiffIdM\oplus\ZZZ^{l}$.
\end{theorem}
This statement was claimed without explanations in the proof of\;\cite[Th.\;1.5]{Maks:AGAG:2006}.
But in general there are central extension that do not split.
Therefore we will present a proof of Theorem\;\ref{th:splitting_pi1Orbf} in\;\S\ref{sect:proof:th:splitting_pi1Orbf}.
Of course, the statement of Theorem\;\ref{th:splitting_pi1Orbf} is non-trivial only for the surfaces with $\pi_1\DiffIdM\not=0$ or (which is equivalent) with $\chi(\Mman)\geq0$.

\medskip

To show finiteness of homotopy dimension of $\Orbff$ we have to calculate the homotopy type of $\Orbffn$.
The statement here is the following theorem
\begin{theorem}\label{th:Stabf_DiffIdMn_StabIdf}{c.f.\;\cite[Eq.\;(8.8)]{Maks:AGAG:2006}.}
If $\func$ satisfies \AxBd\ and \AxSPN, then 
$$\Stabf\;\cap\;\DiffIdMn \ = \ \StabIdf.$$
Hence $\Stabfn\cap\DiffIdMn=\StabIdf$, and thus $$\pi_0\bigl(\Stabfn\cap\DiffIdMn\bigr)=0.$$
\end{theorem}
The proof of this theorem will be given in \S\ref{sect:proof:th:Stabf_DiffIdMn_StabIdf}.

\begin{corollary}
$\pi_i\Orbffn=0$ for all $i\geq 0$.
\end{corollary}
\begin{proof}
By (2) of Lemma\;\ref{lm:partial_1} the map $p:\DiffMn\to\Orbfn$ is a Serre fibration.
Moreover, $\StabIdfn=\StabIdf$ is contractible, and by Lemma\;\ref{lm:DMcr_contr} so is $\DiffIdMn$.
Hence $\pi_i\Orbfn=0$ for all $i\geq2$ and we also obtain an isomorphism
$$
\pi_1\Orbfn 
\ \stackrel{\partial_1}{\approx} \ 
\pi_0(\Stabfn\cap\DiffIdMn)
\ \stackrel{\text{Th.\ref{th:Stabf_DiffIdMn_StabIdf}}}{\approx} \
0,
$$
where $\pi_0(\Stabfn\cap\DiffIdMn)$ is the kernel of the induced homomorphism
$j_0:\pi_0\Stabfn \to \pi_0\DiffMn$ induced by the inclusion map $j:\Stabfn\subset\DiffMn$.
\end{proof}

The proof that $\Orbff$ has the weak homotopy type of a $CW$-complex of dimension $\leq 2n-1$ is similar to\;\cite{Maks:TrMath:2008}.
For the completeness we briefly recall the principal arguments.

Let $n$ be the total number of critical points of $\func$,
$$\mathcal{P}_{n}(\Int\Mman) = \{(x_1,\ldots,x_n) \ | \ x_i\in\Int\Mman, x_i\not=x_j \ \text{for} \ i\not=j  \} \; \subset \; \prod_{n} \Int\Mman,$$
and $P_n$ be the $n$-th symmetrical group of all permutations of $n$ distinct symbols.
Evidently, $P_n$ freely acts on $\mathcal{P}_{n}(\Int\Mman)$ by permutations of coordinates.
The factor space $\mathcal{F}_{n}(\Int\Mman)=\mathcal{P}_{n}(\Int\Mman)/P_n$ is the \myemph{$n$-th configuration space of the interior $\Int\Mman$}, i.e. the space on unordered $n$-tuples of mutually distinct points of $\Int\Mman$.

Then we can define the following map $k:\Orbff \to \mathcal{F}_{n}(\Int\Mman)$ associating to every $g\in\Orbff$ the set $\Sigma_{g}$ of its critical points.
It can be shown similarly to~\cite{Maks:TrMath:2008} that $k$ is a locally trivial fibration such that every connected component of a fiber is homeomorphic with $\Orbffn$ and therefore is weakly contractible.
From the exact sequence of homotopy groups of this fibration we obtain that $k$ is a monomorphism on $\pi_1$ and isomorphism on all $\pi_i$ for $i\geq2$.
Let $\mathcal{F}(\func)$ be the covering space of $\mathcal{F}_{n}(\Int\Mman)$ corresponding to the subgroup $k(\pi_1\Orbff) \subset \pi_1 \mathcal{F}_{n}(\Int\Mman)$.
Then $k$ lifts to the map $\tilde k:\Orbff\to\mathcal{F}(\func)$ which induces isomorphisms between all the corresponding homotopy groups and therefore is a homotopy equivalence.
It remains to note that $\mathcal{F}(\func)$ is a manifold (usually non-compact) and $\dim\mathcal{F}(\func)=2n$, whence it is hootopy equivalent to a $(2n-1)$-dimensional CW-complex.

\medskip
This completes Theorem\;\ref{th:hom-type-Orbits} modulo Theorems\;\ref{th:Stabf_DiffIdMn_StabIdf} and\;\ref{th:splitting_pi1Orbf}.
To prove them we have to calculate the group $\pi_0\Dpn$.
\qed

\section{Group $\pi_0\Dpn$}\label{sect:group_pi0Dpn}
In this section we calculate the group $\pi_0\Dpn$.
The exposition is similar to\;\cite[Th.\;6.2]{Maks:AGAG:2006} but we take to account $\NT$-points.

Let $\FReebf$ be the framed \KR-graph of $\func$.
An edge $e$ of $\FReebf$ will be called \myemph{external}, if it is either tangent to some $\NT$-point, or is incident either to a $\PT$- or to a $\partial$-vertex.
Otherwise, $e$ is \myemph{internal}.

Moreover, an internal edge $e$ will be called an \myemph{$\NT$-edge} (resp. an \myemph{$\ST$-edge}) if at least one of its vertexes is an $\NT$-vertex (resp. both vertexes of $e$ are $\ST$-vertexes).

Suppose that $\func:\Mman\to\Psp$ satisfies \AxBd\ and \AxSPN.
Let $\gamma$ be a regular leaf of $\partf$, so $\gamma$ is homeomorphic to $S^1$.
Denote by $\comp{\gamma}$ the connected component of $\partfreg$ containing $\gamma$.
Then there exists a Dehn twist $\tau_{\gamma}$ along $\gamma$ preserving $\partf$ and being identity outside arbitrary small neighbourhood $\Uman_{\gamma} \subset \comp{\gamma}$ of $\gamma$ consisting of full leaves of $\partf$, see\;\cite[\S6]{Maks:AGAG:2006} and Figure\;\ref{fig:dehn-twist}.
In particular, we have that $\tau_{\gamma}\in\Dpn$.
\begin{figure}[ht]
\includegraphics[height=2cm]{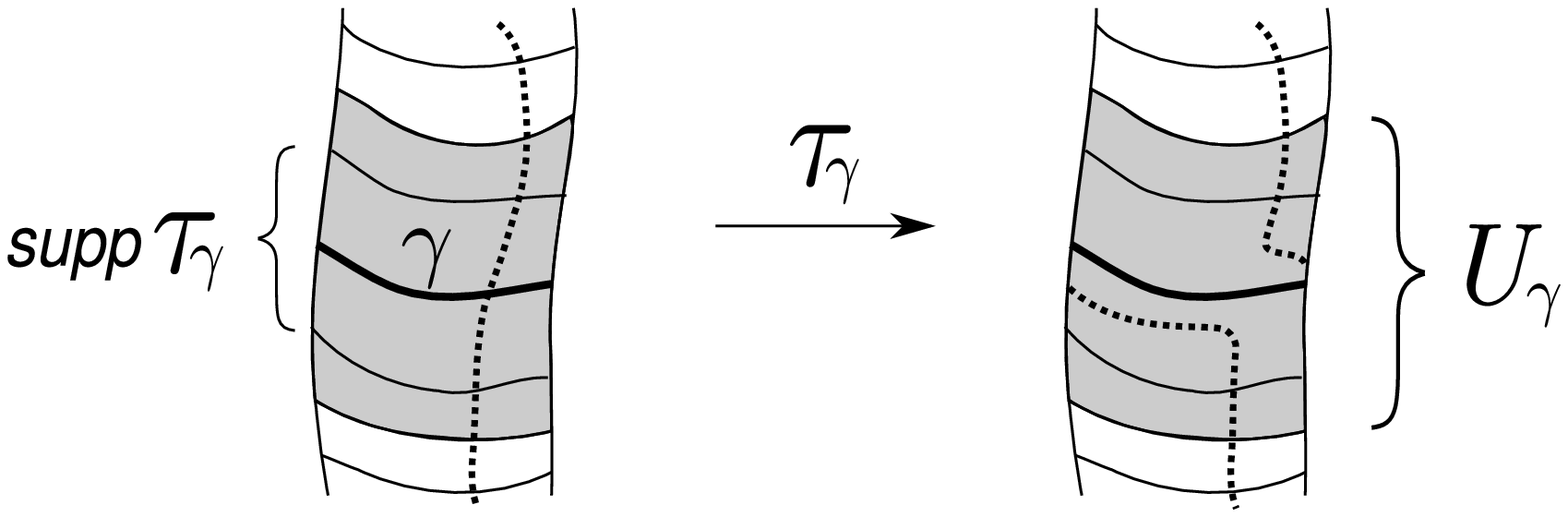}
\caption{}\protect\label{fig:dehn-twist}
\end{figure}
Notice that the image of $\comp{\gamma}$ in $\FReebf$ is a certain edge $e$.
We will say that $\gamma$ (as well as $\comp{\gamma}$) is \myemph{internal} (\myemph{external}) if so is $e$, and that $\tau_{\gamma}$ is a \myemph{twist around $e$}.

Now let $e_1,\ldots,e_{\intk}$ be all the internal edges of $\FReebf$.
For each $i$ choose any internal leaf $\gamma_{i}$ corresponding to $e_i$ and take any Dehn twist $\tau_i\in\Dpn$ along $\gamma_i$.
Put
$$
T=\mathop\cup\limits_{i=1}^{\intk}\supp\tau_i. 
$$
Let $\Jgrp=\langle \tau_1,\ldots,\tau_{\intk}\rangle \subset\Dpn$ be a subgroup generated by the internal Dehn twists.
Since $\supp\tau_i\cap\supp\tau_j=\varnothing$ for $i\not=j$, we see that $\Jgrp$ is a free abelian group with basis $\langle \tau_1,\ldots,\tau_{\intk}\rangle$, so $\Jgrp\approx \ZZZ^{\intk}$.

\begin{theorem}\label{th:pi0Dpn}{\rm c.f.\;\cite[Th.\;6.2]{Maks:AGAG:2006}.}
Suppose that $\func:\Mman\to\Psp$ satisfies axioms \AxBd\ and \AxSPN, and has at least one critical point of type $\ST$.
Then the inclusion $\zeta:\Jgrp\subset\Dpn$ is a homotopy equivalence.
In particular, we have an isomorphism
$$
\zeta_0:\Jgrp \equiv \pi_0\Jgrp \to\pi_0\Dpn,
$$
so $\pi_0\Dpn\approx\ZZZ^{\intk}$ is freely generated by the isotopy classes of internal Dehn twists.
\end{theorem}
The proof is similar to\;\cite[Th.\;6.2]{Maks:AGAG:2006}.
It suffices to establish the following statement:
\begin{lemma}
{\rm1)}~Every $\dif\in\Dpn$ is isotopic in $\Dpn$ to a product of internal Dehn twists, whence $\zeta:\Jgrp\to\pi_0\Dpn$ is surjective.

{\rm2)}~If $\dif=\tau_{1}^{m_1}\circ\cdots\circ\tau_{\intk}^{m_{\intk}}\in\StabIdf$ for some $m_i\in\ZZZ$, then $m_i=0$ for all $i$ and thus $\dif=\id_{\Mman}$, whence $\zeta:\Jgrp\to\pi_0\Dpn$ is a monomorphism.
\end{lemma}
\begin{proof}
First suppose that $\Mman$ is orientable.
Let $\AFld$ be a vector field on $\Mman$ satisfying assumptions of Lemma\;\ref{lm:Hamilt_vf_and_prop_orient}.

\begin{claim}\label{clm:shift_func_on_T}
For every $\dif\in\Dpn$ there exists a unique $\Cinf$ function $\sigma:\Mman\setminus T\to\RRR$ being shift function for $\dif$ with respect to $\AFld$. 
If $\dif$ is fixed on $\Mman\setminus T$, then $\sigma=0$ on $\Mman\setminus T$.
\end{claim}
\begin{proof}
a) Let $z$ be an $\NT$-point of $\func$.
Then $\dif\in\ker(\atnz)=\gimShAz$, so $\dif$ has a (unique) $\Cinf$ shift function $\sigma_{z}$ defined on some $\AFld$-invariant closed neighbourhood $\Vman_{z}$ of $z$ containing no other critical points of $\func$.
We can assume that $\Vman_{z}\cap\Vman_{z'}=\varnothing$ for $z\not=z'$.
Then the functions $\sigma_{z}$ define a unique shift function $\sigma_{\NT}$ for $\dif$ on the following set
$\Uman_{\NT}:=\mathop\cup\limits_{z\in\singfN}\Vman_{z}.$

b) Since $\func$ has at least one $\ST$-point, there exists a critical component $\crcomp$ of $\partf$ (that is a connected component of $\partfcr$) which is not a local extreme.
Let $z\in\crcomp\cap\singf$.
Then $z$ is an $\ST$-point of $\func$.
The assumption $\dif\in\Dpn$ means that $\dif\in\gimShAz$, so there exists a neighbourhood $\Vman_z$ of $z$ and a unique  $\Cinf$ function $\sigma_{z}:\Vman_z\to\RRR$ such that $\dif(x)=\AFlow(x,\sigma_z(x))$ for all $x\in\Vman_z$.
Notice that $\crcomp\setminus\singf$ is a disjoint union of intervals, these functions extend to a unique $\Cinf$ function $\sigma_{\crcomp}$ defined on some neighbourhood $\Vman(\crcomp)$ of $\crcomp$ such that $\dif(x) = \AFlow(x,\sigma_{\crcomp}(x))$ for all $x\in\Vman(\crcomp)$, see~\cite[Lm.~6.4]{Maks:AGAG:2006}.

We can assume that $\Vman(\crcomp)$ is $\AFld$-invariant, so $\partial\Vman(\crcomp)$ consists of regular leaves of $\partf$. 

Put $\Uman_{\ST}:=\cup_{\crcomp}\Vman(\crcomp)$, where $\crcomp$ runs over all critical components of $\partf$ that are not local extremes.
We can also assume that 
$$\Vman(\crcomp)\cap\Vman(\crcomp')=\Vman(\crcomp)\cap\Uman_{\NT}=\varnothing$$ for distinct critical components $\crcomp$ and $\crcomp'$.
Then the functions $\sigma_{\crcomp}$ define a unique shift function $\sigma_{\ST}$ for $\dif$ on $\Uman_{\ST}$.

c) Notice that the set $\overline{\Mman\setminus\Uman_{\ST}}$ is a union of cylinders and $2$-disks.
Moreover, every cylinder contains no critical point of $\func$, while every $2$-disk contains a unique critical point of $\func$ and this point is a local extreme of $\func$.
Let $B_1,\ldots,B_k$ be all the connected components of $\overline{\Mman\setminus\Uman_{\ST}}$ having one or the following properties: either
\begin{itemize}
 \item
$B_i$ is a cylinder such that $B_i\cap\partial\Mman\not=\varnothing$, or
 \item
$B_i$ is a $2$-disk containing a $\PT$-point of $\func$.
\end{itemize}

In particular, $B_i\cap\Uman_{\NT}=\varnothing$.
Denote $\Uman_{\partial,\PT}=\cup_{i=1}^{k} B_i$.
We claim that $\sigma_{\ST}$ extends to a $\Cinf$ function on $\Uman_{\ST}\cup\Uman_{\partial,\PT}$.

Indeed, suppose $B_i\approx S^1\times I$ is a cylinder.
Then $B_i\cap\Uman_{\ST}$ is a connected $\AFld$-invariant neighbourhood of one of the connected components of $\partial B_i$ and this neighbourhood does not contain another component of $\partial B_i$.
So we can assume that $B_i\cap\Uman_{\ST}=S^1\times[0,\eps]$.
Then the function $\sigma_{\ST}$ on $B_i\cap\Uman_{\ST}$ extends to a $\Cinf$ shift function for $\dif$ on all of $B_i$, see\;\cite[Lm.\;4.12(2)]{Maks:AGAG:2006}.
Denote this function by $\beta_i:B_i\to\RRR$.

Suppose $B_i$ is a $2$-disk containing a $\PT$-point $z$.
Then $Y_i:=B_i\cap\Uman_{\ST}$ is a neighbourhood of $\partial B_i$ and by (1) of Lemma\;\ref{lm:ext_shift_func_for_PN_points} $\sigma_{\ST}$ extends to a unique $\Cinf$ shift function $\beta_i:B_i\to\RRR$ for $\dif$ on all of $B_i$.

Denote $\Uman := \Uman_{\NT}\cup\Uman_{\ST}\cup\Uman_{\partial,\PT}$.
Then the functions 
$$
\sigma_{\NT}:\Uman_{\NT}\to\RRR,
\qquad 
\sigma_{\ST}:\Uman_{\ST}\to\RRR, 
\qquad 
\beta_i:B_i\to\RRR, \ (i=1,\ldots,\intk),
$$
define a unique $\Cinf$ shift function $\sigma:\Uman\to\RRR$ for $\dif$.
Notice that $\Mman\setminus\Uman$ is contained in the union of internal components of $\partfreg$.
Therefore we can assume that in fact $\Mman\setminus\Uman\subset T$, so $\Uman\supset\Mman\setminus T$.

If $\dif$ is fixed on $\Mman\setminus\Uman\subset T$, then it follows from uniqueness of $\sigma_{\NT}$, $\sigma_{\ST}$ and uniqueness of extensions $\beta_i$ of $\sigma_{\ST}$ to $\Uman_{\partial,\PT}$, that $\sigma\equiv0$.
\end{proof}

In order to complete statement 1) it remains to construct an isotopy of $\dif$ in $\Dpn$ to a diffeomorphism fixed on $\Mman\setminus T$.
Let $\mu:\Mman\to[0,1]$ be a $\Cinf$ function constant on leaves of $\partf$ and such that $\mu=1$ on some neighbourhood of $\Uman$.
Define the following map:
$$
H:\Mman\times I\to\Mman, \qquad 
H_t(x) = \AFlow(x, \, t \cdot \mu(x) \cdot \sigma_{\dif}(x)).
$$
Then $H$ is $\Cinf$, $H_0=\id_{\Mman}$, and $H_1=\dif$ on some neighbourhood of $\Uman$, see for details~\cite[Lemma~4.14]{Maks:AGAG:2006}.
In particular, every $H_t\in\Dpn$.
Hence the following isotopy $G_t=\dif \circ H_t^{-1}$ deforms $G_0=\dif$ in $\Dpn$ to a diffeomorphism $G_1$ fixed on some neighbourhood of $\Uman$.
In other words $\supp G_1 \subset T$, whence $G_1$ is isotopic in $\Dpn$ to a product of some internal Dehn twists.

\medskip 

2) Let $\dif=\tau_{1}^{m_1}\circ\cdots\circ\tau_{\intk}^{m_{\intk}}\in\StabIdf$.
We have to show that $m_i=0$ for all $i=1,\ldots,\intk$.
For each $i$ let $\Uman_i$ be a cylinder neighbourhood of $\gamma_i$ containing $\supp\tau_i$, see Figure\;\ref{fig:dehn-twist}.
Then $\dif|_{\Uman_i}=\tau_i^{m_i}|_{\Uman_i}$.

Since $\dif\in\StabIdf$ and $\func$ has at least one $\ST$-point, it follows from (iv) of Lemma\;\ref{lm:Hamilt_vf_and_prop_orient} that there exists a unique $\Cinf$ shift function $\afunc:\Mman\to\RRR$ for $\dif$, so $\dif(x)=\AFlow(x,\afunc(x))$ for all $x\in\Mman$.
Therefore an isotopy of $\dif$ to $\id_{\Mman}$ in $\Dpn$ can be given by $\dif_t(x)=\AFlow(x,t\afunc(x))$, $(t\in I)$.

On the other hand by Claim\;\ref{clm:shift_func_on_T} there exists a unique $\Cinf$ function $\sigma:\Mman\to\RRR$ being shift function for $\dif$ on $\Mman\setminus T$, whence $\afunc=\sigma$ on $\Mman\setminus T$.
Moreover, as $\dif$ is fixed on $\Mman\setminus T$, we have that $\afunc=\sigma=0$ on $\Mman\setminus T$.
Hence $\dif_t$ is fixed on $\Mman\setminus T$ for all $t\in I$.
This implies that $\dif|_{\Uman_i}=\tau_i^{m_i}|_{\Uman_i}$ is isotopic to $\id_{\Uman_i}$ relatively some neighbourhood of $\partial\Uman_i$.
But this is possible if and only if $m_i=0$.
This proves Theorem\;\ref{th:pi0Dpn} for orientable case.
\medskip 

Suppose $\Mman$ is non-orientable.
Let $p:\tMman\to\Mman$ be the oriented double covering of $\Mman$, and $\xi$ be $p$-equivariant involution of $\tMman$.
Then we have a function $\tfunc=\func\circ p:\tMman\to\Psp$.
Since every internal leaf $\gamma_i$ is two-sided, it follows that $p^{-1}(\gamma_i)$ consists of two connected components $\gamma_{i1}$ and $\gamma_{i2}$ being internal leafs of the foliation $\partittf$ of $\tfunc$.
Then there are Dehn twists $\tau_{i1},\tau_{i2}\in\tDpnt$ along $\gamma_{i1}$ and $\gamma_{i2}$ respectively such that $\tau_{i2}=\xi\circ\tau_{i1}\circ\xi$ and $\hat\tau_i:=\tau_{i2}\circ\tau_{i1}$ is a lifting of $\tau_{i}$.
Evidently, each $\tau_{ij}$ is internal, and by the oriented case the group $\pi_0\Dpnt$ is generated by $\tau_{ij}$ for $i=1,\ldots,\intk$ and $j=1,2$.

Consider the subgroup $\tDpnt$ of $\Dpnt$ consisting of symmetric $\dif$, i.e. $\dif\circ\xi=\xi\circ\dif$.
Then similarly to\;\cite[Claim\;6.5.1]{Maks:AGAG:2006} it can be shown that the isotopy classes of $\hat\tau_i$, $(i=1,\ldots,\intk)$ generate $\pi_0\tDpnt$, whence $\pi_0\tDpnt\approx \ZZZ^{\intk}$.
Moreover, it follows from (ix) of Lemma\;\ref{lm:Hamilt_vf_and_prop_nonorient} that there is a natural homeomorphism between $\Dpn$ and $\tDpnt$.
Hence $\pi_0\Dpn\approx \ZZZ^{\intk}$ and this group is generated by the isotopy classes of internal Dehn twists.
Theorem\;\ref{th:pi0Dpn} is completed.
\end{proof}

\subsection{Action of $\Jgrp$ on $H_1(\Mman,\singf)$.}
Let $\JgrpN$ and $\JgrpS$ be the subgroups of $\Jgrp$ generated by internal $\NT$-twists and $\ST$-twists respectively.
Evidently, $\Jgrp = \JgrpN\oplus\JgrpS$.
Notice that $\Jgrp$ naturally acts on the first relative homology group $H_1(\Mman,\singf)$ with integer coefficients. 
Let $\nu:\Jgrp\to\Aut(H_1(\Mman,\singf))$ be the corresponding homomorphism.
\begin{lemma}\label{lm:act_Jgrp_on_H_1}
If $\Mman$ is orientable, then $\ker(\nu)=\JgrpN$.
\end{lemma}
\begin{proof}
Evidently, every $\NT$-internal Dehn twist $\tau\in\JgrpN$ is isotopic to $\id_{\Mman}$ relatively to $\singf$, so $\ker(\nu)\supset\JgrpN$.

To prove the converse inclusion consider an intersection form $\langle\cdot,\cdot\rangle$ on $H_1(\Mman,\singf)$.
Evidently, the action of $\JgrpS$ on $H_1(\Mman,\singf)$ is given by
$$
\prod_{i=1}^{a} \tau_i^{m_i} \cdot x  = x + \sum_{i=1}^{a} m_i\langle x,\gamma_i\rangle \,\gamma_i,
$$
see also\;\cite[Eq.\;(6.1)]{Maks:AGAG:2006}.
Notice that $\ST$-internal curves represent linearly independent cycles in $H_1(\Mman,\singf)$.
This implies that $\nu|_{\JgrpS}$ is a monomorphism, so $\ker(\nu)=\JgrpN$.
\end{proof}

\section{Proof of Theorem\;\ref{th:Stabf_DiffIdMn_StabIdf}}
\label{sect:proof:th:Stabf_DiffIdMn_StabIdf}
We have to show that $\Stabf\cap\DiffIdMn = \StabIdf$.
By Lemma\;\ref{lm:Dpn_and_ketrstong_cap_DiffIdM}
$\Stabf\cap\DiffIdMn = \Dpn\cap\DiffIdMn$, moreover $\Didpn = \StabIdf$.
Hence it suffices to prove the following proposition:
\begin{proposition}\label{pr:intersections_with_ker_lambda}{\rm c.f. \cite[Pr.\;8.5]{Maks:AGAG:2006}}
$\Dpn\cap\DiffIdMn = \Didpn$.
\end{proposition}
\begin{proof}
Evidently, $\Dpn\cap\DiffIdMn \supset \Didpn$.

Conversely let $\dif\in\Dpn\cap\DiffIdMn$.
Then by Theorem\;\ref{th:pi0Dpn} $\dif$ is isotopic in $\Dpn$ to some $\dif'\in\Jgrp$, whence we can assume that $\dif\in\Jgrp$ itself.
Write $\dif=\dif_{\NT}\circ\dif_{\ST}$, where $\dif_{\NT}\in\Jgrp_{\NT}$ and $\dif_{\ST}\in\Jgrp_{\ST}$.
We have to show that $\dif_{\ST}=\dif_{\NT}=\id_{\Mman}$.

{\em Proof that $\dif_{\ST}=\id_{\Mman}$.}
If $\Mman$ is orientable, then by Lemma\;\ref{lm:act_Jgrp_on_H_1} $\dif$ trivially acts on $H_1(\Mman,\singf)$, i.e. $\dif\in\ker(\nu)=\dif_{\NT}$.
Hence $\JgrpS=\id_{\Mman}$.

Suppose $\Mman$ is non-orientable.
Then $\dif_{\ST}$ lifts to a diffeomorphism
$$\tdif_{\ST} \ \in \ \tDpnt\cap\tDiffIdtMn \ \subset \ \Dpnt\cap\DiffIdtMn$$
of the oriented double covering $\tMman$ of $\Mman$.
By the orientable case we have that $\tdif_{\ST}=\id_{\tMman}$, whence $\dif_{\ST}=\id_{\Mman}$ as well.

{\em Proof that $\dif_{\NT}=\id_{\Mman}$.}
Let $\tau_1,\ldots,\tau_{\intv}$ be the internal Dehn twists generating $\JgrpN$, so we can write 
$\dif_{\NT}=\tau_{1}^{m_1}\circ\tau_{\intv}^{m_{\intv}}$ for some $m_i\in\ZZZ$.
Then $\supp\tau_{i}$ is contained in some closed $2$-disk $B_i\subset\Mman$ which in turn contains a unique critical point $z_i$ of $\func$ being an $\NT$-point and such that $\supp\tau_{i} \cap (z_i\cup \partial B_i)=\varnothing$.
\begin{claim}\label{clm:deform_dif_to_id_rel_nbhsingfN}
There exists an isotopy $\dif_t:\Mman\to\Mman$ between $\dif_0=\id_{\Mman}$ and $\dif_1=\dif_{\NT}$ fixed on some neighbourhood of \ $\mathop\cup\limits_{i=1}^{\intv}  \ \{z_i\}$.
Hence $m_i=0$ for all $i$,and thus $\dif_{\NT}=\id_{\Mman}$.
\end{claim}
\begin{proof}
Let $\dif_t:\Mman\to\Mman$ be a $1$-isotopy between $\dif_0=\id_{\Mman}$ and $\dif_1=\dif_{\NT}$ in $\DiffIdMn$.
Then $\tn{z_i}(\dif_t)$ continuously depends on $t$ for each $i$.

\myemph{There exists another isotopy $\dif'_t$ between $\dif_0=\id_{\Mman}$ and $\dif_1=\dif_{\NT}$ in $\DiffIdMn$ such that  $\tn{z_i}(\dif'_t)=\id$ for all $t\in I$ and $i=1,\ldots,\intv$}.

Indeed, if $z_i$ is an $\NZT$-point, then the set of possible values for $\tn{z_i}(\dif_t)$ is finite, whence $\tn{z_i}(\dif_t)=\tn{z_i}(\dif_0)=\id$, so there is nothing to do.

Suppose $z_i$ is an $\NNT$-point.
Then in some local coordinates at $z_i$ in which 
$\nabla\AFld(z)=\left(\begin{smallmatrix}
 0 & \lambda \\ 0 & 0
\end{smallmatrix}\right)$
we have that 
$$
\tn{z_i}(\dif_t) = \left(\begin{smallmatrix}
 1 & \bar\lambda(t) \\ 0 & 1
\end{smallmatrix}\right),
$$
where $\bar\lambda:I\to\RRR$ is a \myemph{continuous} function with $\bar\lambda(0)=\bar\lambda(1)=0$.
Then there exists a $1$-isotopy $g_t:\Mman\to\Mman$ fixed outside some neighbourhood of $\singfN$ and such that $g_0=g_1=\id_{\Mman}$, $g_t(z_i)=z_i$, and $\tn{z_i}(g_t)=\tn{z_i}(\dif_t)$ for all $t\in I$ and $i=1,\ldots,\intv$.
Such an isotopy $(g_t)$ can be constructed by introducing a parameter in\;\cite[Lm.\;5.4]{Palais:TRAMS:1959}.

Evidently, the isotopy $\dif'_t=g_t^{-1}\circ\dif_t$ deforms $\id_{\Mman}$ to $\dif$ in $\DiffIdMn$ and satisfies $\tn{z_i}(\dif_t)=\id$.

Now similarly to the proof of\;\cite[Lm.\;5.2]{Palais:TRAMS:1959} we can change $\dif'_t$ to another isotopy $\dif''_t$ between $\id_{\Mman}$ and $\dif_{\NT}$ such that $\dif''_t$ is fixed on some neighbourhood $\Vman$ of $\mathop\cup\limits_{i=1}^{\intv}  \ \{z_i\}$ for each $t\in I$.
Then $(\dif''_t)$ satisfies the statement of Claim\;\ref{clm:deform_dif_to_id_rel_nbhsingfN}.
\end{proof}

Proposition\;\ref{pr:intersections_with_ker_lambda} and therefore Theorem\;\ref{th:Stabf_DiffIdMn_StabIdf} are completed.
\end{proof}

\section{Proof of Theorem\;\ref{th:splitting_pi1Orbf}} 
\label{sect:proof:th:splitting_pi1Orbf}
As noted above, we have to prove our theorem only for the case $\chi(\Mman)\geq0$.

Let $\{g_{\alpha}\}_{\alpha\in A}$ be any set of generators for $\JgrpId$.
For each $\alpha\in A$ let $\omega_{\alpha}:I \to \DiffIdM$ be a path such that $\omega_{\alpha}(0)=\id_{\Mman}$ and $\omega_{\alpha}(1)=g_{\alpha}$, see Figure\;\ref{fig:paths_in_diff}a).
Since $g_{\alpha}\in\JgrpId\subset\Stabf$, we see that the map 
$\nu:I\to\Orbf$ defined by $\nu(t)=\func\circ\omega_{\alpha}(t)$ is a loop, i.e. $\nu(0)=\nu(1)=\func$.

\begin{figure}[ht]
\begin{tabular}{ccc}
\includegraphics[height=2.5cm]{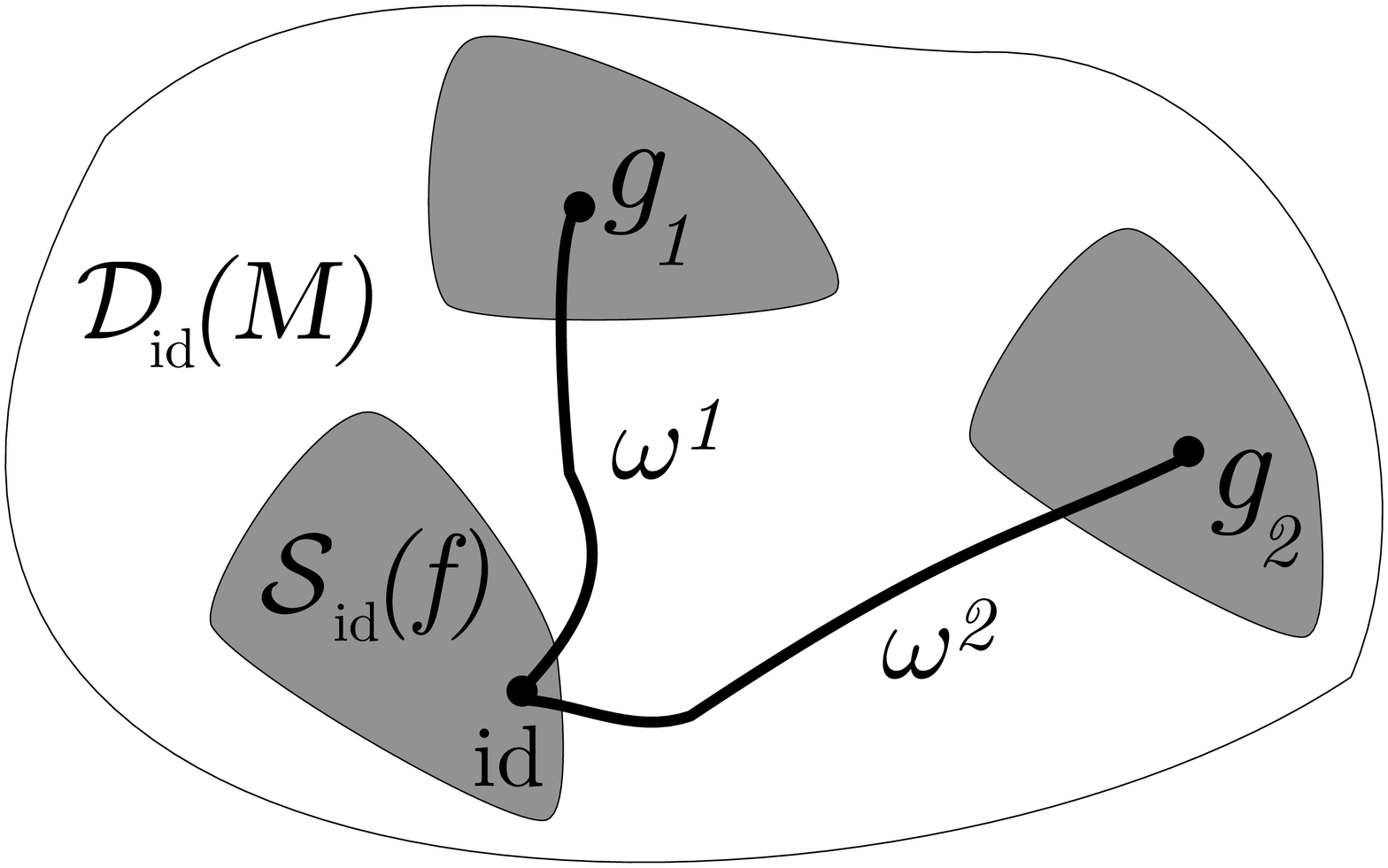}
&
\qquad \qquad 
&
\includegraphics[height=2.5cm]{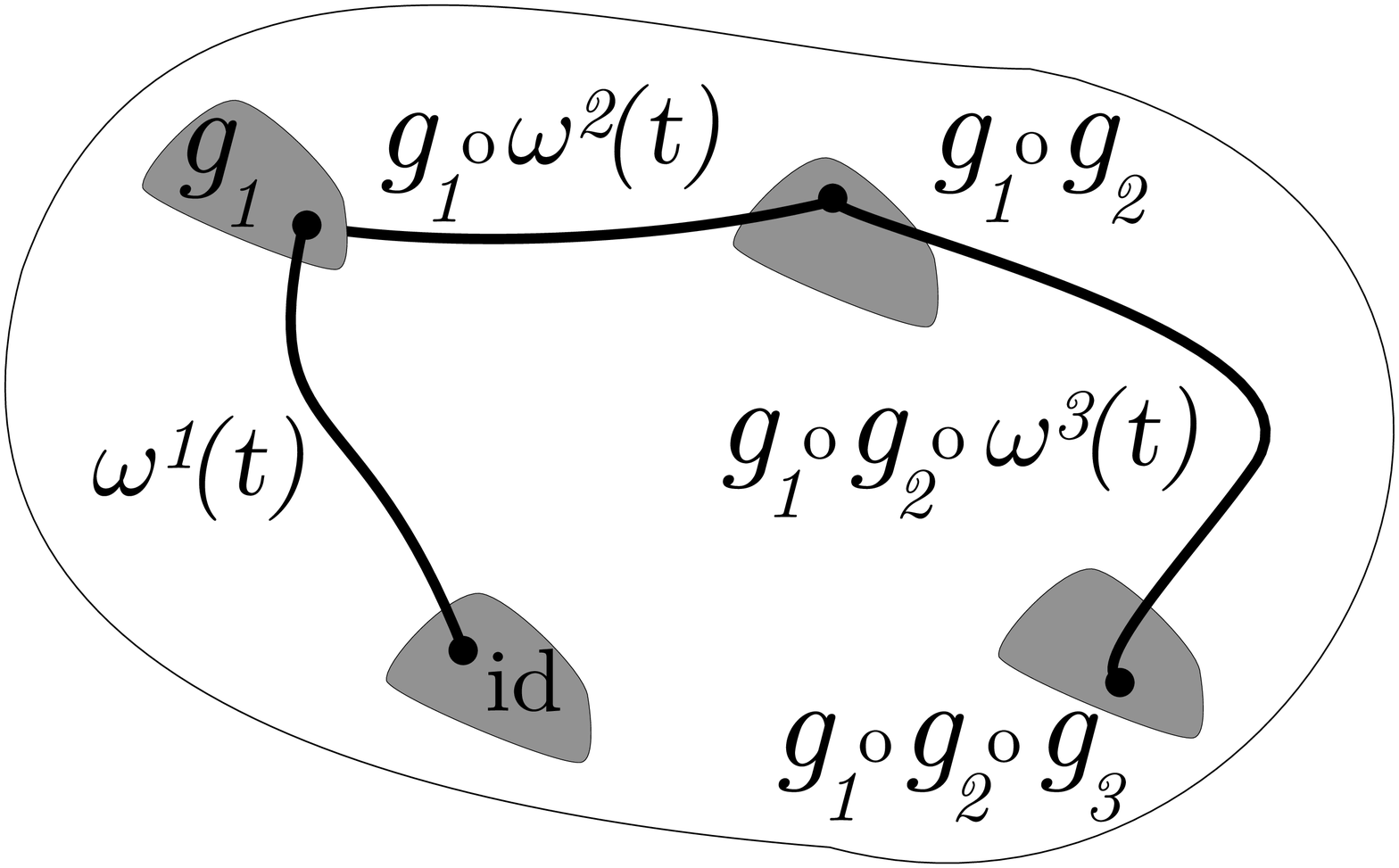} \\
a) Paths $\omega_i$
& & b) A path $\kappa_{g_1\circ g_2 \circ g_3}$
\end{tabular}
\caption{}\protect\label{fig:paths_in_diff}
\end{figure}

Let also $\fgrp$ be a free group generated by symbols $\{\hat g_{\alpha}\}_{\alpha\in A}$.
Then there exists a unique epimorphism $\eta:\fgrp\to\JgrpId$ such that $\eta(\hat g_{\alpha}) = g_{\alpha}$, and a unique homomorphism $\psi:\fgrp\to\pi_1\Orbf$ defined by $\psi(\hat g_{\alpha}) = [\func\circ \omega_{\alpha}]$ for all $\alpha\in A$.
Moreover we have the following commutative diagram:
$$
\xymatrix{
                  &  \fgrp\ar[d]^{\eta} \ar[dl]_{\psi} & \\ 
\hgrp \ar@/_/[r]_{\partial_1}  &  \JgrpId  \ar@{-->}@/_/[l]_{s}  & 
}
$$

Evidently, if $\ker\eta\subset\ker\psi$, then $\partial_1$ admits a section $s:\JgrpId\to\hgrp$, whence $\hgrp\approx\pi_1\DiffIdM\times \JgrpId$. 
Thus for the proof of Theorem\;\ref{th:splitting_pi1Orbf} we have to find conditions when $\ker\eta\subset\ker\psi$.

First we present exact formulas for $\psi$.
Let $\hat h=(\hat g_{1})^{\eps_{1}} \cdots (\hat g_{a})^{\eps_{a}} \in \fgrp$, where $\eps_{i}\in\ZZZ$.
Put $h = \eta(\hat h) = g_{1}^{\eps_{1}} \cdots g_{a}^{\eps_{a}}\in\JgrpId$ and define the following path $\kappa_{\hat\dif}:[0,a]\to\DiffIdM$ by
\begin{equation}\label{equ:path_kappa}
\kappa_{\hat\dif}(t) = 
\begin{cases}
 \omega_{1}(t)^{\eps_{1}}, & t\in[0,1], \\
 g_{1}^{\eps_{1}}\circ\omega_{2}(t-1)^{\eps_{2}}, & t\in[1,2], \\
 g_{1}^{\eps_{1}}\circ g_{2}^{\eps_{2}} \circ\omega_{3}(t-2)^{\eps_{3}}, & t\in[2,3], \\
 \cdots \cdots \cdots \cdots \\
 g_{1}^{\eps_{1}}\circ \cdots \circ g_{a-1}^{\eps_{a-1}} \circ\omega_{a}(t-a+1)^{\eps_{a}}, & t\in[a-1,a],
\end{cases}
\end{equation}
see Figure\;\ref{fig:paths_in_diff}b).
Evidently, $\kappa_{\hat\dif}(0)=\id_{\Mman}$ and $\kappa_{\hat\dif}(1)=\dif\in\JgrpId\subset\Stabf$.
Hence the map $\nu_{\hat\dif}:I\to\Orbf$ defined by $\nu_{\hat\dif}(t)=\func\circ\kappa_{\hat\dif}(t)$ is a loop.
Then it is easy to see that $\psi(\hat\dif)=[\nu_{\hat\dif}]$.

Suppose $\hat\dif\in\ker\eta$, so $\dif=\eta(\hat\dif)=\id_{\Mman}$.
Then $\kappa_{\hat\dif}$ is a loop in $\DiffIdM$.

\begin{lemma}\label{lm:suff_cond_triv_centr_ext}
Let $\{g_{\alpha}\}_{\alpha\in A}$ be a set of generators for $\JgrpId$, $\fgrp$ be a free group generated by symbols $\{\hat g_{\alpha}\}_{\alpha\in A}$, and $H=\{\hat\dif_{\beta}\}_{\beta\in B} \subset \fgrp$ be the subset whose normal closure coincides with $\ker\eta$, so 
$$\JgrpId=\langle \ \{ g_{\alpha}\}_{\alpha\in A} \ | \ \{\hat\dif_{\beta}\}_{\beta\in B} \ \rangle$$ is a presentation for $\JgrpId$.
For each $\alpha\in A$ take a path $\omega_{\alpha}:I \to \DiffIdM$ such that $\omega_{\alpha}(0)=\id_{\Mman}$ and $\omega_{\alpha}(1)=g_{\alpha}$.
Then each of the following conditions implies that $\ker\eta\subset\ker\psi$.

{\rm 1)}~For each $\beta\in B$ the loop $\kappa_{\dif_{\beta}}$ is null-homotopic in $\DiffIdM$.

{\rm 2)}~There exists a subset $\Qman\subset\Mman$ consisting of $k>\chi(\Mman)$ points and such that $\omega_{\alpha}(t)$ is fixed on $\Qman$ for all $\alpha\in A$ and $t\in I$.
\end{lemma}
\begin{proof}
Statement 1) is trivial.

2) Let $x_1,\ldots,x_k\in \Qman$ be $k$ distinct points, $\Diff(\Mman,k)$ be the group of diffeomorphisms of $\Mman$ that fix each of these points, and $\DiffId(\Mman,k)$ be the identity path component of $\Diff(\Mman,k)$.
Then for each $\hat\dif_{\beta}$ the loop $\kappa_{\hat\dif_{\beta}}$ is contained in $\DiffId(\Mman,k)$.
Since $\chi(\Mman)<k$, it follows that $\DiffId(\Mman,k)$ is contractible, whence $\kappa_{\hat\dif_{\beta}}$ is null-homotopic in $\DiffId(\Mman,k)\subset\DiffIdM$.
Therefore $\psi(\hat\dif_{\beta}) = [\func\circ\kappa_{\hat\dif_{\beta}}]$ is null-homotopic in $\hgrp$, so $\hat\dif_{\beta}\in\ker\psi$.
\end{proof}

Thus for the proof of Theorem\;\ref{th:splitting_pi1Orbf} it suffices to show that for every surface $\Mman$ with $\chi(\Mman)\geq0$ one of the conditions (1) or (@) of Lemma\;\ref{lm:suff_cond_triv_centr_ext} is satisfied.

\begin{lemma}
Suppose $\Mman$ is one of the surfaces $S^2$, $\prjplane$, $D^2$, $\MobiusBand$ or $S^1\times I$.
Then there exists even infinite subset $\Qman\subset\Mman$ such that \myemph{every} internal Dehn twist is fixed on $\Qman$ and isotopic to $\id_{\Mman}$ relatively $\Qman$.
In particular, so does every $\dif\in\JgrpId$, whence by {\rm2)} of Lemma\;\ref{lm:suff_cond_triv_centr_ext} $\ker\eta\subset\ker\psi$.
\end{lemma}
\begin{proof}
Suppose $\Mman$ is either $2$-disk $D^2$ or M\"obius band $\MobiusBand$, or a cylinder $S^1\times I$.
Put $\Qman=\partial\Mman$ if $\Mman$ is either $D^2$ or $\MobiusBand$, and $\Qman=S^1\times0$ if $\Mman=S^1\times I$.
Then every internal Dehn twist is fixed near $\Qman$ and is isotopic to $\id_{\Mman}$ relatively to $\Qman$, see e.g. \cite{Alexander:PNAS:1923}, \cite[Th.\;3.4 \& 5.2]{Epstein:AM:1966}, \cite{Smale:ProcAMS:1959}.

Let $\Mman$ be either $2$-sphere $S^2$ or a projective plane $\prjplane$.
In this case $\func$ always has a local extreme.
Denote this point by $z$ and let $\Qman\subset\Mman\setminus T$ be a closed neighbourhood of $z$ diffeomorphic to $2$-disk.
Then $\Mman\setminus\Qman$ is either a $2$-disk (if $\Mman=S^2$) or a M\"obius band (if $\Mman=\prjplane$).
Moreover, every internal Dehn twist $\tau$ is fixed on some neighbourhood of $\Qman$.
Hence $\tau$ is isotopic to the identity relatively $\Qman$.
\end{proof}

\begin{lemma}
Let $\Mman$ be either a $2$-torus $T^2$ or a Klein bottle $\Kleinb$.
Suppose that every internal leaf $\gamma_i$, $(i=1,\ldots,\intk)$, separates $\Mman$.
Then there exists a subset $\Qman\subset\Mman$ satisfying {\rm2)} of Lemma\;\ref{lm:suff_cond_triv_centr_ext}.
\end{lemma}
\begin{proof}
Let $\Uman_i$ be an open neighbourhood of $\gamma_i$ diffeomorphic to $S^1\times (0,1)$ such that $\partial\Uman_i$ consists of two regular leafs of $\partf$ and $\supp\tau_i\subset\Uman_i$.
Since $\gamma_i$ separates $\Mman$, it follows that $\Mman\setminus\Uman_i$ consists of two connected components $B_i$ and $C_i$.
Moreover, as $\Mman$ is either a $2$-torus $T^2$ or a Klein bottle $\Kleinb$, we have that one of the components, say $B_1$, is either a $2$-disk or a M\"obius band.

a) Suppose $B_1$ is a M\"obius band.
Then $\Mman$ is a Klein bottle, and $C_1$ as well as $C_1\cup\Uman_1$ are M\"obius bands.
Put $\Qman=\partial\Uman_1\cap B_1$.
Then $\Qman$ is a simple closed curve ``parallel'' to $\gamma_1$ and also separating $\Mman$ into two M\"obius bands.
Moreover, every internal Dehn twist $\tau_i$ is fixed on some neighbourhood $\Qman$, and therefore it is isotopic to $\id_{\Mman}$ relatively $\Qman$.

b) Suppose that neither of $B_i$ or $C_i$ is a M\"obius band.
Then $B_i$ is a $2$-disk.
Renumbering $\gamma_i$ (if necessary) we can assume that there exists $r\in\{1,\ldots,\intk\}$ such that 
\begin{itemize}
 \item
if $i=1,\ldots,r$, then $B_i$ is not contained in any of $B_j$ for $j=1,\ldots,\intk$, and $j\not=i$, so $B_i$ is ``maximal'';
 \item
for $j=r+1,\ldots,\intk$ every $B_j$ is contained in some $B_i$ for $i=1,\ldots,r$.
\end{itemize}
Put 
\begin{equation}\label{equ:Q_M_setminus_BUi_1r}
\Qman \ := \ \mathop\cap\limits_{i=1}^{r} C_i \ = \ \mathop\cap\limits_{i=1}^{r} \Mman\setminus(B_i\cup \Uman_i) \ =  \ \Mman\setminus\mathop\cap\limits_{i=1}^{r}(B_i\cup \Uman_i).
\end{equation}
Then $\Qman$ is connected as a complement to a disjoint union of closed $2$-disks on a connected surface.
It is also easy to see that $\Qman$ does not contain any internal curve.
Indeed, suppose $\gamma_j\subset \Qman$.
Since $\gamma_j\subset\Uman_j$, we obtain from\;\eqref{equ:Q_M_setminus_BUi_1r} that $j\in\{r+1,\ldots,\intk\}$, whence $B_j \subset B_i$ for some $i=1,\ldots,r$.
On the other hand $\gamma_j$ separates $\Qman$ into two \myemph{non-empty} components $D_j$ and $D'_j$ such that one of them, say $D_j$, is contained in $B_j$.
Hence $D_j \subset B_j \subset B_i \subset \Mman\setminus\Qman$ which contradicts to the assumption that $D_j\subset \Qman$.

Thus $\Qman$ contains no internal curves and therefore every $\tau_i$ is fixed on some neighbourhood of $\Qman$.
As $\Mman\setminus \Qman$ is a union of $2$-disks, we see that $g$ is isotopic to the identity relatively $\Qman$.
Hence so does every $g\in\JgrpId$.
\end{proof}

\begin{lemma}
Let $\Mman$ be either a $2$-torus $T^2$ or a Klein bottle $\Kleinb$.
Suppose that $\gamma_1$ does not separate $\Mman$.
\begin{enumerate}
 \item[(i)]
If $\Mman$ is a $2$-torus $T^2$, then there exists a subset $\Qman$ satisfying {\rm2)} of Lemma\;\ref{lm:suff_cond_triv_centr_ext}, whence $\ker\eta\subset\ker\psi$.
 \item[(ii)]
If $\Mman$ is a Klein bottle $\Kleinb$, then $\ker\eta\subset\ker\psi$ as well.
\end{enumerate}
\end{lemma}
\begin{proof}
Since all $\gamma_i$ are mutually disjoint simple closed curves on a $2$-torus or a Klein bottle, we can assume that for some $a=1,\ldots,\intk$ the curves $\gamma_1,\ldots,\gamma_{a}$ are non-separating and isotopic each other, while each of $\gamma_{a+1},\ldots,\gamma_{\intk}$ separate $\Mman$ so that one of the components of $\Mman\setminus\gamma_i$ is a $2$-disk.
It follows that $\tau_i$ is isotopic to $\tau_1$ for $i=1,\ldots,a$, and to $\id_{\Mman}$ for $i=a+1,\ldots,\intk$.

Let $\Qman$ be a regular leaf of $\partf$ contained in $\Uman_1\setminus\supp\tau_1$, see Figure\;\ref{fig:twist-klein}.
Then every $\tau_{j}$ is fixed on some neighbourhood of $\Qman$ in $\Uman_1\setminus\supp\tau_1$.

Now let $\dif\in\JgrpId$.
Thus 
$\dif=\tau_{1}^{i_1}\circ\cdots\circ\tau_{a}^{i_a}\circ\tau_{a+1}^{i_{a+1}}\circ\cdots\circ\tau_{\intk}^{i_{\intk}}$ for some $i_{j}\in\ZZZ$ and $\dif$ is isotopic to $\id_{\Mman}$.
Put $d=i_1+\cdots+i_a$.
Since $\Mman\setminus \Qman$ is a cylinder, we see that $\dif$ is isotopic to $\tau_{1}^{d}$ relatively to $\Qman$.

\medskip 

(i)~Suppose $\Mman$ is a $2$-torus $T^2$.
As $\dif$ is isotopic to $\id_{T^2}$, we have that $d=0$, whence $\tau_1^{d}$ (and thus $\dif$ itself) is isotopic to $\id_{T^2}$ relatively to $\Qman$.

\medskip 

(ii)~Let $\Mman$ be a Klein bottle $\Kleinb$.
It is well known that $\tau_1^2$ is isotopic to $\id_{\Kleinb}$, see\;\cite[Lm.\;5]{Lickorish:PCPS:1963}.
Moreover, we can assume $\Qman$ is invariant (but not fixed!) under such an isotopy, see Figure\;\ref{fig:twist-klein}.
Since $\dif$ is isotopic to $\id_{\Kleinb}$, we obtain that $d$ is even and $\dif$ is also isotopic to $\id_{\Kleinb}$ via an isotopy which leaves $\Qman$ invariant.

\begin{figure}[ht]
\centerline{\includegraphics[width=0.8\textwidth]{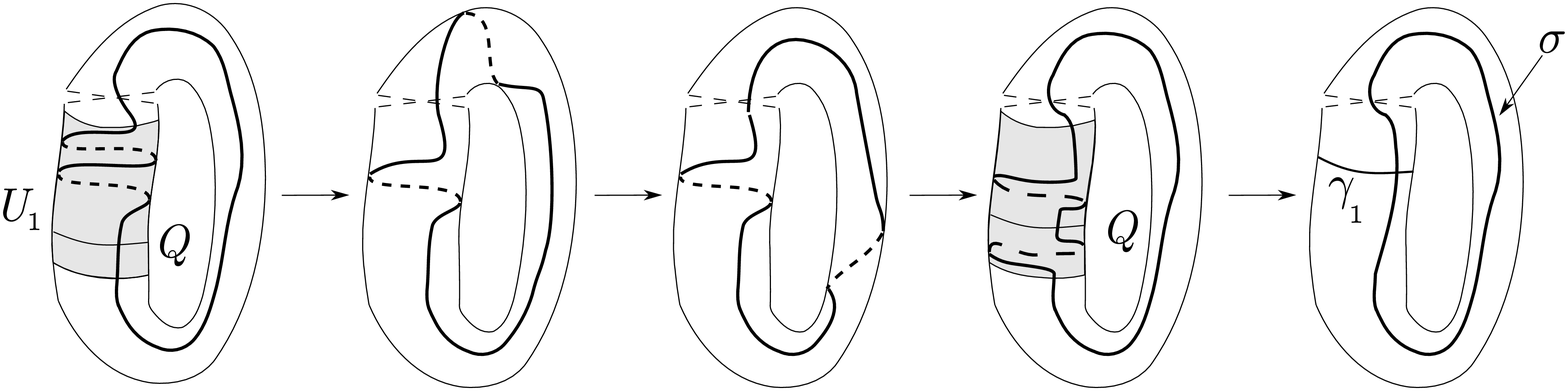}}
\caption{Isotopy of $\tau_1^2$ to $\id_{\Kleinb}$}\protect\label{fig:twist-klein}
\end{figure}

Let $\{g_{\alpha}\}_{\alpha\in A}$ be a set of generators for $\JgrpId$, $\fgrp$ be a free group generated by symbols $\{\hat g_{\alpha}\}_{\alpha\in A}$, and $H=\{\hat\dif_{\beta}\}_{\beta\in B} \subset \fgrp$ be the subset whose normal closure coincides with $\ker\eta$.
By the previous arguments for each $g_{\alpha}$ there exists a path $\omega_{\alpha}:I\to\DiffId(\Kleinb)$ between $\id_{\Kleinb}$ and $g_{\alpha}$ such that $\omega_{\alpha}(t)(\Qman)=\Qman$ for all $\alpha\in A$ and $t\in I$.

Let $\beta\in B$ and $\kappa_{\hat\dif_{\beta}}:I\to\DiffId(\Kleinb)$ be the corresponding loop in $\DiffId(\Kleinb)$.
Then $\kappa_{\hat\dif_{\beta}}(t)(\Qman)=\Qman$ for all $t\in I$ as well.
Now it is easy to see that $\kappa_{\hat\dif_{\beta}}$ is null-homotopic, so the assumption (2) of Lemma\;\ref{lm:suff_cond_triv_centr_ext} holds.

Indeed, let $\sigma\subset\Kleinb$ be a simple closed curve shown in Figure\;\ref{fig:twist-klein}.
Then $\Kleinb\setminus\sigma$ consists of two M\"obius bands.
It is well-known that $\pi_1\DiffId(\Kleinb)=\ZZZ$ and this group is generated by the isotopy which \myemph{rotates $\Kleinb$ twice along $\sigma$}, see e.g. \cite{Gramain:ASENS:1973}.
In particular, if $\kappa:I\to\DiffId(\Kleinb)$ is a loop being not null-homotopic, then $\kappa(t)(\Qman)\not=\Qman$ for some $t\in I$.
Therefore $\kappa_{\hat\dif_{\beta}}$ is null-homotopic for all $\hat\dif_{\beta}\in\ker\eta$.
\end{proof}

\bibliographystyle{amsplain}
\bibliography{a}

\end{document}

%% file: macro.tex
\makeatletter
\newcommand\testshape{family=\f@family; series=\f@series; shape=\f@shape.}
\def\myemphInternal#1{\if n\f@shape%
\begingroup\itshape #1\endgroup\/%
\else\begingroup\bfseries #1\endgroup%
\fi}
\def\myemph{\futurelet\testchar\MaybeOptArgmyemph}
\def\MaybeOptArgmyemph{\ifx[\testchar \let\next\OptArgmyemph
                 \else \let\next\NoOptArgmyemph \fi \next}
\def\OptArgmyemph[#1]#2{\index{#1}\myemphInternal{#2}}
\def\NoOptArgmyemph#1{\myemphInternal{#1}}
\makeatother

\newtheorem{theorem}[subsection]{Theorem}
\newtheorem{lemma}[subsection]{Lemma}
\newtheorem{proposition}[subsection]{Proposition}
\newtheorem{corollary}[subsection]{Corollary}
\newtheorem{claim}[subsection]{Claim}

\newtheorem{remark}[subsection]{Remark}
\newtheorem{example}[subsection]{Example}

\newtheorem{definition}[subsection]{Definition}

\newenvironment{axiom}[1]
{
\smallskip
{\bf Axiom #1.}\begin{it}
}
{
\end{it}\smallskip
 }

\makeatletter
\@addtoreset{equation}{section}
\@addtoreset{figure}{section}
\@addtoreset{table}{section}
\makeatother

\providecommand\eqref[1]{(\ref{#1})}

\newcommand\CCC{{\mathbb C}}
\newcommand\DDD{{\mathbb D}}
\newcommand\NNN{{\mathbb N}}

\newcommand\RRR{{\mathbb R}}

\newcommand\ZZZ{{\mathbb Z}}

\newcommand\PathComp[2]{ {#1}_{#2} }

\newcommand\id{\mathrm{id}}

\newcommand\Int{\mathrm{Int}}
\newcommand\Per{\mathrm{Per}}
\newcommand\supp{\mathrm{supp\,}}

\newcommand\Orbit{\mathcal{O}}
\newcommand\Stab{\mathcal{S}}
\newcommand\Diff{\mathcal{D}}
\newcommand\Aut{\mathrm{Aut}}

\newcommand\tDiff{\widetilde{\Diff}}
\newcommand\tDiffId{\widetilde{\Diff}_{\id}}

\newcommand\Cinf{\mathcal{C}^{\infty}}

\newcommand\Unity[1]{ \PathComp{#1}{\id} }

\newcommand\StabId{\Unity{\Stab}}

\newcommand\fgrp{\mathcal{F}}
\newcommand\grp{\mathcal{G}}
\newcommand\hgrp{\mathcal{H}}

\newcommand\Agrp{\mathcal{A}}

\newcommand\eps{\varepsilon}

\newcommand\aCircle{S^1}
\newcommand\Psp{P}
\newcommand\Mman{M}
\newcommand\tMman{\widetilde{\Mman}} 
\newcommand\AFlow{\mathbf{\AFld}}

\newcommand\partit{\Delta}

\newcommand\func{f}

\newcommand\DiffM{\Diff(\Mman)}

\newcommand\dif{h}

\newcommand\singf{\Sigma_{\func}}
\newcommand\singtf{\Sigma_{\tfunc}}

\newcommand\smone{\Ci{\Mman}{\Psp}}

\newcommand\smr{\Ci{\Mman}{\RRR}}

\newcommand\Mcr{(\Mman,\singf)}

\newcommand\Orbf{\Orbit(\func)}
\newcommand\Orbff{\Orbit_{\func}(\func)}

\newcommand\DiffMcr{\Diff\Mcr}
\newcommand\DiffIdMcr{\Diff_{\id}\Mcr}
\newcommand\DiffIdM{\Diff_{\id}(\Mman)}

\newcommand\StabIdfr[1]{\StabId(\func)^{#1}}

\newcommand\Stabf{\Stab(\func)}
\newcommand\StabIdf{\Stab_{\id}(\func)}

\newcommand\tdif{\tilde{\dif}}

\newcommand\Kleinb{\mathbb{K}}

\newcommand\Dpartf{\Diff(\partf)}

\newcommand\Dpartfpl{\Diff^{+}(\partf)}

\newcommand\DpartfIdr[1]{\Diff_{\id}(\partf)^{#1}}

\newcommand\Stabtf{\Stab(\tfunc)}
\newcommand\tDifftM{\widetilde{\Diff}(\tMman)}
\newcommand\tDiffptM{\widetilde{\Diff}^{+}(\tMman)}

\newcommand\DiffId{\Diff_{\id}}

\newcommand\Reebf{\Gamma(\func)}

\newcommand\AutfR{\Aut(\Reebf)}

\newcommand\MobiusBand{M\text{\"o}}
\newcommand\prjplane{\RRR\mathrm{P}^2}

\newcommand\rankpiO{k}

\newcommand\crcomp{K}

\newcommand\Uman{U}
\newcommand\Vman{V}
\newcommand\Wman{W}

\newcommand\KR{KR}

\newcommand\Qman{Q}

\newcommand\KRfunc{\func_{\Gamma}}
\newcommand\prKR{p_{\func}}

\newcommand\AFld{F}

\newcommand\BFld{G}
\newcommand\Shift{\varphi}
\newcommand\partf{\Delta_{\func}}

\newcommand\afunc{\alpha}
\newcommand\bfunc{\beta}

\newcommand\Stop{\mathsf{S}}
\newcommand\Wr[1]{\mathcal{C}^{#1}}
\newcommand\Sr[1]{\Stop^{#1}}

\newcommand\EAFlow{\mathcal{E}(\AFld)}

\newcommand\EidAFlow[1]{\mathcal{E}_{\id}(\AFld)^{#1}}

\newcommand\EidAV[1]{\mathcal{E}_{\id}(\AFld,\Vman)^{#1}}
\newcommand\EAV{\mathcal{E}(\AFld,\Vman)}

\newcommand\DAV{\mathcal{D}(\AFld,\Vman)}
\newcommand\DAFlow{\mathcal{D}(\AFld)}

\newcommand\DidAV[1]{\mathcal{D}_{\id}(\AFld,\Vman)^{#1}}
\newcommand\DidAFlow[1]{\mathcal{D}_{\id}(\AFld)^{#1}}

\newcommand\Cont[1]{\mathcal{C}^{#1}}
\newcommand\Cr[3]{\Cont{#1}(#2,#3)}
\newcommand\Ci[2]{\Cr{\infty}{#1}{#2}}

\newcommand\orig{0}

\newcommand\ShA{\varphi}
\newcommand\imShA{Sh(\AFld)}
\newcommand\ShAV{\varphi_{\Vman}}
\newcommand\ShAW{\varphi_{\Wman}}
\newcommand\imShAV{Sh(\AFld,\Vman)}
\newcommand\imShAW{Sh(\AFld,\Wman)}

\newcommand\funcAV{\mathsf{func}(\AFld,\Vman)}

\newcommand\tfunc{\widetilde{\func}}
\newcommand\partittf{\partit_{\tfunc}}

\newcommand\AxBd{{\rm(A1)}}
\newcommand\AxSPN{{\rm(A2)}}
\newcommand\AxFld[1]{{\rm(AFLD)}}
\newcommand\AxFibr{{\rm(A3)}}

\newcommand\orb{o}
\newcommand\ADom{\mathsf{dom}(\AFld)}

\newcommand\VmanS{\widehat{\Vman}}
\newcommand\WmanS{\widehat{\Wman}}

\newcommand\gdif{g}

\newcommand\Dm{D}

\newcommand\smrz{\mathcal{A}}

\newcommand\partfreg{\partf^{\mathrm{reg}}}
\newcommand\partfcr{\partf^{\mathrm{cr}}}

\newcommand\zvfrm[2]{\mathsf{F}_{#1}(#2)}
\newcommand\zfrm[1]{\mathsf{F}_{#1}}

\newcommand\singfN{\Sigma_{\func}^{\NT}}
\newcommand\singfS{\Sigma_{\func}^{\ST}}
\newcommand\singfP{\Sigma_{\func}^{\PT}}

\newcommand\FReebf{\widehat{\Gamma}(\func)}
\newcommand\AutfFR{\Aut(\FReebf)}

\newcommand\dd[1]{\tfrac{\partial}{\partial #1}}
\newcommand\ddd[2]{\tfrac{\partial #1}{\partial #2}}
\newcommand\flatsign[1]{\bar{#1}}
\newcommand\dAx{\flatsign{X}}
\newcommand\dAy{\flatsign{Y}}
\newcommand\FixA{\Sigma_{\AFld}}

\newcommand\gD[2]{\widehat{\Diff}(#1,#2)}

\newcommand\gDAz{\gD{\AFld}{z}}
\newcommand\gDpartfz{\gD{\partf}{z}}

\newcommand\gDplus[2]{\widehat{\Diff}^{+}(#1,#2)}
\newcommand\gDpAz{\gDplus{\AFld}{z}}

\newcommand\tn[1]{J_{#1}}
\newcommand\tnz{J_{z}}
\newcommand\atnz{\mathsf{sh}_{z}}
\newcommand\GLR[1]{\mathrm{GL}(#1,\RRR)}
\newcommand\GLRp[1]{\mathrm{GL}^{+}(#1,\RRR)}
\newcommand\ST{\mathsf{S}}
\newcommand\PT{\mathsf{P}}
\newcommand\NT{\mathsf{N}}
\newcommand\NNT{\mathsf{NN}}
\newcommand\NZT{\mathsf{NZ}}

\newcommand\gimShAz{\widehat{Sh}(\AFld,z)}
\newcommand\gimShAzz{\widehat{Sh}(\AFld_z,z)}
\newcommand\gShAz{\widehat{\varphi}_{z}}
\newcommand\gCiMRz{\Cinf_{z}(\Mman)}

\newcommand\SPECA{{\rm(SP1)}}
\newcommand\SPECB{{\rm(SP2)}}

\newcommand\PTA{{\rm(P1)}}
\newcommand\PTB{{\rm(P2)}}
\newcommand\PTC{{\rm(P3)}}
\newcommand\PTD{{\rm(P4)}}
\newcommand\PTE{{\rm(P5)}}
\newcommand\PTF{{\rm(P6)}}

\newcommand\NA{{\rm(i)}}
\newcommand\NB{{\rm(ii)}}
\newcommand\NC{{\rm(iii)}}

\newcommand\NTA{{\rm(N1)}}
\newcommand\NTB{{\rm(N2)}}
\newcommand\NTC{{\rm(N3)}}

\newcommand\Dpn{\Diff^{\NT}(\partf)}
\newcommand\Didpn{\DiffId^{\NT}(\partf)}
\newcommand\Didpnr[1]{\DiffId^{\NT}(\partf)^{#1}}

\newcommand\DiffMcrid{\Diff(\Mman,\singf,\id)}

\newcommand\tDpnt{\widetilde{\Diff}^{\NT}(\partittf)}
\newcommand\Dpnt{\Diff^{\NT}(\partittf)}

\newcommand\DiffMn{\Diff^{\NT}(\Mman,\singf)}
\newcommand\DiffIdMn{\DiffId^{\NT}(\Mman,\singf)}

\newcommand\tDiffIdtMn{\widetilde{\Diff}_{\id}^{\NT}(\tMman,\singtf)}

\newcommand\DiffIdtMn{\Diff_{\id}^{\NT}(\tMman,\singtf)}

\newcommand\Ggrpz{\mathcal{G}_{z}}

\newcommand\tew[1]{\mathsf{T}_{#1}}
\newcommand\stongr{\mu}
\newcommand\stonkr{\lambda}
\newcommand\comp[1]{K_{#1}}

\newcommand\Orbfn{\Orbit^{\NT}(\func)}
\newcommand\Orbffn{\Orbit^{\NT}_{\func}(\func)}
\newcommand\Stabfn{\Stab^{\NT}(\func)}
\newcommand\StabIdfn{\StabId^{\NT}(\func)}
\newcommand\Jgrp{\mathcal{J}}
\newcommand\JgrpId{\Jgrp_{0}}
\newcommand\JgrpN{\Jgrp_{\NT}}
\newcommand\JgrpS{\Jgrp_{\ST}}

\newcommand\cond[1]{$(\mathsf{#1})$}

\newcommand\condShAVop{\cond{A}}
\newcommand\condBsShAVop{\cond{A1}}

\newcommand\condShAUVop{\cond{B}}
 \newcommand\condBsShAUVop{\cond{B1}}

\newcommand\condImShUopImSh{\cond{C}}
\newcommand\condBsimShimShUop{\cond{C1}}

\newcommand\condznpnrec{\cond{R1}}
\newcommand\condzpid{\cond{R2}}
\newcommand\condzUinv{\cond{S1}}
\newcommand\condzUpbi{\cond{S2}}

\newcommand\condShAop{\cond{Z}}

\newcommand\imShAUV{Sh(\AFld|_{\Uman},\Vman)}
\newcommand\funcAUV{\mathsf{func}(\AFld|_{\Uman},\Vman)}
\newcommand\ShAUV{\varphi_{\Uman,\Vman}}
\newcommand\AFlowU{\AFlow_{\Uman}}
\newcommand\domAU{\mathsf{dom}(\AFlowU)}

\newcommand\intk{{\mathbf k}}
\newcommand\intl{{\mathbf l}}
\newcommand\intv{{\mathbf v}}